\documentclass[11pt,reqno]{amsart}
\usepackage[foot]{amsaddr}
\usepackage{latexsym}
\usepackage{amsmath}
\usepackage{amssymb}
\usepackage{amsthm}
\usepackage{epsfig}
\usepackage{xcolor}
\usepackage{graphicx}
\usepackage{bm}
\usepackage{enumitem}
\usepackage{mathtools}
\usepackage{tikz}
\usetikzlibrary{arrows.meta,calc,decorations.pathreplacing}
\usepackage{float}
\usepackage{placeins}
\usepackage{mathrsfs}
\usepackage[toc,page]{appendix}
\usepackage{bbm}
\usepackage{ytableau}
\usepackage{tkz-berge}
\usepackage{booktabs}
\usepackage{microtype}
\usepackage[linesnumbered, ruled, vlined]{algorithm2e}
\usepackage[top=1in, bottom=1in, left=1in, right=1in]{geometry}
\usepackage[colorlinks=true]{hyperref}

\usepackage{array}
\usepackage{multirow}

\hypersetup{
    colorlinks,
    linkcolor={blue!80!black},
    citecolor={blue!80!black},
    urlcolor={blue!80!black},
}
\colorlet{linkequation}{blue}

\definecolor{proofblue}{HTML}{245B93}
\definecolor{proofteal}{HTML}{16857A}
\definecolor{prooforange}{HTML}{C4661F}
\definecolor{proofgreen}{HTML}{3A7D44}
\definecolor{proofink}{HTML}{243447}
\definecolor{proofbg}{HTML}{F6F8FB}

\def\R{\mathbb{R}}

\def\a{\mathbf{a}}
\def\g{\mathbf{g}}
\def\e{\mathbf{e}}

\def\u{\mathbf{u}}
\def\bv{\mathbf{v}}
\def\x{\mathbf{x}}
\def\y{\mathbf{y}}
\def\z{\mathbf{z}}
\def\w{\mathbf{w}}
\def\m{\mathbf{m}}
\def\n{\mathbf{n}}

\def\P{\mathbb{P}}
\def\C{\mathbb{C}}
\def\E{\mathbb{E}}
\def\N{\mathbb{N}}
\def\1{\mathbf{1}}
\def\cA{\mathcal{A}}
\def\cU{\mathcal{U}}
\def\cX{\mathcal{X}}
\def\cY{\mathcal{Y}}
\def\cZ{\mathcal{Z}}
\def\cB{\mathcal{B}}
\def\cE{\mathcal{E}}

\def\cQ{\mathcal{Q}}
\def\cP{\mathcal{P}}
\def\cC{\mathcal{C}}

\def\cV{\mathcal{V}}
\def\cL{\mathcal{L}}

\def\cR{\mathcal{R}}

\def\cW{\mathcal{W}}
\def\cD{\mathcal{D}}
\def\cS{\mathcal{S}}
\def\bD{\mathbf{D}}

\def\de{\operatorname{d}}
\def\Im{\operatorname{Im}}
\def\Re{\operatorname{Re}}
\def\op{\mathrm{op}}

\def\eps{\varepsilon}
\def\Oprec{O_{\prec}}

\def\Tr{\operatorname{Tr}}

\def\diag{\operatorname{diag}}

\def\supp{\operatorname{supp}}

\def\spec{\operatorname{spec}}
\def\bw{\boldsymbol{w}}
\def\by{\boldsymbol{y}}

\def\<{\langle} 
\def\>{\rangle}

\newtheorem{theorem}{Theorem}[section]
\newtheorem{lemma}[theorem]{Lemma}
\newtheorem{coro}[theorem]{Corollary}

\newtheorem{proposition}[theorem]{Proposition}

\newtheorem{assumption}{Assumption}

\theoremstyle{definition}
\newtheorem{defi}[theorem]{Definition}
\newtheorem{example}[theorem]{Example}
\newtheorem{remark}[theorem]{Remark}

\numberwithin{equation}{section}
\numberwithin{figure}{section}

\title[Anisotropic local law under quadratic-form concentration]{The anisotropic local law for sample covariance matrices\\
under quadratic-form concentration}

\author[Ma]{Renyuan Ma}
\email{jack.ma.rm2545@yale.edu}

\author[Misiakiewicz]{Theodor Misiakiewicz$^{\dagger}$}
\email{theodor.misiakiewicz@yale.edu}

\address{$^{\dagger}$Department of Statistics and Data Science, Yale University}

\date{}

\allowdisplaybreaks

\begin{document}

\begin{abstract}
           We study sample covariance matrices $K = \frac{1}{N} \sum_{i=1}^N \x_i \x_i^* \in \R^{n \times n}$ in the proportional regime $n \asymp N$. The columns $\x_1, \ldots, \x_N \in \R^n$ are independent and centered, with common covariance $\E \x_i \x_i^* = \Sigma$, but may otherwise have strongly and nonlinearly dependent coordinates. Assuming only that quadratic forms of the columns concentrate uniformly at the optimal rate $| \x_i^* A \x_i - \Tr \Sigma A | \prec \| A \|_F$, together with polynomial norm moments and a standard nondegeneracy condition on $\Sigma$, we prove the optimal anisotropic local law: on regular spectral domains, uniformly down to spectral scales $\eta:= \Im z \geq N^{-1 + \tau}$,
    \[
    \big| \< \u , \big( (K-z)^{-1} - (-zI_n-z\widetilde m_0(z)\Sigma \big)^{-1} \bv \> \big| \prec \sqrt{\frac{\Im \widetilde m_0 (z)}{N\eta}} + \frac{1}{N\eta}
    \] 
    for all deterministic unit vectors $\u,\bv \in \C^n$, where $\widetilde m_0(z)$ is the Stieltjes transform of the deformed Marchenko-Pastur law. This removes the higher-cumulant tensor assumption of Fan, Ma, Paquette, and Wang (2026), thereby answering the question raised in their work. The result applies, among other examples, to every centered log-concave column distribution with bounded, nondegenerate covariance, nonlinear tilts of Gaussian vectors, deep random features, and a high-temperature spherical 4-spin model for which the cumulant assumption is known to fail.

        The proof adapts the zig-zag strategy to a Hermitian block linearization of $K$. Along a covariance-preserving Ornstein-Uhlenbeck flow, we follow a three-parameter characteristic of the block Dyson equation. The forward comparison combines It\^o calculus with Ward identities. For the reverse comparison, we differentiate fluctuation moments along the same Ornstein–Uhlenbeck flow and decompose its generator into one-column contributions. Conditional Schur-complement identities express these contributions in terms of linear and quadratic forms in one column, which are controlled by quadratic-form concentration alone, without expanding the vector into its coordinates. 

\end{abstract}

\maketitle

\setcounter{tocdepth}{1}
\tableofcontents

\clearpage

\section{Introduction}

\subsection{The problem}

Let $\x_1,\ldots,\x_N\in\R^n$ be independent centered random vectors with common covariance $\Sigma=\E\x_i\x_i^*$, and let
\begin{equation}\label{eq:Kdef}
 K=\frac1N\sum_{i=1}^N\x_i\x_i^*
   =\frac1N X X^*,\qquad X=[\x_1,\ldots,\x_N]\in \R^{n \times N},
\end{equation}
be the associated sample covariance matrix. We work in the proportional regime $n\asymp N$. Together with its companion Gram matrix $\widetilde K=N^{-1}X^*X$, the matrix $K$ is a basic object of random matrix theory and high-dimensional statistics, with applications ranging from covariance estimation and regression \cite{bai2010spectral,dobriban2018high,hastie2022surprises} and wireless communications \cite{tulino2004random,couillet2011random} to kernel methods and the theory of neural networks \cite{el2010spectrum,cheng2013spectrum,pennington2017nonlinear,louart2018random,jacot2018neural,mei2022generalization,couillet2022random}. Its global spectral theory goes back to the seminal work of Marchenko and Pastur \cite{marvcenko1967distribution}. Under suitable concentration assumptions on the columns, the empirical spectral distribution of $K$ is well approximated by the \emph{deformed Marchenko--Pastur law}, a deterministic distribution depending only on the aspect ratio $n/N$ and spectrum of $\Sigma$ \cite{silverstein1995empirical,silverstein1995strong}. This macroscopic description, however, neither resolves the spectrum at the scale of individual eigenvalue spacings nor describes the orientation of the eigenvectors, both of which are essential in many modern applications.

The \emph{local law} program \cite{erdHos2009local,erdHos2017dynamical} seeks to describe these finer spectral properties through the resolvent $R(z)=(K-z)^{-1}$, where $z=E+i\eta\in\C^+$. The imaginary part $\eta$ sets the resolution, and a local law asserts that $R(z)$ is close to a deterministic matrix uniformly down to the optimal scale $\eta\ge N^{-1+\tau}$, for any fixed $\tau >0$, just above the typical eigenvalue spacing. \emph{Averaged laws} control the normalized trace of $R(z)$, \emph{entrywise laws} its individual entries $R_{ij} (z)$, and \emph{anisotropic laws} the bilinear forms $\u^*R(z)\bv$ along arbitrary deterministic directions. The most comprehensive results have been established for separable columns $\x_i=\Sigma^{1/2}\z_i$, where $\z_i$ has independent entries. Pillai and Yin \cite{pillai2014universality} proved the averaged and entrywise laws for $\Sigma=I_n$, and Bloemendal, Erd\H{o}s, Knowles, Yau, and Yin \cite{bloemendal2014isotropic} established the corresponding isotropic law. For general population covariance, Knowles and Yin \cite{knowles2017anisotropic} proved the anisotropic law in the optimal form
\begin{equation}\label{eq:intro-anisotropic-local-law}
 \big|\u^*R(z)\bv-\u^*\Pi(z)\bv\big|\prec\Psi(z):=\sqrt{\frac{\Im\widetilde m_0(z)}{N\eta}}+\frac1{N\eta},\qquad
 \Pi(z)=-\big(z+z\widetilde m_0(z)\Sigma\big)^{-1},
\end{equation}
uniformly over deterministic unit vectors $\u,\bv\in\C^n$ and over $z$ in regular spectral domains. Here $\widetilde m_0$ is the Stieltjes transform of the deformed Marchenko--Pastur law associated with $\widetilde K$, characterized by the fixed-point equation \eqref{eq:MP} below, and $\prec$ denotes stochastic domination, that is, an inequality valid up to $N^{\eps}$ factors outside an event of probability at most $C_{\eps,D}N^{-D}$ for every $\eps,D>0$ (Definition~\ref{def:prec}). 

The anisotropic law is essential for many applications. It implies eigenvector delocalization in every deterministic direction and is a key ingredient in the spectral analysis of spiked models \cite{baik2005phase,bloemendal2016principal} and in the proofs of edge universality for general population covariance matrices \cite{bao2015universality,lee2016tracy,knowles2017anisotropic,fan2022tracy}. Estimates of this type also underpin the deterministic equivalents used to analyze ridge regression and random-feature models in high dimensions \cite{louart2018random,hastie2022surprises,mei2022generalization,cheng2024dimension,misiakiewicz2024non}. Separability, however, is a restrictive assumption. Proofs of \eqref{eq:intro-anisotropic-local-law} and related results typically rely on resolvent expansions and fluctuation averaging over the independent coordinates of $\z_i$. Yet many natural high-dimensional distributions admit no such representation. Examples include uniform measures on convex bodies and, more generally, log-concave measures; Gibbs measures with nonquadratic Hamiltonians; and random-feature models of the form $\x=\sigma(W\z)$ arising in machine learning \cite{pennington2017nonlinear,louart2018random,mei2022generalization}. In such settings, the column coordinates can exhibit dense nonlinear dependencies that cannot be eliminated by a linear change of variables. 

At global spectral scales, independence has long been known to be unnecessary. Bai and Zhou \cite{bai2008large} showed that the deformed Marchenko-Pastur limit persists whenever 
\begin{equation}\label{eq:qualitative_proba_convergence}
\frac{1}{n}(\x^*A\x-\Tr\Sigma A)\overset{\P}{\longrightarrow} 0,
\end{equation}
uniformly over matrices of bounded operator norm. Yaskov \cite{yaskov2016necessary} subsequently identified necessary and sufficient forms of this condition; see also \cite{elkaroui2009concentration,pajor2009limiting} for global laws with dependent coordinates and \cite{louart2018concentration} for deterministic equivalents of the resolvent at fixed spectral parameters under concentration of measure. The key observation is that, by the Schur complement formula, the resolvent of $K$ depends on any individual column only through linear and quadratic forms in that column. Controlling these forms therefore suffices to establish global laws. In the local setting, Pillai and Yin similarly noted that independence among the entries of a column enters their entrywise-law argument only through large-deviation estimates for these forms \cite[Theorem~3.6]{pillai2014universality}. 

At local spectral scales, the qualitative law of large numbers \eqref{eq:qualitative_proba_convergence} is no longer sufficient \cite[Section~2.6]{fan2026anisotropic}. The natural quantitative strengthening is concentration at the optimal rate,
\begin{equation}\label{eq:intro-quad-form-conc}
 \big|\x^*A\x-\Tr\Sigma A\big|\prec\|A\|_F,
\end{equation}
uniformly over deterministic matrices $A\in\C^{n\times n}$. This is the rate achieved by a Gaussian vector, and it holds in remarkable generality: for separable columns with uniformly bounded moments by classical large deviation estimates \cite{rudelson2013hanson,erdHos2017dynamical}; for vectors satisfying convex concentration with subpolynomial constant, in particular under a dimension-free logarithmic Sobolev inequality, by Adamczak's Hanson--Wright inequality for dependent vectors \cite{adamczak2015note}; and for every isotropic log-concave vector \cite{bao2025extreme} as a consequence of the recent progress of Chen \cite{chen2021almost} and Klartag and Lehec \cite{klartag2022bourgain} on the Kannan--Lov\'asz--Simonovits conjecture.

Building a local theory on the quadratic form concentration assumption alone is the program initiated by Fan, Ma, Paquette, and Wang \cite{fan2026anisotropic}. Under \eqref{eq:intro-quad-form-conc}, polynomial norm moments, and a standard nondegeneracy condition on $\Sigma$ (Assumptions~\ref{ass:A1} and~\ref{ass:A2} below), they proved the optimal averaged local law for $K$ and $\widetilde K$ together with an entrywise law for the Gram resolvent, and derived the standard consequences: absence of outliers, eigenvalue rigidity, and $\ell^\infty$-delocalization of the eigenvectors of $\widetilde K$. Their anisotropic law, however, requires an additional hypothesis on all higher-order cumulant tensors $\kappa_k(\x_i)$, $k\ge3$, of the columns \cite[Assumption~3]{fan2026anisotropic}: a non-generic decay of cumulant contractions against rank-one directions which drives their tensor-network analysis of fluctuation averaging. They verify the additional hypothesis for several structured models, including separable columns and certain conditionally mean-zero distributions, as well as one-layer random-feature models under restricted activations and random-feature tilts of Gaussian measures at small coupling.  As they demonstrate with a spherical spin-glass example (Example~\ref{ex:spherical-p-spin}), this higher-cumulant condition is genuinely stronger than quadratic-form concentration: there are measures satisfying a logarithmic Sobolev inequality with a dimension-free constant, hence \eqref{eq:intro-quad-form-conc}, whose fourth cumulant is a dense, disordered tensor that violates their Assumption~3. For such disordered-cumulant vectors, the anisotropic local law remained open.

This isolates a clean question: \emph{does concentration of quadratic forms \eqref{eq:intro-quad-form-conc} alone imply the anisotropic local law at the optimal scale?}

\subsection{Main contributions}

This paper answers the question affirmatively. Our main result is the optimal anisotropic local law \eqref{eq:intro-anisotropic-local-law} under the concentration assumption \eqref{eq:intro-quad-form-conc}. The following informal theorem is a special case of Theorem~\ref{thm:target} below.

\begin{theorem}[Anisotropic local law (informal)]\label{thm:informal}
Suppose that the columns $\x_i$ are independent and centered with covariance $\Sigma$, satisfy \eqref{eq:intro-quad-form-conc} and polynomial norm moments, and that $n\asymp N$, $\|\Sigma\|_\op\le C$, and a positive fraction of the eigenvalues of $\Sigma$ is bounded below. Let $\bD$ be a regular spectral domain bounded away from the origin. Then, uniformly in $z\in\bD$ and over deterministic unit vectors $\u,\bv\in\C^n$,
\[
 \big|\u^*R(z)\bv-\u^*\Pi(z)\bv\big|\prec\Psi(z),
\]
where $\Pi(z)$ and $\Psi(z)$ are stated in \eqref{eq:intro-anisotropic-local-law}.
The same bound holds for every block of the linearized resolvent of $X$, in particular for the Gram resolvent $(\widetilde K-z)^{-1}$ with deterministic equivalent $\widetilde m_0(z)I_N$.
\end{theorem}

We impose no symmetry and no structural condition on the columns beyond the quadratic concentration assumption \eqref{eq:intro-quad-form-conc}. In particular, Theorem~\ref{thm:informal} removes the higher-order cumulant condition in \cite[Theorem~2.8]{fan2026anisotropic}. The error $\Psi(z)$ and the spectral scale $\eta\ge N^{-1+\tau}$ are optimal, matching the separable theory \cite{knowles2017anisotropic}, and the estimate holds simultaneously for all $z\in\bD$ on a single event of overwhelming probability. As a consequence, the eigenvectors of $K$ and of $\widetilde K$ associated with eigenvalues in the regular part of the spectrum are delocalized in every deterministic direction (Corollary~\ref{cor:delocalization}). 

The hypotheses are verified in Section~\ref{sec:examples} for several classes of columns whose coordinates are nonlinearly dependent and for which no cumulant structure is assumed: columns satisfying convex concentration with a subpolynomial constant, in particular columns whose law satisfies a dimension-free logarithmic Sobolev inequality (Example~\ref{ex:convex}); all centered log-concave columns with bounded, nondegenerate covariance, for instance the centered uniform measure on a convex body (Example~\ref{ex:logconcave}); nonlinear tilts $e^{-\frac12\x^*\Lambda\x-W(\x)}$ of Gaussian measures with a uniformly bounded Hessian perturbation $W$, with no further structure imposed on $W$ (Example~\ref{ex:tilt}); deep random-feature models of fixed depth with Lipschitz activations (Example~\ref{ex:deepRF}); and independent samples from the spherical $p$-spin Gibbs measure at high temperature (Example~\ref{ex:spherical-p-spin}). The last example is the one used in \cite{fan2026anisotropic} to show that their cumulant hypothesis can fail under \eqref{eq:intro-quad-form-conc}; Theorem~\ref{thm:informal} yields the anisotropic local law for it. For the tilted and compositional models, verifying the cumulant hypothesis was left open in \cite{fan2026anisotropic}; we show that no such verification is needed. In the separable case, finally, our argument gives a new, dynamical proof of the anisotropic local law of Knowles and Yin \cite{knowles2017anisotropic} for real columns.

\subsection{Technical contribution}\label{sec:technical-contribution}

Our proof adapts the zig-zag strategy to sample covariance matrices and provide a novel implementation of the comparison step that requires no information about higher cumulants. Introduced by Cipolloni, Erd\H{o}s, and Schr\"oder \cite{cipolloni2022mesoscopic} and developed in \cite{cipolloni2023edge,cipolloni2024out,erdos2024eigenstate,erdos2025zigzag}, the zig-zag strategy proves local laws dynamically; see \cite{erdos2025lecture} for an exposition. The random matrix evolves along a stochastic flow while the spectral parameter follows a characteristic of the associated Dyson equation, chosen so that the leading It\^o correction is cancelled. The local law is transported from far away from the spectrum, where it is elementary, down to the optimal scale through a sequence of short steps. Each step combines transport along the flow and characteristic (the zig) with a Green function comparison that removes the Gaussian component injected by the flow (the zag).

Our implementation differs from previous zig-zag proofs in two main respects. First, since $K$ is quadratic in $X$, we work with the block linearization $\cR$ of $X$ with diagonal source $(Z,\beta I_N)$ with $Z\in\C^{n\times n}$ and $\beta\in\C$. The associated block Dyson equation admits an explicit characteristic that stays in the three-parameter family $Z\in\operatorname{span}(I_n,\Sigma)$, $\beta\in\C$, and along which the deterministic equivalent evolves by scalar multiplication. We use the covariance-preserving Ornstein--Uhlenbeck flow on the columns. The averaged law needed along the characteristic for the zig step is proved separately, by a static fluctuation-averaging argument adapted from \cite{fan2026anisotropic}. Second, and principally, our zag step treats each column as a whole. Previous implementations use entrywise Green function comparisons based on cumulant expansions in the matrix entries, which require information beyond \eqref{eq:intro-quad-form-conc}. We instead differentiate the fluctuation moments $\E|\u^*(\cR-M)\u|^{2L}$ along the column Ornstein--Uhlenbeck flow. The generator decomposes as a sum of one-column generators $\cL_i$, allowing us to estimate each contribution $\E\,\cL_i(\cdot)$ conditionally on the other columns.

Let us further comment on that second point and the reason the reverse comparison step avoids the cumulant assumption in \cite{fan2026anisotropic}. The static proofs of anisotropic laws in \cite{knowles2017anisotropic,fan2026anisotropic} expand $\u^*R\bv$ in the coordinates of a column. Taking high moments then produces contractions of cumulant tensors $\kappa_k(\x)$ against resolvents, whose control requires suitable structure in these tensors. In our proof, the Schur complement formula expresses the dependence of $\u^*\cR\u$ on $\x_i$, conditionally on the other columns, as an explicit rational function of a linear form $\x_i^*\mathbf b_i$ and a quadratic form $\x_i^*A_i\x_i$. Here $\mathbf b_i$ and $A_i$ depend only on the minor and are therefore independent of $\x_i$ (Lemma~\ref{lem:schur-one-column}). The one-column generator $\cL_i$ preserves this class of observables (Lemma~\ref{lem:product-and-chain-rule} and the proof of Lemma~\ref{lem:L_i-bounds}). These linear and quadratic forms are precisely the quantities controlled by \eqref{eq:intro-quad-form-conc}, and their centerings, $\E_i\,\x_i^*\mathbf b_i=0$ and $\E_i[\x_i^*A_i\x_i-\Tr\Sigma A_i]=0$, provides the leading cancellations. Two features make this second-order information sufficient. First, the Gaussian noise driving the flow gives the exact generator identity
\[
\frac{\de}{\de t}\E f(X_t)
=
\sum_i\E\,\cL_i f(X_t),
\]
with no higher-order expansion remainder. Second, each zag reverses only the short flow of the preceding zig. The generator estimates give a sublinear differential inequality for a normalized fluctuation moment and integrating its $2L$-th root yields an additive comparison error. The time-step condition and the gap between the coarse and target precisions allow this error to be absorbed at each zag; see Proposition~\ref{prop:full-block-zag}.

\subsection{Related work}

\emph{Local laws for covariance-type matrices.}
Beyond the separable setting \cite{pillai2014universality,bloemendal2014isotropic,knowles2017anisotropic}, the matrix Dyson equation has been used to establish local laws for Gram matrices with general variance profiles and for ensembles with correlated entries \cite{alt2017local,ajanki2019stability,erdos2019random}. These works impose quantitative conditions on correlations across entries, whereas we allow dense dependence among the coordinates of each independent column.

\emph{Dependent coordinates.}
For log-concave ensembles, \cite{adamczak2010quantitative} established operator-norm bounds for $K$, and \cite{chafai2018convergence} proved convergence of the extreme eigenvalues. Bao and Xu \cite{bao2025extreme} obtained eigenvalue rigidity for isotropic log-concave columns and, under an additional unconditionality assumption, upper-edge Tracy-Widom fluctuations. Their proof builds on the entrywise law of \cite{pillai2014universality} and the concentration estimates following from \cite{chen2021almost,klartag2022bourgain}. Deterministic equivalents at fixed spectral parameters under concentration assumptions were developed in \cite{elkaroui2009concentration,louart2018random,louart2018concentration}, while anisotropic estimates for nonlinear models at a fixed distance from the spectrum appear in \cite{wang2024nonlinear,fan2026anisotropic}. The work most closely related to ours is \cite{fan2026anisotropic}, discussed in the introduction. We use their regular domain framework and their operator norm bound \cite[Lemma~3.8]{fan2026anisotropic}. In Section~\ref{sec:averaged}, we also adapt their proof of the averaged law to the generalized resolvents along our characteristic.

\emph{Dynamical methods.}
The characteristic flow method goes back to the work of Pastur \cite{pastur1972spectrum}. It was used to give a short dynamical proof of a local law in \cite{vonsoosten2019random} and developed further for Dyson Brownian motion and Wigner matrices in \cite{huang2019rigidity,adhikari2020dyson,bourgade2021extreme}. Its combination with Green function comparison, the zig-zag strategy, appeared in \cite{cipolloni2022mesoscopic} and was developed and named in \cite{cipolloni2023edge,cipolloni2024out}; see \cite{erdos2025lecture} for a survey and \cite{erdos2024eigenstate,erdos2025cusp,erdos2025zigzag} for further applications. Section~\ref{sec:technical-contribution} explains how our implementation differs from these earlier proofs.

\emph{Statistics and machine learning.}
Random-feature maps give rise to columns of the form $\x=\sigma(W\z)$ with dependent coordinates. Global laws and deterministic equivalents for these models were established in \cite{pennington2017nonlinear,louart2018random,benigni2021eigenvalue,fan2020spectra,mei2022generalization,wang2024nonlinear}. These results underpin the analysis of high-dimensional ridge regression and interpolation \cite{wu2020optimal,hastie2022surprises,mei2022generalization,couillet2022random,defilippis2024dimension}. For related deep random-feature models (Example~\ref{ex:deepRF}), anisotropic deterministic equivalents at fixed spectral parameters were proved in \cite{schroder2023deterministic}. Subject to the remaining assumptions of Theorem~\ref{thm:target}, our result reduces the control of within-column dependence to the quadratic-form concentration hypothesis \eqref{eq:intro-quad-form-conc}.

\subsection{Organization}

Section~\ref{sec:main} states the assumptions, the deterministic theory, and the main results, and discusses examples and open problems. Section~\ref{sec:prelim} introduces the block linearization and its characteristic flow, reduces Theorem~\ref{thm:target} to a local law for the generalized resolvent (Theorem~\ref{thm:generalized-local-law}), and describes the zig-zag scheme. Sections~\ref{sec:zig} and~\ref{sec:zag} prove the zig and zag estimates, and Section~\ref{sec:proof-main} combines them with the spectral-scale bootstrap to prove the main results. Section~\ref{sec:averaged} proves the averaged law along the characteristic needed in the zig step.

\section{Model and main results}\label{sec:main}

\subsection{Notation and stochastic domination}

For $m\in\N$ we write $[m]=\{1,\ldots,m\}$. For a matrix $A$, denote by $A^*$ its conjugate transpose and by $\Im A=(A-A^*)/(2i)$ its imaginary part when $A$ is square. For vectors, $\|\cdot\|_2$ denotes the Euclidean norm; for matrices, $\|\cdot\|_\op$ and $\|\cdot\|_F$ denote the operator and Hilbert-Schmidt norms. Constants $c,C>0$ may change from line to line and depend only on the fixed model and regularity parameters. We write $a\asymp b$ when $c\le a/b\le C$. All test vectors in local laws are deterministic and may be complex. Uniformity in auxiliary parameters is made precise through the following notion.

\begin{defi}[Stochastic domination]\label{def:prec}
Let $\xi=\xi^{(N)}(u)$ and $\zeta=\zeta^{(N)}(u)\ge0$, with a possibly $N$-dependent parameter $u$.  We write $\xi\prec\zeta$, uniformly in $u$, if for every $\epsilon,D>0$ there is $C_{\epsilon,D}<\infty$ such that
\[
 \sup_u\P \bigl(|\xi^{(N)}(u)|>N^\epsilon \zeta^{(N)}(u)\bigr)\le C_{\epsilon,D}N^{-D}.
\]
We write $\xi=O_\prec (\zeta)$ if $|\xi|\prec\zeta$, and we say that an event holds with overwhelming probability if its complement has probability at most $C_DN^{-D}$ for every $D>0$.
\end{defi}

\subsection{Assumptions}

We work under the two assumptions of \cite{fan2026anisotropic}.

\begin{assumption}[Basic assumptions]\label{ass:A1}
    There exist constants $c,C >0$ such that $c\le n/N\le C$, $\| \Sigma \|_\op \leq C$, and $\Sigma$ is positive definite with at most $(1-c)n$ of its eigenvalues belonging to $[0,c]$.
\end{assumption}

\begin{assumption}[Quadratic-form concentration]\label{ass:A2}
 The vectors $\x_1,\ldots,\x_N\in\R^n$ are independent with $\E\x_i = 0$ and $\E\x_i \x_i^* = \Sigma$. For each $k \geq 1$, there exists a constant $C_k>0$ such that, uniformly in $i$,
\[
 \E\|\x_i\|_2^k\le n^{C_k}.
\]
Finally, for any $\eps,D >0$, there exists a constant $C \equiv C(\eps,D)>0$ such that, for every $i = 1, \ldots, N$ and every $A\in\C^{n\times n}$,
\begin{equation}\label{eq:assumption-qf}
 \P \big[ |\x_i^*A\x_i-\Tr(\Sigma A)| > N^\eps \|A\|_F \big] \le C N^{-D}.
\end{equation}
\end{assumption}

In other words, $\x_i^*A\x_i-\Tr(\Sigma A) \prec \|A\|_F$ uniformly in $i$ and in $A$. The columns need not be identically distributed, and no assumption is made on their higher cumulants. Applying \eqref{eq:assumption-qf} to $A=\bv\bv^*$ gives
\[
 |\bv^*\x_i|^2=\bv^*\Sigma \bv+O_\prec(\|\bv\|_2^2)
 =O_\prec(\|\bv\|_2^2),
\]
and therefore $|\bv^*\x_i|\prec\|\bv\|_2$.

\subsection{The deformed Marchenko--Pastur law}

Let $\sigma_1,\ldots,\sigma_n$ be the ordered eigenvalues of $\Sigma$.  The associated deformed Marchenko--Pastur Stieltjes transform $\widetilde m_0:\C_+\to\C_+$ is the unique solution in $\C_+$ of
\begin{equation}\label{eq:MP}
 z=-\frac1{\widetilde m_0(z)}
 +\frac1N\sum_{\alpha=1}^n
 \frac{\sigma_\alpha}{1+\sigma_\alpha\widetilde m_0(z)}.
\end{equation}
Existence and uniqueness go back to \cite{marvcenko1967distribution,silverstein1995analysis,silverstein1995strong}; see \cite[Section~2.2]{fan2026anisotropic} for a review of the facts used here. The function $\widetilde m_0$ is the Stieltjes transform of a probability measure $\widetilde\mu_0$ on $[0,\infty)$, the deformed Marchenko--Pastur law of $\widetilde K$, and under Assumption~\ref{ass:A1} we have $\supp \widetilde\mu_0\subset[0,C_*]$ for a constant $C_*$. We use the same symbol for its Stieltjes integral on $\C\setminus[0,\infty)$,
\[
 \widetilde m_0(z)=\int_{[0,\infty)}\frac{d\widetilde\mu_0(\lambda)}
 {\lambda-z},
\]
which is analytic off $[0,\infty)$ and continues to satisfy \eqref{eq:MP}. Define
\begin{equation}\label{eq:Pi-Psi}
 \Pi(z)=-(z+z\widetilde m_0(z)\Sigma)^{-1},\qquad
 \Psi(z)=\sqrt{\frac{\Im\widetilde m_0(z)}{N\eta}}
          +\frac1{N\eta},\quad z=E+i\eta.
\end{equation}
The matrix $\Pi(z)$ is the deterministic equivalent of $R(z)$, and $\Psi(z)$ is the error control parameter of \cite{knowles2017anisotropic,fan2026anisotropic}.

The support of $\widetilde\mu_0$ is described by the meromorphic inverse of \eqref{eq:MP},
\begin{equation}\label{eq:z0-inverse}
 z_0(m)=-\frac1m+\frac1N\sum_{\alpha=1}^n
 \frac{\sigma_\alpha}{1+\sigma_\alpha m}.
\end{equation}
Its real critical points, counted with multiplicity on the extended real line, are $0 > m_1 \geq m_2 \geq m_3 > m_4 \geq \ldots > m_{2p_*-2} \geq m_{2p_*-1} > - \sigma_{n}^{-1}$ and $m_{2p_*}\in(-\infty,-\sigma_{n}^{-1})\cup(0,\infty]$; see \cite[Section~2.2]{fan2026anisotropic} for details. Setting $x_j=z_0(m_j)$, one has
\[
 \supp  \widetilde\mu_0\cap(0,\infty)
 =\bigcup_{k=1}^{p_*}[x_{2k},x_{2k-1}]\cap(0,\infty).
\]
The points $x_j$ are the edges (or cusps) of $\widetilde\mu_0$, and the intervals $[x_{2k},x_{2k-1}]$ are its bulk components. Finally, at each point $x>0$, it was shown in \cite{silverstein1995analysis} that $\widetilde \mu_0$ admits a continuous density    
\[
\rho_0 (x) = \lim_{z \in \C_+ \to x} \frac{1}{\pi} \Im \widetilde m_0(z).
\]

Local laws are stated on spectral domains on which the deterministic theory is regular. We use the regular spectral domains of \cite[Definition~2.3]{fan2026anisotropic}, which follow \cite{knowles2017anisotropic}, with one additional requirement: the domain must be bounded away from the origin and from the negative real axis. This requirement is needed for the square-root reparametrization of Section~\ref{sec:block}. For definiteness, we record the properties that are used below.

\begin{defi}
A ($\tau$-)regular spectral domain is a set $\bD\subset\{z\in\C^+:N^{-1+\tau}\leq\Im z\le1\}$ such that the following conditions hold for some constants $C,c>0$:
\begin{enumerate}[label=(\alph*)]
    \item (Vertically closed) For all $E+i\eta\in\bD$, one has $E+it\in\bD$ for every $t\in[\eta,1]$.
    \item (Boundedness) For all $z\in\bD$, we have
    \begin{equation}\label{eq:regular-basic}
    c\leq \Re z,\quad 
    c\le |z|\le C,\qquad
    c\le|\widetilde m_0(z)|\le C,\qquad
    \min_\alpha|1+\sigma_\alpha\widetilde m_0(z)|\ge c.
    \end{equation}
    \item (Regular density) Let $\kappa=\min_{1\le j\le2p_*}|x_j-E|$ denote the distance to the closest edge. For all $z\in\bD$, $z = E + i \eta$, we have
    \begin{equation}\label{eq:regular-density}
    \Im \widetilde m_0(z) \asymp
    \begin{cases}
    \sqrt{\kappa+\eta},
     &E\in\supp  \widetilde\mu_0,\\[2pt]
    \eta/\sqrt{\kappa+\eta},
     &E\notin\supp  \widetilde\mu_0.
    \end{cases}
    \end{equation}
    \item (Fixed-point stability) For all $z\in\bD$, we have
    \[
        c\sqrt{\kappa+\eta}\leq|\widetilde m_0(z)|^2|z_0'(\widetilde m_0(z))|\leq C\sqrt{\kappa+\eta}.
    \]
    Moreover, there exists a constant $c_*>0$ such that whenever $\kappa + \eta \leq c_*$, we also have
    \[
        |z_0''(\widetilde m_0(z))|\geq c.
    \]
\end{enumerate}
\end{defi}

The preceding regularity assumptions hold in a regular bulk and near a regular edge. For a regular bulk component $[x_{2k},x_{2k-1}]$, define
\[
    \bD_k^b(\delta,\tau) = \{ E + i \eta : x_{2k} + \delta \leq E \leq x_{2k-1}- \delta, \ N^{-1 + \tau} \leq \eta \leq 1 \}.
\]
For a regular soft edge $x_j$ uniformly separated from zero, define
\[
    \bD_j^e(\delta,\tau) = \{ E + i \eta : |E - x_j| \leq \delta, \ N^{-1 + \tau} \leq \eta \leq 1 \}.
\]
The following lemma is adapted from \cite{knowles2017anisotropic} and \cite[Lemma 2.4]{fan2026anisotropic}.

\begin{lemma} 
 Suppose Assumption \ref{ass:A1} holds. Fix any constant $\tau \in (0,1)$.
\begin{enumerate}
\item[(i)] For a bulk component $[x_{2k},x_{2k-1}]$, suppose there exist constants $\delta,c_0 >0$ such that $\rho_0 (x) > c_0$ for all $x \in [x_{2k} + \delta, x_{2k-1}- \delta]$. Then $\bD_k^b(\delta,\tau)$ is a regular domain.
\item[(ii)] For an edge $x_j = z_0 (m_j)$, suppose there exists a constant $\delta>0$ such that $x_j >\delta$, $|x_j - x_k| >\delta$ for each $k \neq j$, and $\min_{\alpha} | m_j + \sigma_\alpha^{-1} | > \delta$. Then there exists a constant $\delta'>0$ such that $\bD_j^e(\delta',\tau)$ is a regular domain.
\end{enumerate}
\end{lemma}

\begin{proof}
Conditions (a)--(c) follow directly from \cite[Lemma 2.4]{fan2026anisotropic}. For condition (d), suppose $z\in\bD_j^e(\delta',\tau)$ (near a regular edge). Regular-edge stability then gives $|\widetilde m_0(z)|^2|z_0'(\widetilde m_0(z))|\asymp\sqrt{\kappa+\eta}$; this is the estimate
\cite[(A.16)]{knowles2017anisotropic}, also used in
\cite[Lemma~C.3(b)]{fan2022tracy}. The edge curvature $|z_0''(\widetilde m_0(z))|$ is bounded away from zero by
\cite[Lemma~A.3, equation (A.9)]{knowles2017anisotropic}; see also
\cite[Propositions~3.1--3.3]{fan2022tracy}. If $z\in\bD_k^b(\delta,\tau)$ (in a regular bulk), then bulk stability gives $|\widetilde m_0(z)|^2|z_0'(\widetilde m_0(z))|\asymp 1$; this is the estimate \cite[Lemma A.6, equation (A.17)]{knowles2017anisotropic}. Combining the two cases completes the proof.
\end{proof}

The population matrix $\Sigma=\Sigma_N$, the domain $\bD=\bD_N$, its edge locations $x_j=x_{j,N}$, and the number of connected components of the domain may depend on $N$. All constants in \eqref{eq:regular-basic}--\eqref{eq:regular-density}, and all constants derived from them below, are uniform in $N$. At spectral parameters at fixed distance from the spectrum, including near the negative real axis, an anisotropic law under Assumptions~\ref{ass:A1} and~\ref{ass:A2} alone was already obtained in \cite[Theorem~2.10]{fan2026anisotropic}, with error $N^{-1/2+o(1)}$; the content of the present paper is the local regime.

\subsection{Main theorem}
Define the resolvents
\[
    K=N^{-1}XX^*,\quad \widetilde K=N^{-1}X^*X,\quad R(z)=(K-zI_n)^{-1},\quad \widetilde R(z)=(\widetilde K-zI_N)^{-1}.
\]
The full covariance pencil and its deterministic equivalent are
\begin{equation}\label{eq:FMPW-pencil}
    \mathcal L_{\rm F}(z)=\begin{pmatrix}-zI_n&N^{-1/2}X\\
 N^{-1/2}X^*&-I_N\end{pmatrix},\qquad \mathcal M_{\rm F}(z)=\diag\bigl(\Pi(z),z\widetilde m_0(z)I_N\bigr).
\end{equation}
 By the Schur complement formula,
 \[
\cL_F(z) = \begin{pmatrix}R(z)&N^{-1/2}R(z)X\\
 N^{-1/2}X^* R(z)&z \widetilde R(z)\end{pmatrix},
 \]
 and both resolvents are recovered from the diagonal blocks.

\begin{theorem}[Anisotropic local law]\label{thm:target}
Suppose Assumptions \ref{ass:A1} and \ref{ass:A2} hold. Fix $\tau>0$, and let $\bD$ be a $\tau$-regular spectral domain. Then, for all deterministic pair of unit vectors $U,V\in\C^{n+N}$,
\begin{equation}\label{eq:main-full}
 \left|U^*\left(\mathcal L_{\rm F}(z)^{-1}
 -\mathcal M_{\rm F}(z)\right)V\right|\prec\Psi(z)
\end{equation}
uniformly for $z\in\bD$.  In particular, for all deterministic
unit vectors $\u,\bv\in\C^n$ and $\widetilde \u, \widetilde \bv \in \C^N$,
\begin{equation}\label{eq:main-upper}
 |\u^*(R(z)-\Pi(z))\bv|\prec\Psi(z),\qquad
 |\widetilde \u^*(\widetilde R(z)-\widetilde m_0(z)I_N)\widetilde \bv|\prec\Psi(z).
\end{equation}
Moreover, the uniformity in the spectral parameter holds on a single event:
for every $\epsilon,D>0$,
\begin{equation}
 \P\left(
 \sup_{z\in\bD}
 \frac{\left|U^*(\mathcal L_{\rm F}(z)^{-1}
       -\mathcal M_{\rm F}(z))V\right|}{\Psi(z)}
 >N^\epsilon\right)\le C_{\epsilon,D}N^{-D}.           \label{eq:main-simultaneous}
\end{equation}
\end{theorem}

\begin{remark}[Relation to \cite{fan2026anisotropic}]\label{rem:fmpw-comparison}
Theorem~\ref{thm:target} gives the counterpart of \cite[Theorem~2.8]{fan2026anisotropic} without its higher-cumulant Assumption~3; only their Assumptions~1 and~2 (our Assumptions~\ref{ass:A1} and \ref{ass:A2}) remain. The error $\Psi(z)$ and the scale $\eta\ge N^{-1+\tau}$ are optimal, matching the separable theory of \cite{knowles2017anisotropic}. The second bound in \eqref{eq:main-upper} follows from \eqref{eq:main-full} since $|z|\ge c$ on $\bD$. A word on notation: following \cite{knowles2017anisotropic}, our $\Pi(z)$ in \eqref{eq:Pi-Psi} denotes the \emph{deterministic} equivalent of the resolvent $R(z)$; in \cite{fan2026anisotropic}, the same symbol denotes the random linearized resolvent, which is our $\mathcal L_{\rm F}(z)^{-1}$.
\end{remark}

The standard consequence of an anisotropic law is the delocalization of eigenvectors in arbitrary deterministic directions. For an eigenvalue $\lambda$ of $K$ (respectively $\widetilde K$), let $P_\lambda$ (respectively $\widetilde P_\lambda$) denote the orthogonal projection onto the corresponding eigenspace.

\begin{coro}[Delocalization in deterministic directions]\label{cor:delocalization}
Suppose Assumptions \ref{ass:A1} and \ref{ass:A2} hold. Let $I\subset(0,\infty)$ be a compact interval such that, for every $\tau\in(0,1)$, the domain $\bD_\tau=\{E+i\eta:E\in I,\ N^{-1+\tau}\le\eta\le1\}$ is a $\tau$-regular spectral domain. Then, for all deterministic unit vectors $\u\in\C^n$ and $\bv\in\C^N$,
\[
 \sup_{\lambda\in I\cap\operatorname{Spec}K}\|P_\lambda\u\|
 \prec N^{-1/2},\qquad
 \sup_{\lambda\in I\cap\operatorname{Spec}\widetilde K}
 \|\widetilde P_\lambda\bv\|\prec N^{-1/2}.
\]
\end{coro}

By \cite[Lemma~2.4]{fan2026anisotropic}, the hypothesis on $I$ holds for every compact subinterval of the interior of a bulk component on which the density of $\widetilde\mu_0$ is bounded below, and for a neighborhood of every regular edge $x_j>0$.

\begin{proof}
Fix $0 <\epsilon<1$ and $D>0$, put $\tau=\epsilon$ and $\eta=N^{-1+\epsilon}$, and apply \eqref{eq:main-simultaneous} with the domain $\bD_\epsilon$, the vectors $U=V=(\u,0)$, and the exponent $\epsilon$. On the resulting event of probability at least $1-C_{\epsilon,D}N^{-D}$, for every eigenvalue $\lambda\in I$ of $K$ the point $z=\lambda+i\eta$ lies in $\bD_\epsilon$, and the spectral decomposition of $K$ gives
\[
 \|P_\lambda\u\|_2^2\le\sum_{\mu\in\spec K}\frac{\eta^2}{(\mu-\lambda)^2+\eta^2}\|P_\mu\u\|_2^2
 =\eta\,\Im\,\u^*R(z)\u\le\eta\big(\|\Pi(z)\|_\op+N^{\epsilon}\Psi(z)\big)\le CN^{-1+2\epsilon}.
\]
In the last step we used $\|\Pi(z)\|_\op\le (c\min_\alpha|1+\sigma_\alpha\widetilde m_0(z)|)^{-1}\le C$ by \eqref{eq:regular-basic}, and $\Psi(z)\le C$ since $N\eta\ge1$. For $\widetilde K$, the lower-right block of \eqref{eq:FMPW-pencil} gives, with $V=(0,\bv)$ and $z=\lambda+i\eta$ for $\lambda\in I\cap\spec\widetilde K$,
\[
 \|\widetilde P_\lambda\bv\|_2^2\le\eta\,\Im\,\bv^*\widetilde R(z)\bv\le\frac{\eta}{|z|}\big(|z\widetilde m_0(z)|+N^{\epsilon}\Psi(z)\big)\le CN^{-1+2\epsilon}.
\]
Since $\epsilon$ and $D$ are arbitrary, both claims follow.
\end{proof}

\subsection{Discussion and examples}\label{sec:examples}

Assumption~\ref{ass:A2} is a quantitative strengthening of the condition
\begin{equation}\label{eq:weak-concentration}
 \frac{1}{n} \big(\x^*A\x-\Tr \Sigma A\big)\overset{\P}{\longrightarrow} 0\qquad\text{uniformly over }\|A\|_\op\le1,
\end{equation}
under which Bai and Zhou \cite{bai2008large} established the global Marchenko-Pastur law; see \cite{yaskov2016necessary} for its necessary and sufficient form. Condition \eqref{eq:weak-concentration} alone does not suffice at local scales: \cite[Section~2.6]{fan2026anisotropic} exhibit columns satisfying \eqref{eq:weak-concentration} for which $\Tr K$ fluctuates on a polynomial scale, which is incompatible with eigenvalue rigidity, hence with an averaged local law at the optimal scale. The rate in \eqref{eq:assumption-qf} is the one satisfied by a Gaussian column, and by Markov's inequality it can equivalently be formulated as $L^p$ bounds on the centered quadratic form.

Below we describe classes of column distributions satisfying Assumptions \ref{ass:A1} and \ref{ass:A2}, and hence the anisotropic local law in Theorem \ref{thm:target}. For separable columns $\x=\Sigma^{1/2}\z$ with independent, centered entries having uniformly bounded moments of all orders, \eqref{eq:assumption-qf} is classical and follows from large-deviation estimates for quadratic forms \cite[Theorem~7.7]{erdHos2017dynamical}; in this case, Theorem \ref{thm:target} recovers the anisotropic local law (for real-valued columns) of Knowles and Yin \cite{knowles2017anisotropic} by a new, dynamical proof. The random-feature model $\x=\sigma(W\z)$ with restricted activations and a class of conditionally mean-zero distributions were shown in \cite{fan2026anisotropic} to satisfy Assumption~\ref{ass:A2} together with their cumulant condition. By contrast, the following examples concern columns with nonlinear coordinate dependence for which no cumulant structure is assumed.

\begin{example}[Convex concentration]\label{ex:convex} Suppose that the independent centered columns $\x_i$ have covariance $\Sigma$ satisfying Assumption~\ref{ass:A1}, and that there is a deterministic sequence $K_N=N^{o(1)}$ such that, uniformly in $i\in[N]$, every convex $1$-Lipschitz function $f:\R^n\to\R$ satisfies
\begin{equation}\label{eq:convex-concentration}
 \P\big(|f(\x_i)-\E f(\x_i)|>t\big)\le 2\exp(-t^2/K_N^2), \qquad t >0.
\end{equation}
Adamczak proved a Hanson--Wright inequality for such vectors.
\begin{proposition}[{\cite[Theorem~2.3]{adamczak2015note}}]\label{prop:adamczak}
Let $\x\in\R^n$ be a centered random vector satisfying the convex concentration property \eqref{eq:convex-concentration} with constant $K$. There exists an absolute constant $c>0$ such that, for every $A\in\R^{n\times n}$ and every $t>0$,
\[
 \P\big(|\x^*A\x-\E\,\x^*A\x|>t\big)\le 2\exp\Big(-c\min\Big\{\frac{t^2}{K^4\|A\|_F^2},\frac{t}{K^2\|A\|_\op}\Big\}\Big).
\]
\end{proposition}

Since $\E\,\x_i^*A\x_i=\Tr\Sigma A$, the choice $t=N^\eps\|A\|_F$ and $K=K_N=N^{o(1)}$ gives \eqref{eq:assumption-qf} for real $A$; complex $A$ are handled by splitting into real and imaginary parts. Applying \eqref{eq:convex-concentration} to $f=\|\cdot\|_2$, together with $\E\|\x_i\|_2\le(\Tr\Sigma)^{1/2}\le\sqrt{Cn}$, gives the norm moments. Hence Assumption~\ref{ass:A2} holds, and Theorem~\ref{thm:target} applies. Several familiar classes of distributions satisfy the convex concentration property \eqref{eq:convex-concentration}. For example, it holds with $K_N=O(1)$ for vectors with independent, uniformly bounded coordinates, and for their images under linear maps of bounded operator norm, by Talagrand's inequality \cite{talagrand1996new}. By the Herbst argument \cite{bakry2014analysis}, it also holds, for all Lipschitz functions and not only convex ones, whenever the law of $\x_i$ satisfies a logarithmic Sobolev inequality with a dimension-free constant; by the Bakry--\'Emery criterion, this is the case for strongly log-concave vectors. Proposition \ref{prop:adamczak} is what Examples~\ref{ex:tilt}--\ref{ex:spherical-p-spin} use, together with stability of the Lipschitz concentration property under composition with Lipschitz maps. 
\end{example}

\begin{example}[Log-concave columns]\label{ex:logconcave}
    Let the columns be independent, centered, and log-concave, with covariance $\Sigma$ satisfying Assumption~\ref{ass:A1}; for instance, uniformly distributed on an isotropic convex body. Then $\z_i=\Sigma^{-1/2}\x_i$ is isotropic and log-concave. Using the polylogarithmic bound on the Kannan--Lov\'asz--Simonovits constant of \cite{klartag2022bourgain}, following the earlier progress in \cite{chen2021almost}, Bao and Xu \cite[Lemma~3.3]{bao2025extreme} proved that $|\z_i^*B\z_i-\Tr B|\prec\|B\|_F$ for every deterministic $B\in\C^{n\times n}$. With $B=\Sigma^{1/2}A\Sigma^{1/2}$ this gives
\[
 |\x_i^*A\x_i-\Tr\Sigma A|\prec\|\Sigma^{1/2}A\Sigma^{1/2}\|_F\le\|\Sigma\|_\op\|A\|_F,
\]
and the norm moments follow from Borell's lemma \cite{borell1974convex}, which gives exponential tails for $\|\z_i\|_2/\sqrt n$. Hence Assumption~\ref{ass:A2} holds for every log-concave column distribution, and Theorem~\ref{thm:target} applies whenever the covariance satisfies Assumption~\ref{ass:A1}. Bao and Xu \cite{bao2025extreme} proved eigenvalue rigidity for isotropic log-concave columns; Theorem~\ref{thm:target} adds anisotropic control for general covariance, with no unconditionality or other symmetry assumption.
\end{example}

\begin{example}[Nonlinear tilts of Gaussian measures]\label{ex:tilt}
Let $cI_n\preceq\Lambda\preceq CI_n$, let $\w\in\R^n$ have density proportional to $e^{-V}$ with
\[
 V(\y)=\frac12\y^*\Lambda\y+W(\y),\qquad \sup_{\y\in\R^n}\|\nabla^2W(\y)\|_\op\le\frac c2,
\]
and let the columns be independent copies of $\x=\w-\E\w$. Then $\frac c2I_n\preceq\nabla^2V\preceq(C+\frac c2)I_n$, so the law of $\w$ satisfies a logarithmic Sobolev inequality with a constant depending only on $c$ by the Bakry--\'Emery criterion, and $\x$ satisfies \eqref{eq:convex-concentration} with $K_N=O(1)$. The Brascamp--Lieb and Cram\'er--Rao inequalities give $(C+\frac c2)^{-1}I_n\preceq\Sigma\preceq\frac2cI_n$, so Assumption~\ref{ass:A1} holds, and Example~\ref{ex:convex} gives Assumption~\ref{ass:A2}. At sufficiently small coupling $\lambda$, this class contains the random-feature tilts $W(\y)=\lambda\sum_{i\le d}\sigma(\bw_i^*\y)$ with bounded $\sigma''$ and $\|\sum_{i\le d}\bw_i\bw_i^*\|_\op\le C$, considered in \cite[Proposition~2.18]{fan2026anisotropic}, where the cumulant condition was verified under bounds on the derivatives of $\sigma$ of all orders. Theorem~\ref{thm:target} requires neither the ridge structure nor derivative bounds beyond the second order: the perturbation $W$ may couple all coordinates, and its cumulant tensors need not have any sparse or low-rank structure.
\end{example}

\begin{example}[Deep random features]\label{ex:deepRF}
    Fix a depth $L$ and widths $d_0,\ldots,d_L=n$. For $\ell\in[L]$, let $T_\ell:\R^{d_{\ell-1}}\to\R^{d_\ell}$ be a deterministic linear map with $\|T_\ell\|_\op\le C$, and let $\sigma_\ell:\R^{d_\ell}\to\R^{d_\ell}$ be $C$-Lipschitz, for instance the entrywise application of a Lipschitz activation function. Let $\by\in\R^{d_0}$ be a random vector whose law satisfies a logarithmic Sobolev inequality with a dimension-free constant, let
\[
 F(\by)=\sigma_L\big(T_L\sigma_{L-1}(T_{L-1}\cdots\sigma_1(T_1\by)\cdots)\big),
\]
and let the columns be independent copies of $\x=F(\by)-\E F(\by)$, whose covariance $\Sigma$ is assumed to satisfy Assumption~\ref{ass:A1}. Since $L$ is fixed, $F$ is Lipschitz with a bounded constant, so the Herbst argument gives \eqref{eq:convex-concentration} with $K_N=O(1)$, and Example~\ref{ex:convex} gives Assumption~\ref{ass:A2}. In random-feature models the maps $T_\ell$ are random weight matrices independent of $\by$; the conclusion then holds conditionally on the weights, on the event that their operator norms are bounded and Assumption \ref{ass:A1} holds. Anisotropic deterministic equivalents at fixed spectral parameters for related deep random-feature models were obtained in \cite{schroder2023deterministic}. The cumulant condition of \cite{fan2026anisotropic} was verified in \cite[Proposition~2.17]{fan2026anisotropic} for one layer and polynomial or entire activations with bounded growth, and its verification for deep models was left open; under the assumptions above, Theorem~\ref{thm:target} gives the anisotropic local law at the optimal scale without it.
\end{example}

\begin{example}[High-temperature spherical $p$-spin model]\label{ex:spherical-p-spin} 
    Our last example is the spherical spin-glass model considered in \cite[Section~2.6]{fan2026anisotropic}. Fix an even integer $p$ and let $J\in(\R^n)^{\otimes p}$ be a symmetric tensor with independent standard Gaussian entries up to symmetry, so that $\|J\|_F\asymp n^{p/2}$ with overwhelming probability. Define
\begin{equation}\label{eq:p-spin}
 H_J(\x)=n^{-(p-1)/2}\langle J,\x^{\otimes p}\rangle,\qquad \x\in S^{n-1}(\sqrt n),
\end{equation}
and let $\mu_{J,\beta}$ be the Gibbs measure with density proportional to $e^{-\beta H_J(\x)}$ with respect to the uniform probability measure on $S^{n-1}(\sqrt n)$, where $\beta>0$ is the inverse temperature. Conditionally on $J$, let $\x_1,\ldots,\x_N$ be independent samples from $\mu_{J,\beta}$. For every fixed $p$ there exists $\beta_0(p)>0$ such that, with overwhelming probability over $J$, the measure $\mu_{J,\beta}$ satisfies a logarithmic Sobolev inequality with a constant independent of $n$ for all $\beta\le\beta_0(p)$; this follows from the Bakry--\'Emery criterion on the sphere \cite[Section~2.6]{fan2026anisotropic}. Since $p$ is even, $\mu_{J,\beta}$ is symmetric and $\E\x_i=0$; since $\|\x_i\|_2^2=n$, all norm moments are trivial and $\Tr\Sigma=n$. The Poincar\'e inequality gives $\|\Sigma\|_\op=O(1)$, which together with $\Tr\Sigma=n$ forces a positive fraction of the eigenvalues of $\Sigma$ to be bounded below, so Assumption~\ref{ass:A1} holds. The logarithmic Sobolev inequality gives dimension-free Lipschitz concentration, and Example~\ref{ex:convex} gives Assumption~\ref{ass:A2}. Consequently, conditionally on $J$, Theorem~\ref{thm:target} gives the anisotropic local law for every fixed even $p$ at sufficiently high temperature. For $p=4$ and $\beta\leq \beta_*$ for some constant $\beta_*>0$, the fourth cumulant tensor of $\mu_{J,\beta}$ is disordered and violates the cumulant condition of \cite{fan2026anisotropic}. Theorem~\ref{thm:target} therefore establishes the anisotropic local law in an example that lies outside the reach of the cumulant hypothesis.
\end{example}

\medskip
\noindent\emph{Open problems.}
We mention three directions. First, our regular spectral domains are bounded away from the origin, as required by the square-root reparametrization of Section \ref{sec:block}. Establishing the anisotropic law on regular domains approaching zero, in particular near the hard edge when $n/N \to 1$, remains open; for independent entries, see \cite{cacciapuoti2013local,bloemendal2014isotropic}.  Second, at a uniformly regular soft edge bounded away from zero, \cite[Corollary~2.7]{fan2026anisotropic} gives optimal eigenvalue rigidity under Assumptions~\ref{ass:A1} and \ref{ass:A2}, while Theorem~\ref{thm:target} supplies the corresponding anisotropic local law. It is therefore natural to ask whether these two assumptions alone imply edge universality at every 
regular soft edge. This is known for broad classes of independent-entry models \cite{bao2015universality,lee2016tracy,knowles2017anisotropic,ding2018necessary,fan2022tracy}, and Bao and Xu \cite{bao2025extreme} proved the Tracy--Widom law for isotropic log-concave columns under an additional unconditionality assumption. A proof in the present setting would require an additional dynamical or Green-function comparison argument adapted to whole-column dependence. Third, the zig-zag strategy was developed as a dynamical method for proving, in particular, multi-resolvent local laws \cite{cipolloni2022mesoscopic,erdos2025lecture}. Developing multi-resolvent analogues of both the characteristic-flow and whole-column comparison estimates in the present covariance setting could lead to eigenvector-overlap or eigenstate-thermalization-type bounds for nonseparable sample covariance matrices, in analogy with the results for Wigner-type matrices \cite{cipolloni2023edge,erdos2024eigenstate}. For general $\Sigma$, one expects the relevant deterministic eigenvector profile to depend on the spectral energy. Such conclusions require substantial additional analysis and do not follow from the one-resolvent anisotropic law proved here.

\section{Notation and preliminaries}\label{sec:prelim}
This section introduces the notation and preliminary results needed to prove the anisotropic local law (Theorem \ref{thm:target}). We follow the zig--zag strategy of \cite{cipolloni2022mesoscopic,erdos2025lecture}, adapting it to the block setting. Rather than proving Theorem \ref{thm:target} directly, we first establish a local law for a slightly different linearized block and then transfer the result to the setting of Theorem \ref{thm:target}.

\subsection{The characteristic flow}\label{sec:block}
For $Z\in\C^{n\times n}$ and $\beta\in\C$, with $\Im Z>0$ and $\Im\beta>0$, define
\begin{equation} \label{eq:full-L}
 \cL(X;Z,\beta) =
 \begin{pmatrix}
   -Z&N^{-1/2}X\\
   N^{-1/2}X^*&-\beta I_N
 \end{pmatrix},
 \qquad
 \cR(X;Z,\beta)=\cL(X;Z,\beta)^{-1}=\begin{bmatrix}
     \cR_{11} & \cR_{12}\\
     \cR_{21} & \cR_{22}
 \end{bmatrix}.
\end{equation}
At the scalar source $(Z,\beta)=(wI,w)$, the block matrix inversion formula gives
\begin{equation} \label{eq:singular-linearization-blocks}
 \cR_{11}=w(N^{-1}XX^*-w^2)^{-1},\qquad \cR_{22}=w(N^{-1}X^*X-w^2)^{-1}.
\end{equation}
The covariance pencil $\cL_{\rm F}(z)$ defined in \eqref{eq:FMPW-pencil} satisfies 
\begin{equation}\label{eq:linear-pencil-equivalence}
    \cL_{\rm F}(w^2)^{-1} =\diag(I_n,wI_N)\cR(X,wI_n,w)\diag(w^{-1}I_n,I_N).
\end{equation}
To use this identity, we must choose $\sqrt{z}=w$; the next lemma shows that this choice is well-defined.
\begin{lemma}\label{lem:z-companion-w}
Let $\bD$ be a regular domain. Then the principal branch of the square root $\sqrt{\bD}\subseteq\C^+$ exists. Moreover, for any $z=E+i\eta\in\bD$ and $w:=\sqrt{z}\in\C^+$, we have
\[
    c\leq|w|\leq C,\quad c\eta\leq\Im{w}\leq C\eta,\quad \Re{w}\geq c,
\]
for some constants $C,c>0$.
\end{lemma}
\begin{proof}
By definition, a regular domain $\bD$ is bounded away from $0$. Thus, a branch of the square root exists, and $c\leq|w|\leq C$. Moreover, for all $z=E+i\eta\in\bD$, we can write $z=|z|e^{i\theta}$ for some $0<\theta<\pi/2$. Hence, defining $w=|z|^{1/2}e^{i\theta/2}$ shows that $w\in\C^+$ and $\Re w\geq 0$.

Let $w=a+ib$. It follows that $w^2=a^2-b^2+2abi=E+i\eta$, so $a^2\geq c+b^2\geq c$, which proves $\Re w\geq \sqrt{c}$. Finally, rearranging $2ab=\eta$ gives $b/\eta=1/(2a)$. Since $a\leq |w|\leq C$ and $a\geq \sqrt{c}$, this proves $c\eta\leq\Im{w}\leq C\eta$.
\end{proof}

Throughout the paper, for a regular domain $\bD$ and $z\in\bD$, we refer to $w:=\sqrt{z}\in\C^+$ as the square root of $z$. We reserve the notation $\eta$ for the imaginary part of $z$ and always write $z=E+i\eta$. We also apply Lemma \ref{lem:z-companion-w} to switch between $\Im w$ and $\Im z=\eta$ without further comment.

Put $\cD=\diag(Z,\beta I_N)$, and define the self-energy map
\begin{equation}\label{eq:block-self-energy}
 \cS[A]
 =
 \begin{pmatrix}
  (N^{-1}\Tr A_{22})\Sigma&0\\
  0&(N^{-1}\Tr\Sigma A_{11})I_N
 \end{pmatrix}.
\end{equation}
The deterministic approximation to the generalized pencil is a block-diagonal solution of its matrix Dyson equation (MDE).  For a source $(Z,\beta)$, such a solution has the form
\[
 M=\diag(P,qI_N),\qquad P=-(Z+q\Sigma)^{-1},\qquad q=-(\beta+s)^{-1},\qquad s=N^{-1}\Tr\Sigma P,
\]
and obeys
\begin{equation}\label{eq:block-MDE}
 -M^{-1}=\cD+\cS[M].
\end{equation}
No existence or uniqueness statement for an arbitrary generalized source is needed below; the only branch we use is constructed explicitly.

\begin{lemma}[Explicit block characteristic and MDE transport]
\label{lem:block-M-transport}
Suppose an invertible block-diagonal matrix $M_0$ and a source $\cD_0$ satisfy \eqref{eq:block-MDE} at time $t_0$. On every interval on which $\Im Z_t>0$ and $\Im\beta_t>0$, define
\begin{equation}\label{eq:explicit-block-characteristic}
 M_t=e^{(t-t_0)/2}M_0,\qquad -\cD_t=M_t^{-1}+\cS[M_t].
\end{equation}
Then $M_t$ satisfies \eqref{eq:block-MDE} at $\cD_t$, and the pair obeys
\[
    -\frac{\de }{\de t}\cD_t=\frac{1}{2}\cD_t+\cS[M_t],\quad \frac{\de}{\de t}M_t=\frac{1}{2}M_t.
\]
\end{lemma}

\begin{proof}
The definition of $\cD_t$ is equivalent to $-M_t^{-1}=\cD_t+\cS[M_t]$.  Since $\cS$ is linear,
\[
 -\dot{\cD}_t
 =-\frac12M_t^{-1}+\frac12\cS[M_t]
 =\frac12\cD_t+\cS[M_t].
\]
The identity $\dot M_t=M_t/2$ is immediate. This completes the proof.
\end{proof}

Let $H_T>0$ be the terminal time with terminal source $(Z_*,\beta_*)=(wI_n,w)$ and terminal deterministic equivalent matrix
\[
    M_*=\diag((-w-w\widetilde m_0(z)\Sigma)^{-1},w\widetilde{m}_0(z)I_N).
\]
We parameterize the characteristic by its backward distance from the target, with $0\leq h\leq H_T$. Thus, $h=0$ is the terminal point, while larger values of $h$ lie farther inside the easy initialization regime. By Lemma \ref{lem:block-M-transport} and the deformed MP equation \eqref{eq:MP}, we have
\begin{align}
 M_{h}&=e^{-h/2}M_*,\qquad  Z_{h}=e^{h/2}wI_n+(e^{h/2}-e^{-h/2})w\widetilde{m}_0(z)\Sigma,
 \label{eq:overview-characteristic}\\
 \beta_{h}&=e^{h/2}w+(e^{h/2}-e^{-h/2})(-w-(w\widetilde m_0(z))^{-1}).
 \notag
\end{align}
Thus, every generalized source used in the proof is the explicit source
in \eqref{eq:overview-characteristic}. For $t\geq 0$, introduce the covariance-preserving Ornstein--Uhlenbeck flow
\begin{equation}
 \de X_t=-\frac12X_t\,\de t+\Sigma^{1/2}\de B_t,
 \label{eq:overview-OU}
\end{equation}
where $B_t$ is a real standard $n\times N$ matrix Brownian motion,
independent of $X_0:=X$. For fixed $t$, $X_t$ is equal in law to
\[
 e^{-t/2}X_0+\sqrt{1-e^{-t}}\Sigma^{1/2}G_0.
\]
Here, $G_0\in\R^{n\times N}$ has independent standard Gaussian entries. In the remainder of the paper, we use the shorthand
\begin{equation}\label{eq:general-resolvent}
    \cR_{h,t}:=\cR(X_t,Z_{h},\beta_{h }),\quad \cD_h:=\diag(Z_h,\beta_hI_N),\quad \cX_t:=\begin{pmatrix}
        0 & N^{-1/2}X_t\\
        N^{-1/2}X_t^* & 0
    \end{pmatrix}.
\end{equation}

For $h\geq 0$ and $z=E+i\eta$ in a regular domain, define the deterministic control parameters
\[
    \Psi:=\Psi(z)=\sqrt{\frac{\Im\widetilde{m}_0(z)}{N\eta}}+\frac{1}{N\eta},\quad \Psi_h:=e^{-h/2}\Psi.
\]

Our main task is to prove the following theorem.
\begin{theorem}\label{thm:generalized-local-law}
Under Assumptions \ref{ass:A1} and \ref{ass:A2}, let $\bD$ be a regular domain. Uniformly over all $z\in\bD$, the corresponding square roots $w=\sqrt{z}\in\C^+$, and all deterministic unit vectors $\u,\bv\in\C^{n+N}$, we have
\[
    |\u^*(\cR_{0,0}-M_*)\bv|\prec\Psi_0=\Psi(z).
\]
\end{theorem}

\subsection{Bootstraps and zig-zag}\label{subsec:proof-outline}
We now explain the proof strategy. To control a bilinear form of the form $\u^*(\cR_{h,t}-M_h)\bv$ for $\u,\bv\in\C^{n+N}$, it suffices to consider $\u,\bv\in\R^{n+N}$, since we can decompose $\u,\bv$ into their entrywise real and imaginary parts. Moreover, because $\cR_{h,t},M_h$ are complex symmetric matrices, for any $\u,\bv\in\R^{n+N}$ we have
\[
    \frac{1}{4}[(\u+\bv)^*(\cR_{h,t}-M_h)(\u+\bv)-(\u-\bv)^*(\cR_{h,t}-M_h)(\u-\bv)]=\u^*(\cR_{h,t}-M_h)\bv.
\]
Thus, it suffices to consider $\u=\bv\in\R^{n+N}$. For a regular domain $\bD$, we first restrict attention to the smaller domain $\hat\bD:=\{z\in\bD:N\eta\Im\widetilde m_0(z)\geq N^{3C_0\delta}\}$, where we fix constants $C_0 >100$ and $0 < \delta <\tau/(10 C_0)$. For each $z\in\hat\bD$, we perform a constant number of bootstraps on the spectral scale, descending from height one to height $\eta$ in multiplicative increments of $N^{-\delta}$. Each bootstrap iteration uses the so-called ``zig--zag'' strategy.

Both the ``zig'' and ``zag'' comparisons use the covariance-preserving Ornstein--Uhlenbeck flow
\begin{equation}
 \de X_s=-\frac12X_s\,\de s+\Sigma^{1/2}\de B_s,
 \label{eq:bootstrap-OU}
\end{equation}
where $B_s$ is a real standard $n\times N$ matrix Brownian motion
independent of $X$. In the zig step, the source moves forward along
\eqref{eq:overview-characteristic} together with this flow. The
order-one It\^o correction is canceled by the motion of $M$, and the
remaining martingale and drift terms are estimated in
Proposition~\ref{prop:full-block-zig}.  In the zag step, the new source is
held fixed while the OU smoothing is reversed. Removing one column
then leaves only complete linear and centered quadratic forms of that
column. We control the errors using both Assumption~\ref{ass:A2} and a coarse estimate obtained from the previous bootstrap iteration.

Each grid point at OU time zero represents the original ensemble. Figure~\ref{fig:proof-strategy} displays three consecutive spectral bootstraps. Within each panel, the current zig provides an endpoint estimate at OU time $t_k$, while the estimate already established on the preceding spectral domain supplies the coarse imaginary-part and minor bounds needed for the current zag. The output at $s=0$ then becomes the input to the next zig. After reaching $h_K=0$, the enlarged spectral domain becomes the coarse-input domain for the next complete source sweep.

Thus, the zig moves both the ensemble and the source, while the zag returns to the original ensemble at the new fixed source. Proposition~\ref{prop:full-block-zig} and Proposition~\ref{prop:full-block-zag} provide the two comparison steps. Iterating the spectral bootstrap gives the estimate on $\hat{\bD}$, and a final continuity argument extends it to all of $\bD$.

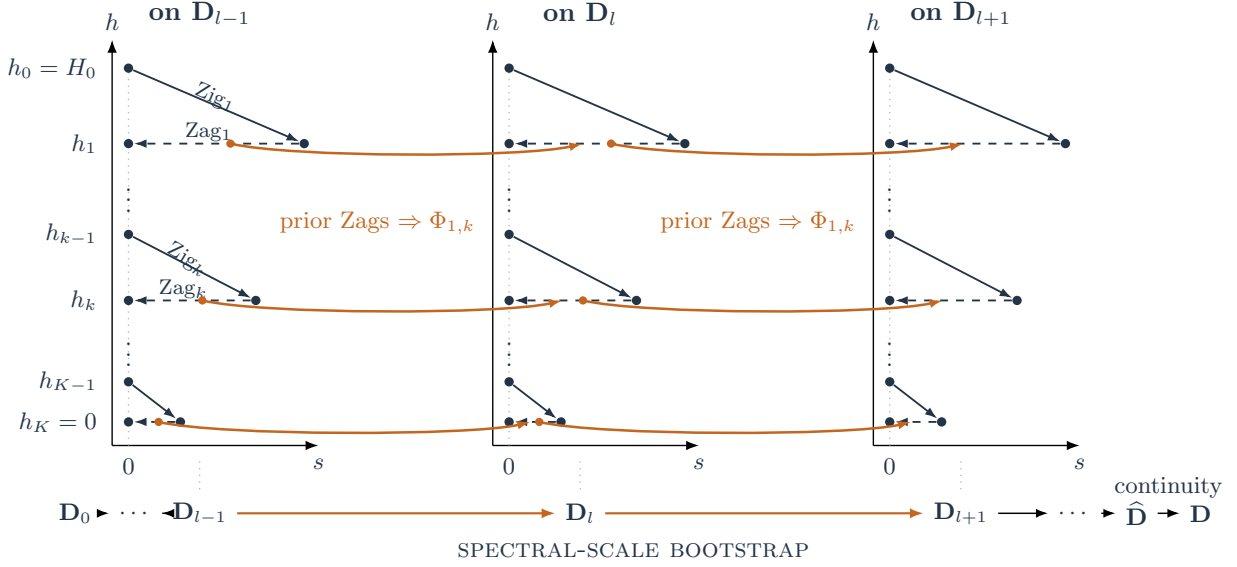
\begin{figure}[t]
\centering
\begin{tikzpicture}[
    x=.94cm,y=.78cm,>=Latex,
    dot/.style={circle,fill=proofink,inner sep=0pt,minimum size=3.5pt},
    zig/.style={-{Latex[length=1.7mm,width=1.15mm]},draw=proofink,line width=.70pt},
    zag/.style={-{Latex[length=1.7mm,width=1.15mm]},draw=proofink,line width=.70pt,dashed},
    prior/.style={-{Latex[length=1.8mm,width=1.2mm]},draw=prooforange,line width=.95pt},
    guide/.style={draw=proofink!35,line width=.5pt,dotted},
    every node/.style={font=\footnotesize,text=proofink}
]

\newcommand{\sweeppanel}[3]{%
  \begin{scope}[shift={(#1,0)},name prefix=#2-]
    \node[font=\small\bfseries] at (1.55,7.62) {on $#3$};
    \draw[-{Latex[length=1.7mm]},line width=.52pt] (.32,.30) -- (3.22,.30);
    \draw[-{Latex[length=1.7mm]},line width=.52pt] (.32,.30) -- (.32,7.18);
    \node[above] at (.32,7.18) {$h$};
    \node[below=1pt] at (3.22,.30) {$s$};
    \draw[guide] (.55,.30) -- (.55,6.82);
    \node[below=1pt] at (.55,.30) {$0$};

    \node[dot] (A0) at (.55,6.70) {};
    \node[dot] (A1) at (.55,5.42) {};
    \node[dot] (B1) at (3.02,5.42) {};
    \draw[zig] (A0) -- (B1);
    \draw[zag] (B1) -- (A1);

    \node at (.55,4.60) {$\vdots$};
    \node[dot] (Akm) at (.55,3.88) {};
    \node[dot] (Ak) at (.55,2.76) {};
    \node[dot] (Bk) at (2.34,2.76) {};
    \draw[zig] (Akm) -- (Bk);
    \draw[zag] (Bk) -- (Ak);

    \node at (.55,1.98) {$\vdots$};
    \node[dot] (AKm) at (.55,1.38) {};
    \node[dot] (AK) at (.55,.70) {};
    \node[dot] (BK) at (1.28,.70) {};
    \draw[zig] (AKm) -- (BK);
    \draw[zag] (BK) -- (AK);

    \coordinate (zagoneout) at ($(A1)!.58!(B1)$);
    \coordinate (zagonein) at ($(A1)!.42!(B1)$);
    \coordinate (zagkout) at ($(Ak)!.58!(Bk)$);
    \coordinate (zagkin) at ($(Ak)!.42!(Bk)$);
    \coordinate (zagKout) at ($(AK)!.58!(BK)$);
    \coordinate (zagKin) at ($(AK)!.42!(BK)$);
    \coordinate (base) at (1.55,.03);
  \end{scope}%
}

\sweeppanel{.10}{L}{\bD_{l-1}}
\sweeppanel{5.45}{M}{\bD_l}
\sweeppanel{10.80}{R}{\bD_{l+1}}

\node[anchor=east] at (.35,6.70) {$h_0=H_0$};
\node[anchor=east] at (.35,5.42) {$h_1$};
\node[anchor=east] at (.35,3.88) {$h_{k-1}$};
\node[anchor=east] at (.35,2.76) {$h_k$};
\node[anchor=east] at (.35,1.38) {$h_{K-1}$};
\node[anchor=east] at (.35,.70) {$h_K=0$};

\node[font=\scriptsize,rotate=-23] at (1.87,6.20) {Zig$_1$};
\node[font=\scriptsize] at (1.78,5.62) {Zag$_1$};
\node[font=\scriptsize,rotate=-31] at (1.48,3.46) {Zig$_k$};
\node[font=\scriptsize] at (1.42,2.96) {Zag$_k$};

\foreach \p in {L-zagoneout,L-zagkout,L-zagKout,
                 M-zagoneout,M-zagkout,M-zagKout}{
  \node[circle,fill=prooforange,inner sep=0pt,minimum size=3.1pt] at (\p) {};
}
\foreach \a/\b in {L-zagoneout/M-zagonein,L-zagkout/M-zagkin,L-zagKout/M-zagKin,
                    M-zagoneout/R-zagonein,M-zagkout/R-zagkin,M-zagKout/R-zagKin}{
  \draw[prior] (\a) .. controls +(.85,-.24) and +(-.85,-.24) .. (\b);
}
\node[text=prooforange,fill=white,inner sep=1.2pt] at (4.15,4.08)
  {prior Zags $\Rightarrow\Phi_{1,k}$};
\node[text=prooforange,fill=white,inner sep=1.2pt] at (9.50,4.08)
  {prior Zags $\Rightarrow\Phi_{1,k}$};

\node (D0) at (-.10,-.85) {$\bD_0$};
\node (dotsL) at (.75,-.85) {$\cdots$};
\node (Dl) at (1.65,-.85) {$\bD_{l-1}$};
\node (Dlp) at (7.00,-.85) {$\bD_l$};
\node (Dlpp) at (12.35,-.85) {$\bD_{l+1}$};
\node (dotsR) at (13.95,-.85) {$\cdots$};
\node (Dh) at (14.82,-.85) {$\widehat{\bD}$};
\node (D) at (15.72,-.85) {$\bD$};
\draw[-{Latex[length=1.6mm]},line width=.55pt] (D0) -- (dotsL);
\draw[-{Latex[length=1.6mm]},line width=.55pt] (dotsL) -- (Dl);
\draw[-{Latex[length=1.8mm]},line width=.82pt,prooforange] (Dl) -- (Dlp);
\draw[-{Latex[length=1.8mm]},line width=.82pt,prooforange] (Dlp) -- (Dlpp);
\draw[-{Latex[length=1.6mm]},line width=.55pt] (Dlpp) -- (dotsR);
\draw[-{Latex[length=1.6mm]},line width=.55pt] (dotsR) -- (Dh);
\draw[-{Latex[length=1.6mm]},line width=.55pt] (Dh) --
  node[above=2pt] {continuity} (D);
\draw[guide] (L-base) -- (Dl.north);
\draw[guide] (M-base) -- (Dlp.north);
\draw[guide] (R-base) -- (Dlpp.north);
\node[font=\small\scshape] at (7.75,-1.48) {spectral-scale bootstrap};
\end{tikzpicture}
\caption{Three consecutive spectral bootstraps, drawn using the sparse coordinate convention of \cite[Figure~1]{erdos2025lecture}. Within each panel, a solid diagonal arrow represents a Zig step, while a dashed horizontal arrow represents the ensuing Zag step. At every displayed source level, an orange curve begins at the Zag estimate already established on $\bD_j$ and enters the corresponding Zag step on $\bD_{j+1}$. Through Lemma~\ref{lem:bootstrap-in-zag}(b), the Herglotz comparison, and the minor and denominator estimates, this provides the coarse input $\Phi_{1,k}$. The diagonal Zig step in the new panel provides $\Phi_{2,k}$, and Lemma~\ref{lem:bootstrap-in-zag}(c) then closes the Zag step. The bottom line records the outer spectral bootstrap and the final continuity argument. Each Zig arrow represents an application of Proposition~\ref{prop:full-block-zig}, conditional on its normalized-trace input.}
\label{fig:proof-strategy}
\end{figure}

\section{The zig estimate}\label{sec:zig}
The zig proof starts from an exact stochastic equation for the moving
resolvent.  After stopping the process at the anisotropic envelope, the two Ward
identities control the linear drift, nonlinear drift, and quadratic
variation.  A maximal martingale estimate and a first-exit argument then remove the stop.

\FloatBarrier
We begin by proving some basic estimates.
\begin{lemma}[Basic estimates]\label{lem:basic-estimates}
Under Assumption \ref{ass:A1}, let $\bD$ be a regular domain. Then, for all $h,t\geq 0$ and $z\in\bD$, the following hold for some constants $C,c>0$:
\begin{enumerate}[label=(\alph*)]
    \item For $\cW=\diag(\Sigma,I_N)\in\R^{(n+N)\times(n+N)}$, we have, in the positive-semidefinite ordering,
    \[
        \diag(\Im Z_h,\Im\beta_h I_N)\geq ce^{h/2}(\eta+(1-e^{-h})\Im\widetilde m_0(z))\cW+ce^{h/2}\eta I_{n+N}.
    \]
    \item We have
    \[
        \|M_*\|_\op\leq C,\quad \|\Im M_*\|_\op\leq C\Im\widetilde m_0(z),\quad \|\cR_{h,t}\|_\op\leq Ce^{-h/2}\eta^{-1}.
    \]
    \item For all $\u\in\R^{n+N}$, we have
    \[
        \|\cR_{h,t}\u\|_2\leq C\sqrt{\frac{e^{-h/2}\u^*\Im\cR_{h,t}\u}{\eta}},\quad \|\cR_{h,t}\|_F\leq C\sqrt{\frac{e^{-h/2}\Tr\Im\cR_{h,t}}{\eta}},
    \]
    \[
        \|\cW^{1/2}\cR_{h,t}\u\|_2\leq C\sqrt{\frac{e^{-h/2}\u^*\Im\cR_{h,t}\u}{\eta+(1-e^{-h})\Im\widetilde m_0(z)}},\quad \|\cW^{1/2}\cR_{h,t}\|_F\leq C\sqrt{\frac{e^{-h/2}\Tr\Im\cR_{h,t}}{\eta+(1-e^{-h})\Im\widetilde m_0(z)}}.
    \]
\end{enumerate}
\end{lemma}
\begin{proof}
For (a), we first use (\ref{eq:MP}) to rewrite $\beta_h$ as
\[
    \beta_h=e^{h/2}w+(e^{h/2}-e^{-h/2})N^{-1}\Tr\Sigma(-w-w\widetilde m_0(z)\Sigma)^{-1}.
\]
Thus, taking imaginary parts gives, in the positive-definite ordering,
\begin{align*}
    \Im Z_{h}&\geq e^{h/2}\min_{\substack{\sigma\in\spec(\Sigma)\\\sigma>0}}\{\Im{w}/\sigma+(1-e^{-h})\Im{w\widetilde{m}_0(z)}\}\Sigma,\\
    \Im\beta_{h}I_N&\geq e^{h/2}(\Im{w}+(1-e^{-h})N^{-1}\Im\Tr\Sigma(-w-w\widetilde m_0(z)\Sigma)^{-1}))I_N.
\end{align*}
We have
\begin{align*}
    \Im{w\widetilde m_0(z)}&=\int\frac{\Im{w(\lambda-\bar z)}}{|\lambda-z|^2}\;d\widetilde\mu_0(\lambda)=\int\frac{(\lambda+|w|^2)\Im w}{|\lambda-z|^2}\;d\widetilde\mu_0(\lambda)\\
    &\geq C\int\frac{\eta}{|\lambda-z|^2}\;d\widetilde\mu_0(\lambda)=C\Im\widetilde m_0(z).
\end{align*}
Here we used $\lambda\geq 0$, $|w|\geq c$, $z=w^2$, and $\Im w\geq C\eta$. Similarly,
\begin{align*}
    &N^{-1}\Im\Tr\Sigma(-w-w\widetilde m_0(z)\Sigma)^{-1})=\frac{1}{N}\sum_{\sigma\in\spec(\Sigma)}\frac{\sigma(\Im{w}+\sigma\Im{w\widetilde m_0(z)})}{|w+w\sigma\widetilde m_0(z)|^2}\geq C(\eta+\Im\widetilde m_0(z)),
\end{align*}
where the sum over $\sigma\in\spec(\Sigma)$ is with multiplicity.
Here, we used Assumption \ref{ass:A1} and the definition of a regular domain. Hence, for $\cD_h=\diag(Z_h,\beta_h I_N)$, the positive-semidefinite ordering gives
\begin{align*}
    \Im\cD_h&=\diag(\Im Z_{h},\Im\beta_{h}I_N)\\
    &\geq Ce^{h/2}\min\{\eta/\|\Sigma\|_\op+(1-e^{-h})\Im{\widetilde m_0(z)},\eta+(1-e^{-h})\Im\widetilde m_0(z)\}\cW\\
    &\geq C'e^{h/2}\underbrace{(\eta+(1-e^{-h})\Im\widetilde m_0(z))}_{\eta_h}\cW .
\end{align*}
Here, we again used Assumption \ref{ass:A1}. We also have the trivial bound
\[
    \Im\cD_h=\diag(\Im Z_{h},\Im\beta_{h})\geq Ce^{h/2}\eta I_{n+N} . 
\]
Combining the two bounds proves (a). For (b), the first bound follows directly from the definition of a regular domain. For the second bound, the same arguments as above give
\begin{align*}
    \|\Im M_*\|_\op&=\max\{\|\Im(-w-w\widetilde m_0(z)\Sigma)^{-1}\|_\op,|\Im w\widetilde m_0(z)|\}\\
    &\leq C\max\{\eta+\Im \widetilde m_0(z),\Im\widetilde m_0(z)\}\leq C\Im\widetilde m_0(z).
\end{align*}
Finally, for the third bound, we use $\|\cR_{h,t}\|_\op\leq \|(\Im\cD_h)^{-1}\|_\op\leq Ce^{-h/2}\eta^{-1}$.

For (c), the identity $A^{-1}-B^{-1}=A^{-1}(B-A)B^{-1}$ gives
\[
    \Im\cR_{h,t}=(\cR_{h,t}-\cR_{h,t}^*)/2i=\cR_{h,t}\Im\cD_h\cR_{h,t}^*,\quad \cD_h=\diag(Z_{h},\beta_{h}).
\]
Since $\cR_{h,t}$ is complex symmetric and $\u$ is real, we have $\|\cR_{h,t}^*\u\|_2=\|\cR_{h,t}\u\|_2$. Thus, using the lower bound from (a), we can write
\[
    \|\cR_{h,t}\u\|_2^2=\|\cR_{h,t}^*\u\|_2^2=\frac{(e^{h/2}\eta)\u^*\cR_{h,t}\cR_{h,t}^*\u}{e^{h/2}\eta}\leq C\frac{e^{-h/2}\u^*\Im\cR_{h,t}\u}{\eta}.
\]
Similarly,
\[
    \|\cW^{1/2}\cR_{h,t}\u\|_2^2=\frac{(e^{h/2}\eta_h)\u^*\cR_{h,t}\cW\cR_{h,t}^*\u}{e^{h/2}\eta_h}\leq C\frac{e^{-h/2}\u^*\Im\cR_{h,t}\u}{\eta_h}.
\]
The corresponding arguments for $\|\cR_{h,t}\|_F$ and $\|\cW^{1/2}\cR_{h,t}\|_F$ are identical and are omitted. This completes the proof.
\end{proof}

The next lemma records the Ward estimates used throughout the characteristic. For $z=E+i\eta\in\C^+$, define 
\[
    \Lambda_h:=\Lambda_h(z)=\frac{e^{-h}\Im\widetilde m_0(z)}{\eta+(1-e^{-h})\Im\widetilde m_0(z)}.
\]

\begin{lemma}[Ward estimates]\label{lem:ward-estimate}
Under Assumptions \ref{ass:A1} and \ref{ass:A2}, let $\bD$ be a regular domain. Fix any constants $0<2\beta<\beta'$. For any $h,t\geq 0$ and any deterministic unit vector $\u\in\R^{n+N}$, define the event
\[
    \Omega:=\{\omega:|\u^*(\cR_{h,t}-M_{h})\u|\leq N^\beta\Psi_h\}.
\]
Then, for all $z\in\bD$ satisfying $N\eta\Im\widetilde m_0(z)\geq N^{\beta'}$, the following estimates hold for some constant $C>0$:
\begin{equation}
    N^{-1/2}\|\cR_{h,t}\u\|_2\1_\Omega\leq CN^{\beta/2}\Psi_h,
    \quad  \|\cW^{1/2}\cR_{h,t}\u\|_2\1_\Omega\leq C\sqrt{\Lambda_h}.
\end{equation}
Here, $\cW=\diag(\Sigma,I_N)$.
\end{lemma}
\begin{proof}
We first establish the unweighted Ward estimates. Applying Lemma \ref{lem:basic-estimates} gives
\begin{align*}
    N^{-1}\|\cR_{h,t}\u\|_2^2\1_\Omega&\leq Ce^{-h/2}\frac{\u^*\Im\cR_{h,t}\u}{N\eta}\1_\Omega\leq Ce^{-h}\frac{\|\Im M_{*}\|_\op}{N\eta}+C\frac{e^{-h/2}}{N\eta}N^\beta\Psi_h\\
    &\leq Ce^{-h}\frac{\Im\widetilde m_0(z)}{N\eta}+N^\beta\Psi_h^2\leq CN^\beta\Psi_h^2.
\end{align*}
For the weighted Ward estimates, Lemma \ref{lem:basic-estimates} similarly gives
\begin{align*}
    \|\cW^{1/2}\cR_{h,t}\u\|_2^2\1_\Omega&\leq Ce^{-h/2}\frac{\u^*\Im\cR_{h,t}\u}{\eta_h}\1_\Omega\leq Ce^{-h}\frac{\|\Im M_*\|_\op}{\eta_h}+Ce^{-h/2}\frac{1}{\eta_h}N^\beta\Psi_h\\
    &\leq Ce^{-h}\frac{\Im\widetilde m_0(z)}{\eta_h}+Ce^{-h}\frac{\Im\widetilde m_0(z)}{\eta_h}\leq C\frac{e^{-h}\Im\widetilde m_0(z)}{\eta+(1-e^{-h})\Im\widetilde m_0(z)}.
\end{align*}
In the third step, we use $N\eta\Im\widetilde m_0(z)\geq N^{\beta'}$ and $2\beta\leq \beta'$ to deduce
\[
    N^\beta\Psi_h\leq e^{-h/2}\sqrt{\frac{N^{2\beta}\Im\widetilde m_0(z)}{N\eta}}+e^{-h/2}\frac{N^\beta}{N\eta}\leq Ce^{-h/2}\Im\widetilde m_0(z).
\]
This completes the proof.
\end{proof}

The following lemma initializes the zig estimate at a remote point.
\begin{lemma}\label{lem:zig-initialization}
Suppose Assumptions \ref{ass:A1} and \ref{ass:A2} hold, $\bD$ is a regular domain, and $H_0\geq 2\log(N)$. Then, uniformly over all $z\in\bD$ and all deterministic unit vectors $\u\in\R^{n+N}$, we have
\[
    |\u^*(\cR_{H_0,0}-M_{H_0})\u|\prec \Psi_{H_0}.
\]
\end{lemma}
\begin{proof}
At the remote source $H_0\geq 2\log(N)$, $-\cD_{H_0}^{-1}$ is close to $M_{H_0}$. Moreover, because $-\cD_{H_0}$ has large magnitude, $\cR_{H_0,0}$ is close to $-\cD_{H_0}^{-1}$. By the explicit source \eqref{eq:overview-characteristic} and $w^2=z$, we have
\[
    -\cD_{H_0}=\diag(-Z_{H_0},-\beta_{H_0}I_N)=M_{H_0}^{-1}+e^{-H_0/2}\underbrace{\diag(w\widetilde m_0(z)\Sigma, (-w-(w\widetilde m_0(z))^{-1})I_N)}_{\cE}.
\]
In particular, the definition of a regular domain and Lemma \ref{lem:basic-estimates} give $\|\cE\|_\op\leq C$ and $\|M_{H_0}\|_\op\leq Ce^{-H_0/2} \leq CN^{-1}$. It then follows from Weyl's inequality that $\|\cD_{H_0}^{-1}\|_\op\leq Ce^{-H_0/2}$. Using the identity $A^{-1}-B^{-1}=A^{-1}(B-A)B^{-1}$, we obtain
\[
    \|-\cD_{H_0}^{-1}-M_{H_0}\|_\op=e^{-H_0/2}\|\cD_{H_0}^{-1}\cE M_{H_0}\|_\op\leq Ce^{-3H_0/2},
\]
Here, we also applied $\|M_{H_0}\|_\op\leq Ce^{-H_0/2}$. On the other hand,
\[
    \|\cR_{H_0,0}-(-\cD_{H_0}^{-1})\|_\op=\|\cR_{H_0,0}\cX\cD_{H_0}^{-1}\|_\op\prec e^{-H_0}\eta^{-1},\quad \cX=\begin{pmatrix}
        0 & N^{-1/2}X\\
        N^{-1/2}X^* & 0
    \end{pmatrix}.
\]
Here we used \cite[Lemma~3.8(b)]{fan2026anisotropic} to deduce $\|\cX\|_\op\prec 1$ and Lemma \ref{lem:basic-estimates} to deduce $\|\cR_{H_0,0}\|_\op\leq e^{-H_0/2}\eta^{-1}$. Combining these estimates gives
\begin{align*}
    \|\cR_{H_0,0}-M_{H_0}\|_\op&\leq \|\cR_{H_0,0}-(-\cD_{H_0}^{-1})\|_\op+\|-\cD_{H_0}^{-1}-M_{H_0}\|_\op\\
    &\prec e^{-H_0}\eta^{-1}+e^{-3H_0/2}\prec e^{-H_0/2}((N\eta)^{-1}+N^{-2})\prec \Psi_{H_0}.
\end{align*}
In particular, this implies
\[
    |\u^*(\cR_{H_0,0}-M_{H_0})\u|\prec \Psi_{H_0}
\]
uniformly over all deterministic unit vectors $\u\in\R^{n+N}$.
\end{proof}

With the initialization in place, we next derive the exact resolvent evolution along the characteristic.
\begin{lemma}[Exact block zig equation]\label{lem:exact-block-zig}
For any $0\leq t\leq h$, abbreviate $\cR_t:=\cR_{h-t,t},M_t=M_{h-t}$, and $\cD_t:=\cD_{h-t}$. Then
\begin{equation}
 \de(\cR_t-M_t)
 =\left[
 \frac12(\cR_t-M_t)
 +\cR_t\cS[\cR_t-M_t]\cR_t
 \right]\de t
 +\frac1N\cE(\cR_t)\de t
 -\cR_t\cW^{1/2}\de\cB_t\cW^{1/2}\cR_t.                  \label{eq:exact-block-zig}
\end{equation}
Here,
\[
    \de\cB_t=
    \begin{pmatrix}0&N^{-1/2}\de B_t\\
    N^{-1/2}\de B_t^*&0
    \end{pmatrix},\quad \cE(\cR_t)
 =\cR_t\cW
 \begin{pmatrix}
  0&[\cR_t]_{12}\\
  [\cR_t]_{21}&0
 \end{pmatrix}
 \cW\cR_t.
\]
\end{lemma}
\begin{proof}
Write $\cL_t=\cR_t^{-1}=-\cD_{t}+\cX_t$.  Since
$\de\cX_t=-\frac{1}{2}\cX_t\de t+\cW^{1/2}\de\cB_t\cW^{1/2}$, It\^o's formula and $\cR_t\cX_t\cR_t=\cR_t+
\cR_t\cD_{t}\cR_t$ give
\begin{align*}
    \de\cR_t
    &=\cR_t\dot\cD_{t}\cR_t\de t
      +\tfrac12\cR_t\cX_t\cR_t\de t
      -\cR_t\cW^{1/2}\de\cB_t\cW^{1/2}\cR_t\\
    &\quad+\cR_t\cW^{1/2}\de\cB_t\cW^{1/2}\cR_t\cW^{1/2}\de\cB_t\cW^{1/2}\cR_t\\
    &=\tfrac12\cR_t\de t
      +\cR_t(\dot\cD_{t}+\tfrac12\cD_{t})\cR_t\de t
      -\cR_t\cW^{1/2}\de\cB_t\cW^{1/2}\cR_t\\
    &\quad+\cR_t\cW^{1/2}\de\cB_t\cW^{1/2}\cR_t
      \cW^{1/2}\de\cB_t\cW^{1/2}\cR_t.
\end{align*}
The Brownian contractions are
\[
    \cW^{1/2}\de\cB_t\cW^{1/2}\cR_t\cW^{1/2}\de\cB_t\cW^{1/2}
    =\cS[\cR_t]\de t+
    \frac1N\begin{pmatrix}
    0&\Sigma[\cR_t]_{12}\\
    [\cR_t]_{21}\Sigma&0
    \end{pmatrix}\de t.
\]
Set $\de C=N^{-1/2}\Sigma^{1/2}\de B_t$. The off-diagonal term is the contraction of real Brownian motion. Indeed, $\cR_t^\top=\cR_t$ and
\begin{align*}
 \E(\de C[\cR_t]_{21}\,\de C)_{aj}
 &=\sum_{k=1}^N\sum_{l=1}^n
   ([\cR_t]_{21})_{kl}
   \E(\de C_{ak}\de C_{l j})\\
 &=\frac1N\sum_{l=1}^n
   \Sigma_{al}([\cR_t]_{21})_{jl}\,\de t
 =\frac1N(\Sigma[\cR_t]_{12})_{aj}\de t.
\end{align*}
The lower-left contraction is its transpose, while the diagonal
contractions are $N^{-1}\Tr([\cR_t]_{22})\Sigma\,\de t$ and $N^{-1}\Tr(\Sigma[\cR_t]_{11})I_N\de t$. Substituting
$\dot{\cD}_t+\cD_t/2=-\cS[M_t]$ and subtracting
$\de M_t=\dot M_t\de t=M_t/2\de t$ proves \eqref{eq:exact-block-zig}.
\end{proof}

The stochastic term will be controlled uniformly in time by the next
moving-envelope estimate.

\begin{lemma}[Moving envelope estimate]
\label{lem:maximal-envelope-martingale}
For any $0\leq h'\leq h$ and $0\leq t\leq h-h'$, define
$b_t=e^{-t/2}(\Lambda_{h-t}/N)^{1/2}$. Let $\xi\leq h-h'$ be a
stopping time and let $\m_t$ be a
continuous complex local martingale with $\m_0=0$, whose
quadratic variation satisfies
\[
 \de\langle\m_t,\bar{\m_t}\rangle_t
 \le \frac{C\Lambda_{h-t}^2}{N}\,\de t,\qquad t\le\xi.
\]
Here, $C>0$ is a constant. Define the stopped martingale
\[
 \widehat\m_t=\int_0^{t\wedge\xi}e^{-s/2}\,\de\m_s.
\]
Then
\begin{equation}
  |\widehat\m_\xi|\prec b_\xi\left(\int_0^{h-h'}\Lambda_{h-t}\de t\right)^{1/2}.
\end{equation}
Here, the quadratic variation of a complex martingale means the sum of the
quadratic variations of its real and imaginary parts.
\end{lemma}

\begin{proof}
Define
\[
    \n_t:=\int_0^{t\wedge\xi}\frac{e^{-s/2}}{b_s}\de\m_s
    =\int_0^t\frac{1}{b_s}\de\widehat{\m}_s.
\]
Then $\n_t$ is a stopped martingale and 
\[
    \de\<\n,\bar{\n}\>_s=\frac{1}{b_s^2}\de\<\widehat{\m},\bar{\widehat{\m}}\>_s
    \leq C\Lambda_{h-s}\de s
\] 
The maximal BDG inequality then gives, for every even moment order $2p$,
\[
    \E\sup_{s\leq h-h'}|\n_s|^{2p}\leq 
    C_p\left(\int_0^{h-h'}\Lambda_{h-t}\de t\right)^{p}.
\]
Since $\de\widehat{\m}_s=b_s\de\n_s$, integration by parts at the random time 
$\xi$ gives
\[
    \widehat{\m}_\xi=b_\xi\n_\xi-\int_0^\xi\n_s\de b_s.
\]
Since $b_s$ is positive and nondecreasing, we have
\[
    \frac{|\widehat{\m}_\xi|}{b_\xi}\leq |\n_\xi|
  +\frac{1}{b_\xi}\int_0^\xi|\n_s|\de b_s\leq \sup_{s\leq\xi}|\n_s|
  +\frac{b_\xi-b_0}{b_\xi}\sup_{s\leq\xi}|\n_s|
  \leq 2\sup_{s\leq h-h'}|\n_s|.
\]
Combining these estimates gives
\[
    \E(|\widehat{\m}_\xi|/b_\xi)^{2p}\leq C_p
    \left(\int_0^{h-h'}\Lambda_{h-t}\de t\right)^{p}.
\]
The claim follows from Markov's inequality.
\end{proof}

We next prove the zig estimate for the quadratic form. To this end, we first establish a sufficiently sharp zig estimate for the normalized trace; the following proposition summarizes this estimate. Its proof follows a static fluctuation-averaging argument adapted from \cite{fan2026anisotropic} and is deferred to Section \ref{sec:averaged}.

\begin{proposition}[Averaged zig]\label{prop:averaged zig}
Suppose Assumptions \ref{ass:A1} and \ref{ass:A2} hold and $\bD$ is a regular domain. Then, for any $0\leq h'\leq h\leq 5\log(N)$, uniformly over all $z\in\bD$, we have
\[
    \sup_{0\leq t\leq h-h'}e^{(h-t)/2}(N\eta)(N^{-1}|\Tr\Sigma(\cR_{h-t,t}-M_{h-t})_{11}|
    +|N^{-1}\Tr(\cR_{h-t,t}-M_{h-t})_{22}|)\prec 1.
\]
\end{proposition}

\begin{proposition}[Zig estimate]
\label{prop:full-block-zig}
Suppose Assumptions \ref{ass:A1} and \ref{ass:A2} hold, $\bD$ is a regular domain, and $0\leq h'\leq h\leq 5\log(N)$. Fix a constant $C_0>10$ and let $\delta\in(0,\tau/(10C_0))$. Assume that, uniformly over all $z\in\bD$ and all deterministic unit vectors $\u\in\R^{N+n}$, we have
\[
    \int_0^{h-h'}\Lambda_{h-t}\de t\leq C\log(N),
    \quad \u^*(\cR_{h,0}-M_h)\u\prec N^{C_0\delta}\Psi_h.
\]
Then, uniformly over all $z\in\bD$ with $N\eta\Im\widetilde m_0(z)\geq N^{3C_0\delta}$ and all deterministic unit vectors $\u\in\R^{N+n}$,
\[
    \u^*(\cR_{h',h-h'}-M_{h'})\u\prec N^{C_0\delta}\Psi_{h'}.
\]
\end{proposition}
\begin{proof}
We first estimate the trace errors, drift, and martingale under a stopping
rule.  A first-exit argument then removes the stop. By Lemma 
\ref{lem:exact-block-zig}, we have
\begin{align*}
  \de(\u^*(\cR_{h-t,t}-M_{h-t})\u)&=\frac{1}{2}\u^*(\cR_{h-t,t}-M_{h-t})\u\de t
  +\u^*\cR_{h-t,t}\cS[\cR_{h-t,t}-M_{h-t}]\cR_{h-t,t}\u\de t\\
  &+\frac{1}{N}\u^*\cE(\cR_{h-t,t})\u\de t
  -\u^*\cR_{h-t,t}\cW^{1/2}\de\cB_t\cW^{1/2}\cR_{h-t,t}\u.
\end{align*}
Multiplying both sides by $e^{-t/2}$ and rearranging gives
\begin{align*}
  \de(e^{-t/2}\underbrace{\u^*(\cR_{h-t,t}-M_{h-t})\u}_{Y_t})&=\underbrace{e^{-t/2}
  \u^*\cR_{h-t,t}\cS[\cR_{h-t,t}-M_{h-t}]\cR_{h-t,t}\u}_{\mathbf{I}_t}\de t
  +\underbrace{\frac{e^{-t/2}}{N}\u^*\cE(\cR_{h-t,t})\u}_{\mathbf{II}_t}\de t\\
  &-e^{-t/2}\underbrace{\u^*\cR_{h-t,t}\cW^{1/2}\de\cB_t\cW^{1/2}
  \cR_{h-t,t}\u}_{\de \m_t}.
\end{align*}

Fix $0<\eps<C_0\delta/10$, define the stopping time
\[
 \xi=\inf\left\{t\in[0,h-h']:
   |\u^*(\cR_{h-t,t}-M_{h-t})\u|\ge N^{\eps+C_0\delta}\Psi_{h-t}\right\}\wedge (h-h').
\]
It follows that $2(C_0\delta+\eps)\leq (2+1/5)C_0\delta<3C_0\delta$, so the assumptions of Lemma \ref{lem:ward-estimate} are satisfied for the stopped process. By continuity and the definition of $\xi$, for all $t\leq\xi$, we have
\[
    |\u^*(\cR_{h-t,t}-M_{h-t})\u|\leq N^{\eps+C_0\delta}\Psi_{h-t}.
\]
Integrating the stochastic differential equation up to time $\xi$ gives
\[
    Y_\xi-e^{\xi/2}Y_0=e^{\xi/2}\int_0^\xi\mathbf{I}_t\de t
    +e^{\xi/2}\int_0^\xi\mathbf{II}_t\de t
    -e^{\xi/2}\int_0^\xi e^{-t/2}\de\m_t.
\]
We now bound each term separately. For the first term,
\[
\begin{aligned}
    \mathbf{I}_t
    &=e^{-t/2}\u^*\cR_{h-t,t}\cW^{1/2}
      \begin{bmatrix}
        N^{-1}\Tr(\cR_{h-t,t}-M_{h-t})_{22}I_n &0\\
        0& N^{-1}\Tr\Sigma(\cR_{h-t,t}-M_{h-t})_{11}I_N
      \end{bmatrix}\\
    &\qquad{}\cW^{1/2}\cR_{h-t,t}\u.
\end{aligned}
\]
Thus, we can bound
\begin{align*}
    |\mathbf{I}_t|
    &\leq Ce^{-t/2}\|\cW^{1/2}\cR_{h-t,t}\u\|_2^2\\
    &\quad\times (|N^{-1}\Tr(\cR_{h-t,t}-M_{h-t})_{22}|
    +|N^{-1}\Tr\Sigma(\cR_{h-t,t}-M_{h-t})_{11}|).
\end{align*}
The definition of $\xi$, Lemma \ref{lem:ward-estimate}, and Proposition \ref{prop:averaged zig} imply that, for any $D>0$,
with probability at least $1-CN^{-D}$,
\[
    e^{\xi/2}\left|\int_0^{\xi}\mathbf{I}_t\de t\right|
    \leq CN^{\eps/10+C_0\delta}e^{\xi/2}\int_0^\xi e^{-t/2}\Psi_{h-t}\Lambda_{h-t}\de t
    \leq CN^{\eps/5+C_0\delta}\Psi_{h-\xi}.
\]
Here, the first inequality uses $e^{-(h-t)/2}(N\eta)^{-1}\leq \Psi_{h-t}$, while the last uses $e^{(\xi-t)/2}\Psi_{h-t}=\Psi_{h-\xi}$
and $N^{C_0\delta+\eps/10}\log(N)\leq CN^{C_0\delta+\eps/5}$.

For the second term, we have
\[
    |\mathbf{II}_t|\leq Ce^{-t/2}N^{-1}\|\cR_{h-t,t}\|_\op
    \|\cW^{1/2}\cR_{h-t,t}\u\|_2\|\cW^{1/2}\cR_{h-t,t}\bar{\u}\|_2
    \leq Ce^{-h/2}\frac{\Lambda_{h-t}}{N\eta}.
\]
For the last inequality, we used
$\|\cR_{h-t,t}\|_\op\leq e^{-(h-t)/2}/\eta$ as in the proof of Lemma 
\ref{lem:ward-estimate}. Therefore,
\[
    e^{\xi/2}\left|\int_0^\xi\mathbf{II}_t\de t\right|
    \leq C\Psi_{h-\xi}\int_0^\xi\Lambda_{h-t}\de t
    \leq CN^{C_0\delta}\Psi_{h-\xi}. 
\]
Here, we used $e^{-(h-\xi)/2}/(N\eta)\leq\Psi_{h-\xi}$
and $C\log(N)\leq N^{C_0\delta}$.

Finally, since
\[
    \de\<\m_t,\bar\m_t\>_t\leq (N^{-1/2}\|\cW^{1/2}\cR_{h-t,t}\u\|_2^2)^2\de t\leq \frac{C\Lambda_{h-t}^2}{N}\de t,\quad t\leq\xi . 
\]
Lemma \ref{lem:maximal-envelope-martingale} shows that, for any $D>0$, with probability
at least $1-CN^{-D}$,
\[
    e^{\xi/2}\left|\int_0^\xi e^{-t/2}\de\m_t\right|\leq 
    N^{\eps/10}\sqrt{\frac{\Lambda_{h-\xi}}{N}}\left(\int_0^{h-h'}\Lambda_{h-t}\de t\right)^{1/2}
    \leq CN^{C_0\delta+\eps/10}\Psi_{h-\xi} . 
\]
Here, the last step uses
$\sqrt{\Lambda_{h-\xi}/N}\leq C\Psi_{h-\xi}$.

We remove the stop by a first-exit argument. Suppose, for contradiction,
that $\xi<h-h'$. By continuity at time $\xi$, we must then have
$|Y_\xi|=N^{C_0\delta+\eps}\Psi_{h-\xi}$. Together with the bound
$|e^{\xi/2}Y_0|\leq N^{C_0\delta+\eps/5}\Psi_{h-\xi}$, the preceding estimates show that, for any $D>0$, with probability at least $1-CN^{-D}$, we have
\[
    |Y_\xi|\leq e^{\xi/2}|Y_0|+CN^{C_0\delta+\eps/5}\Psi_{h-\xi}
    \leq CN^{C_0\delta+\eps/5}\Psi_{h-\xi}\leq
    \frac{1}{2}N^{C_0\delta+\eps}\Psi_{h-\xi}<N^{C_0\delta+\eps}\Psi_{h-\xi}
\]
which is a contradiction. Thus, we must have $\xi=h-h'$. In particular, this implies
\[
    |Y_{h-h'}|=|\u^*(\cR_{h',h-h'}-M_{h'})\u|\leq N^{C_0\delta+\eps}\Psi_{h'}.
\]
Since $\eps$ was arbitrary, this completes the proof.
\end{proof}

\section{The zag estimate}\label{sec:zag}
We next prove the estimates used in the zag step. The source is held fixed while the
column Ornstein--Uhlenbeck flow is run from the original ensemble to its
Gaussian smoothing.  The endpoint law is known at the smoothed time, and
the purpose of the zag is to propagate it backward to time zero.

Solving the OU-flow equation explicitly shows that, for each $t\geq0$, $X_t$ has the same distribution as
\begin{equation}\label{eq:explicite OU solution}
    X_t =e^{-t/2}X_0+\sqrt{1-e^{-t}}G.
\end{equation}
Here, the columns of $G$ are independent Gaussian vectors with mean zero and covariance $\Sigma$, and $X_0$ is independent of $G$. In particular, the columns of $X_t$ have mean zero and covariance $\Sigma$, and the resulting model satisfies Assumptions \ref{ass:A1} and \ref{ass:A2} for all $t\geq 0$.

In this section, both $h$ and $t$ will generally be fixed, and we suppress the indices to avoid notational clutter when there is no risk of confusion.

\subsection{Column OU-generator}
The key tool in the zag step is the following OU-generator identity.
\begin{lemma}[OU-generator]\label{lem:OU-generator}
For $t> 0$, define
\[
    X_t:=e^{-t/2}X_0+\sqrt{1-e^{-t}}G\in\R^{n\times N},\quad X_t=\begin{bmatrix}
        | & \cdots &|\\
        \x_1 & \cdots & \x_N\\
        | & \cdots & |
    \end{bmatrix}.
\]
Here, the columns of $G$ are independent Gaussian vectors with mean zero and covariance $\Sigma$, and $X_0$ is independent of $G$. Let $f:\R^{n\times N}\rightarrow\C$ be a twice continuously differentiable function whose first two derivatives have at most polynomial growth. Then
\[
    \frac{\de }{\de t}\E f(X_t)=\sum_{i=1}^N\E\cL_if(X_t),
\]
where
\[
    \cL_i:=\frac{1}{2}\sum_{\alpha,\beta=1}^n\Sigma[\alpha,\beta]\frac{\de^2}{\de\x_i[\alpha]\de\x_i[\beta]}-\frac{1}{2}\sum_{\alpha=1}^n\x_i[\alpha]\frac{\de}{\de\x_i[\alpha]}
\]
is the OU-generator for the $i$-th column.
\end{lemma}
\begin{proof}
We interchange differentiation and expectation and apply Gaussian integration by parts. These operations are justified by the polynomial growth condition and the dominated convergence theorem. We obtain
\begin{align*}
    \frac{\de}{\de t}\E f(X_t)&=\E\left[-\frac{1}{2}e^{-t/2}\Tr X_0^*\nabla f(X_t)+\frac{e^{-t}}{2\sqrt{1-e^{-t}}}\Tr G^*\nabla f(X_t)\right]\\
    &=-\frac{1}{2}\E\left[\Tr X_t^*\nabla f(X_t)\right]+\frac{1}{2\sqrt{1-e^{-t}}}\E\left[\Tr G^*\nabla f(X_t)\right] . 
\end{align*}
Here, $\nabla f(X_t)\in\C^{n\times N}$ denotes the gradient matrix of $f$; in the last step, we used the definition of $X_t$. Let $\g_i$ denote the $i$-th column of $G$. Gaussian integration by parts gives
\begin{align*}
    &\E[\Tr G^*\nabla f(X_t)]=\sum_{i=1}^N\sum_{\alpha=1}^n\E\left[\g_i[\alpha]\frac{\de f}{\de\x_i[\alpha]}\right]\\
    &=\sum_{i=1}^N\sum_{\alpha=1}^n\sum_{\beta=1}^n\Sigma[\alpha,\beta]\sqrt{1-e^{-t}}\E\left[\frac{\de^2 f}{\de\x_i[\alpha]\de\x_i[\beta]}\right].
\end{align*}
Combining these calculations gives
\[
    \frac{\de}{\de t}\E f(X_t)=\frac{1}{2}\sum_{i=1}^N\sum_{\alpha,\beta=1}^n\Sigma[\alpha,\beta]\E\left[\frac{\de^2 f}{\de\x_i[\alpha]\de\x_i[\beta]}\right]-\frac{1}{2}\sum_{i=1}^N\sum_{\alpha=1}^n\E\left[\x_i[\alpha]\frac{\de f}{\de\x_i[\alpha]}\right].
\]
This completes the proof.
\end{proof}

We next record two useful properties of the operator $\cL_i$.
\begin{lemma}\label{lem:product-and-chain-rule}
Under the hypotheses of Lemma \ref{lem:OU-generator}, let $f,g:\R^{n}\rightarrow\C$ and $h:\C\rightarrow\C$ be such that $h$ is holomorphic on an open neighborhood of the range of $f$, and $f,g$, and $h(f)$ are twice differentiable functions whose first two derivatives have at most polynomial growth. Then, for all $i\in[N]$, the following hold:
\begin{enumerate}
    \item (Product rule)
    \[
        \cL_i(fg)=f\cL_ig+g\cL_if+\nabla f^\top\Sigma\nabla g.
    \]
    \item (Chain rule)
    \[
        \cL_ih(f)=h'(f)\cL_if+\frac{1}{2}h''(f)\nabla f^\top\Sigma\nabla f.
    \]
\end{enumerate}
\end{lemma}
\begin{proof}
The claims follow from direct computations.
\end{proof}

\subsection{Minor and one-column decomposition}
Unlike many classical local-law arguments, our proof requires only minors obtained by removing a single column.
\begin{defi}
For any $i\in[N]$, let $X^{(i)}\in\R^{n\times N}$ be the matrix with columns
\[
    X^{(i)}\e_j=\begin{cases}
        0 & \text{if $i=j$,}\\
        \x_j & \text{otherwise.}
    \end{cases}
\]
Define also
\begin{align*}
    \cR^{(i)}&=\cR(X^{(i)},Z,\beta)\in\C^{(n+N)\times(n+N)},\\
    A_i&:=[\cR^{(i)}]_{11}\in\C^{n\times n},\\
    B_i&=([\cR^{(i)}]_{11},[\cR^{(i)}]_{12})\in\C^{n\times (n+N)}.
\end{align*}
The expectation with respect to the column $\x_i$ is denoted by
\[
    \E_i[\cdot]:=\E[\cdot|\{\x_j\}_{j\neq i}].
\]
\end{defi}

We use Roman indices $i,j,...$ for $[N]$ and Greek indices $\alpha,\beta,...$ for $[n]$. The following resolvent identities follow from the Schur complement formula.
\begin{lemma}[Resolvent identities]\label{lem:schur-one-column}
We have
\begin{enumerate}[label=(\alph*)]
    \item For all $i\in[N]$,
    \[
        \e_i^*\cR\e_i=-\frac{1}{\beta+N^{-1}\x_i^*A_{i}\x_i},\quad 
        \cR_{11}=A_i-\frac{N^{-1}A_i\x_i\x_i^*A_i}{\beta+N^{-1}\x_i^*A_i\x_i}.
    \]
    \item For all $i\in[N]$,
    \begin{equation}\label{eq:minor-identity}
    \u^*\cR\u=\u^*\cR^{(i)}\u-\frac{(N^{-1/2}\x_i^*B_i\u)^2}{\beta+N^{-1}\x_i^*A_i\x_i}+2\u[i]\frac{N^{-1/2}\x_i^*B_i\u}{\beta+N^{-1}\x_i^*A_i\x_i}+\beta^{-1}(\u[i])^2-\frac{(\u[i])^2}{\beta+N^{-1}\x_i^*A_i\x_i} . 
    \end{equation}
\end{enumerate}
\end{lemma}
\begin{proof}
We can write
\[
    \cR=\left(\begin{bmatrix}
    0 & N^{-1/2}X^{(i)}\\
    N^{-1/2}X^{(i)*} & 0
    \end{bmatrix}+
    \left[
    \begin{array}{cc|c}
    -Z & 0 & N^{-1/2}\x_{i} \\
    0 & -\beta I_{N-1} & 0 \\
    \hline
    N^{-1/2}\x_{i}^* & 0 & -\beta
    \end{array}
    \right]
    \right)^{-1}.
\]
Part (a) follows by applying the Schur complement formula to the lower-right block and the Woodbury matrix identity to $\cR_{11}=(-Z+(N\beta)^{-1}XX^*)^{-1}$. Part (b) follows by applying the Schur complement formula to all four blocks, using the Woodbury matrix identity for the upper-left block, and multiplying both sides by $\u\in\R^{n+N}$.
\end{proof}

We also use the following standard closure properties of stochastic domination, recorded in \cite[Lemma~3.3]{fan2026anisotropic}.
\begin{lemma}[Stochastic-domination calculus]\label{lem:stochastic-domination-calculus}
Let $\cU=\cU_N$ and $\cV=\cV_N$ be deterministic parameter sets with $|\cV|\leq N^C$ for some constant $C>0$, and let all comparison scales below be deterministic and nonnegative.
\begin{enumerate}[label=(\alph*)]
    \item If $X(u,v)\prec \zeta(u,v)$ uniformly over $u\in\cU$ and $v\in\cV$, then, uniformly over $u\in\cU$,
    \[
        \sum_{v\in\cV}X(u,v)\prec\sum_{v\in\cV}\zeta(u,v).
    \]
    \item If $X_1\prec\zeta_1$ and $X_2\prec\zeta_2$ uniformly over $u\in\cU$, then $X_1X_2\prec\zeta_1\zeta_2$ uniformly over $u\in\cU$.
    \item Suppose that $X(u)\prec\zeta(u)$ uniformly over $u\in\cU$, where $\zeta$ is deterministic and $\zeta(u)>N^{-C}$ for some constant $C>0$. Suppose further that, for every fixed $k\geq 1$, there is a constant $C_k>0$ such that
    \[
        \E|X(u)|^k\leq N^{C_k}
    \]
    uniformly over $u\in\cU$. Then, uniformly over $u\in\cU$ and all sub-$\sigma$-fields $\mathcal G$ of the underlying probability space,
    \[
        \E[X(u)\mid\mathcal G]\prec\zeta(u).
    \]
\end{enumerate}
\end{lemma}

In the remainder of the section, we develop the estimates needed in the zag step. Define
\[
    \cY_i:=\x_i^*(N^{-1/2}B_i\u),\quad \cZ_i:=(N^{-1/2}\u^*B_i^\top)(\x_i\x_i^*-\Sigma)(N^{-1/2}B_i\u),
\]
\[
    \cA_i:=\Tr(\x_i\x_i^*-\Sigma)(N^{-1}A_i),\quad \cC_i:=(N^{-1/2}\u^*B_i^\top)\Sigma(N^{-1/2}B_i\u),\quad \cP_i:=(\beta+\x_i^*(N^{-1}A_i)\x_i)^{-1}.
\]
Note that the factor $N^{-1/2}$ always accompanies $B_i\u$, while the factor $N^{-1}$ always accompanies $A_i$.

\begin{lemma}\label{lem:L_i-bounds}
Under Assumptions \ref{ass:A1} and \ref{ass:A2}, let $\bD$ be a regular domain and fix $h,t\geq 0$. Suppose there exist $\delta>0$ and a deterministic $\Phi:\bD\rightarrow[N^{-10},1]$ such that, uniformly over all $i\in[N]$, all $z\in\bD$, and all deterministic unit vectors $\u\in\R^{n+N}$, we have
\[
    N^{-1/2}\|B_i\u\|_2, N^{-1}\|A_i\|_F\prec\Phi,\quad \|\Sigma^{1/2}B_i\u\|_2, N^{-1/2}\|\Sigma^{1/2}A_i\|_F\prec N^\delta\sqrt{\Lambda_h},
\]
and
\[
    |\beta+N^{-1}\Tr\Sigma A_i|^{-1}, |\beta+N^{-1}\x_i^*A_i\x_i|^{-1}\prec N^\delta.
\]
Then, for every fixed $L\geq 1$, uniformly over all $l_1,l_2,l_3,l_4\geq 0,l_1+l_2+l_3+l_4=L$, all $i\in[N]$, all $z\in\bD$, and all deterministic unit vectors $\u\in\R^{N+n}$,
\[
    |\E_i\cL_i[\cP_i^L(\u[i]\cY_i)^{l_1}\cZ_i^{l_2}(\u[i]^2\cA_i)^{l_3}(\cC_i\cA_i)^{l_4}]|\prec c_i(\Lambda_h+1) N^{\delta (L+4)}\Phi^{L}.
\]
Here
\[
    |c_i|\leq N^{-1}+N^{-1/2}|\u[i]|+\u[i]^2.
\]
\end{lemma}
\begin{proof}
We first record how the operators $\cL_i$ and $\nabla$ act on the factors $\cY_i,\cZ_i,\cA_i$. This will streamline the later combinatorial arguments. We have
\[
    \cL_i\cY_i=-\frac{1}{2}\cY_i,\quad \cL_i\cA_i=\cL_i\x_i^*(N^{-1}A_i)\x_i=-\cA_i,\quad \cL_i\cZ_i=-\cZ_i
\]
\[
    \nabla\cY_i=N^{-1/2}B_i\u,\quad \nabla\cZ_i=2\cY_i(N^{-1/2}B_i\u),\quad \nabla\cA_i=\nabla\x_i^*(N^{-1}A_i)\x_i=2(N^{-1}A_i\x_i)
\]
\[
    \cL_i\cP_i^L=L\cP_i^{L+1}\cA_i+2L(L+1)\cP_i^{L+2}\x_i^*(N^{-1}A_i)\Sigma(N^{-1}A_i)\x_i,\quad \nabla\cP_i^L=-2L\cP_i^{L+1}(N^{-1}A_i\x_i).
\]
These identities follow directly from the definition of $\cL_i$. We will also repeatedly use the product and chain rules from Lemma \ref{lem:product-and-chain-rule}.

By the product and chain rules, we have
\begin{align*}
    &|\E_i\cL_i[\cP_i^L(\u[i]\cY_i)^{l_1}\cZ_i^{l_2}(\u[i]^2\cA_i)^{l_3}(\cC_i\cA_i)^{l_4}]|=\underbrace{(|\u[i]|)^{l_1+2l_3}|\cC_i|^{l_4}}_{\omega_i}|\E_i\cL_i[\cP_i^L\cY_i^{l_1}\cZ_i^{l_2}\cA_i^{l_3+l_4}]|\\
    &\leq\omega_i|\E_i\cP_i^L\cL_i[\cY_i^{l_1}\cZ_i^{l_2}\cA_i^{l_3+l_4}]|+\omega_i|\E_i\cY_i^{l_1}\cZ_i^{l_2}\cA_i^{l_3+l_4}\cL_i\cP_i^L|+\omega_i|\E_i\nabla(\cP_i^L)^\top\Sigma\nabla(\cY_i^{l_1}\cZ_i^{l_2}\cA_i^{l_3+l_4})|\\
    &=\underbrace{\omega_i|\E_i\cP_i^L\cL_i[\cY_i^{l_1}\cZ_i^{l_2}\cA_i^{l_3+l_4}]|}_{\mathbf{A}}\\
    &+\underbrace{L\omega_i|\E_i\cP_i^{L+1}\cY_i^{l_1}\cZ_i^{l_2}\cA_i^{l_3+l_4+1}|}_\mathbf{B}\\
    &+\underbrace{2L(L+1)\omega_i|\E_i\cP_i^{L+2}\cY_i^{l_1}\cZ_i^{l_2}\cA_i^{l_3+l_4}\x_i^*(N^{-1}A_i)\Sigma(N^{-1}A_i)\x_i|}_{\mathbf{C}}\\
    &+\underbrace{2Ll_1\omega_i|\E_i\cP_i^{L+1}\cY_i^{l_1-1}\cZ_i^{l_2}\cA_i^{l_3+l_4}\x_i^*(N^{-1}A_i)\Sigma(N^{-1/2}B_i\u)|}_{\mathbf{D}}\\
    &+\underbrace{4Ll_2\omega_i|\E_i\cP_i^{L+1}\cY_i^{l_1+1}\cZ_i^{l_2-1}\cA_i^{l_3+l_4}\x_i^*(N^{-1}A_i)\Sigma(N^{-1/2}B_i\u)|}_{\mathbf{E}}\\
    &+\underbrace{4L(l_3+l_4)\omega_i|\E_i\cP_i^{L+1}\cY_i^{l_1}\cZ_i^{l_2}\cA_i^{l_3+l_4-1}\x_i^*(N^{-1}A_i)\Sigma(N^{-1}A_i)\x_i|}_{\mathbf{F}}.
\end{align*}
Here $\mathbf{B}$ and $\mathbf{C}$ arise from evaluating $\cL_i\cP_i^L$, while $\mathbf{D},\mathbf{E}$, and $\mathbf{F}$ arise from computing $\nabla(\cP_i^L)$ and $\nabla(\cY_i^{l_1}\cZ_i^{l_2}\cA_i^{l_3+l_4})$. We begin by bounding $\mathbf{C,D,E}$ and $\mathbf{F}$. By Assumption \ref{ass:A2} and the hypotheses of the lemma, the term $\mathbf{C}$ satisfies
\begin{align*}
    |\mathbf{C}|&\prec\omega_iN^{\delta(L+2)}\Phi^{L+l_2}|\x_i^*(N^{-1}A_i)\Sigma(N^{-1}A_i)\x_i|\\
    &\prec \omega_iN^{\delta(L+2)}\Phi^{L}N^{-1}(N^{-1/2}\|\Sigma^{1/2}A_i\|_F)^2\\
    &\prec N^{-1}(\Phi^2)^{l_4}N^{\delta(L+4)}\Lambda_h\Phi^L\prec c_iN^{\delta(L+4)}\Lambda_h\Phi^L.
\end{align*}
Here, the last step uses $\omega_i\leq |\cC_i|^{l_4}\prec (\Phi^2)^{l_4}\leq 1$.

For $\mathbf{D}$, if $l_1=0$, then $\mathbf{D}=0$, so it remains to consider the case where $l_1\geq 1$. We have
\begin{align*}
    |\mathbf{D}|&\prec\omega_iN^{\delta(L+1)}\Phi^{L+l_2-1}|\x_i^*(N^{-1}A_i)\Sigma(N^{-1/2}B_i\u)|\\
    &\prec\omega_iN^{\delta(L+1)}\Phi^{L-1}N^{-1}\|A_i\|_F\|\Sigma^{1/2}\|_\op N^{-1/2}\|\Sigma^{1/2}B_i\u\|_2\\
    &\prec\omega_iN^{-1/2}N^{\delta(L+2)}\sqrt{\Lambda_h}\Phi^{L}\prec c_iN^{\delta(L+2)}(\Lambda_h+1)\Phi^{L}.
\end{align*}
Here, the last step uses $\omega_i\prec|\u[i]|$ and $\sqrt{\Lambda_h}\leq\Lambda_h+1$. The same arguments show that $|\mathbf{E}|,|\mathbf{F}|\prec c_iN^{\delta(L+4)}\Lambda_h\Phi^L$.

For $\mathbf{B}$, if $l_3+l_4\geq 1$, we have
\[
    |\mathbf{B}|\prec\omega_iN^{\delta(L+1)}\Phi^{L+l_2+1}\prec c_iN^{\delta(L+3)}(\Lambda_h+1)\Phi^L.
\]
Here, we used $\omega_i\prec(|\u[i]|)^{2l_3}|\cC_i|^{l_4}\prec(\u[i]^2+N^{-1})N^{2\delta}(\Lambda_h+1)$. If $l_1\geq 1$, we have
\begin{align*}
    |\mathbf{B}|&\prec\omega_iN^{\delta(L+1)}\Phi^{L+l_2}\E_i|\cY_i|\leq \omega_iN^{\delta(L+1)}\Phi^{L}(\E_i|\cY_i|^2)^{1/2}\\
    &=\omega_iN^{\delta(L+1)}\Phi^{L}(N^{-1}\u^*B_i^*\Sigma B_i\u)^{1/2}\prec c_iN^{\delta(L+2)}(\Lambda_h+1)\Phi^L.
\end{align*}
Here, we used
\[
    (N^{-1}\u^*B_i^*\Sigma B_i\u)^{1/2}\prec N^{-1/2}\|\Sigma^{1/2}B_i\u\|_2\prec N^{-1/2}N^\delta\sqrt{\Lambda_h}\prec N^{-1/2}N^\delta(\Lambda_h+1).
\]
In this case, we also use the fact that $\omega_i\prec|\u[i]|$. Finally, suppose $l_1+l_3+l_4=0$. Then $l_2=L\geq 1$ and $\omega_i\prec 1$, and we have
\[
    |\mathbf{B}|\prec N^{\delta(L+1)}\E_i|\cZ_i|^L|\cA_i|\prec N^{\delta(L+1)}\Phi^{2(L-1)+1}\E_i|\cZ_i|\leq N^{\delta(L+1)}\Phi^L\E_i|\cZ_i|.
\]
Here, we used $2L-1=L+(L-1)\geq L$. We can write
\begin{align*}
    \E_i|\cZ_i|&\leq\E_i|N^{-1/2}\x_i^*B_i\u|^2+|(N^{-1/2}\u^*B_i^\top)\Sigma(N^{-1/2}B_i\u)|\\
    &\prec N^{-1}\|\Sigma^{1/2}B_i\u\|_2^2\prec N^{-1}N^{2\delta}\Lambda_h.
\end{align*}
Combining these estimates shows that $|\mathbf{B}|\prec c_iN^{\delta(L+3)}\Lambda_h\Phi^L$.

For $\mathbf{A}$, the product and chain rules give
\begin{align*}
    &\cL_i[\cY_i^{l_1}\cZ_i^{l_2}\cA_i^{l_3+l_4}]=\cZ_i^{l_2}\cA_i^{l_3+l_4}\cL_i\cY_i^{l_1}+\cY_i^{l_1}\cA_i^{l_3+l_4}\cL_i\cZ_i^{l_2}+\cY_i^{l_1}\cZ_i^{l_2}\cL_i\cA_i^{l_3+l_4}\\
    &+\cA_i^{l_3+l_4}\nabla(\cY_i^{l_1})^\top\Sigma\nabla(\cZ_i^{l_2})+\cZ_i^{l_2}\nabla(\cY_i^{l_1})^\top\Sigma\nabla(\cA_i^{l_3+l_4})+\cY_i^{l_1}\nabla(\cZ_i^{l_2})^\top\Sigma\nabla(\cA_i^{l_3+l_4}).
\end{align*}
Expanding once more with the chain rule and the displayed gradient identities gives
\begin{align*}
    &\mathbf{A}=\underbrace{(-l_1/2-l_2-l_3-l_4)\omega_i\E_i\cP_i^L\cY_i^{l_1}\cZ_i^{l_2}\cA_i^{l_3+l_4}}_{\mathbf{A}_1}\\
    &+\frac12l_1(l_1-1)\omega_i\E_i\cP_i^L\cY_i^{l_1-2}\cZ_i^{l_2}\cA_i^{l_3+l_4}(N^{-1/2}\u^*B_i^\top)\Sigma(N^{-1/2}B_i\u)\\
    &+2l_2(l_2-1)\omega_i\E_i\cP_i^L\cY_i^{l_1+2}\cZ_i^{l_2-2}\cA_i^{l_3+l_4}(N^{-1/2}\u^*B_i^\top)\Sigma(N^{-1/2}B_i\u)\\
    &+2(l_3+l_4)(l_3+l_4-1)\omega_i\E_i\cP_i^L\cY_i^{l_1}\cZ_i^{l_2}\cA_i^{l_3+l_4-2}\x_i^*(N^{-1}A_i)\Sigma(N^{-1}A_i)\x_i\\
    &+2l_1l_2\omega_i\E_i\cP_i^L\cY_i^{l_1}\cZ_i^{l_2-1}\cA_i^{l_3+l_4}(N^{-1/2}\u^*B_i^\top)\Sigma(N^{-1/2}B_i\u)\\
    &+2l_1(l_3+l_4)\omega_i\E_i\cP_i^L\cY_i^{l_1-1}\cZ_i^{l_2}\cA_i^{l_3+l_4-1}\x_i^*(N^{-1}A_i)\Sigma(N^{-1/2}B_i\u)\\
    &+4l_2(l_3+l_4)\omega_i\E_i\cP_i^L\cY_i^{l_1+1}\cZ_i^{l_2-1}\cA_i^{l_3+l_4-1}\x_i^*(N^{-1}A_i)\Sigma(N^{-1/2}B_i\u).
\end{align*}
The arguments for bounding the last six terms are similar to those for $\mathbf{B}-\mathbf{F}$. We therefore omit them and focus on bounding $\mathbf{A}_1$. If $l_3+l_4\geq 1$, then the naive bound gives
\[
    |\mathbf{A}_1|\prec\omega_i N^{\delta L}\Phi^{L+l_2}\prec c_i N^{\delta (L+2)}(\Lambda_h+1)\Phi^L.
\]
Here, as above, we used $\omega_i\prec(\u[i])^{2l_3}|\cC_i|^{l_4}\prec(\u[i]^2+N^{-1})N^{2\delta}(\Lambda_h+1)$. If $l_1\geq 2$, then the naive bound also gives
\[
    |\mathbf{A_1}|\prec\omega_i N^{\delta L}\Phi^{L+l_2}\prec c_i N^{\delta L}(\Lambda_h+1)\Phi^L.
\]
Here, we used $\omega_i\prec|\u[i]|^{l_1}\prec|\u[i]|^2$ and $1\leq \Lambda_h+1$. If $l_2\geq 2$, then we have
\begin{align*}
    |\mathbf{A}_1|\prec\omega_i N^{\delta L}\Phi^{L+l_2-2}\E_i|\cZ_i|\prec c_i N^{\delta(L+2)}(\Lambda_h+1)\Phi^L.
\end{align*}
Here, we used $l_2-2\geq 0$ and, as in the case of $\mathbf{B}$, the bound $\E_i|\cZ_i|\prec N^{-1}N^{2\delta}\Lambda_h$. It remains to consider the case where $l_1,l_2\leq 1$ and $l_3+l_4=0$.

If $l_1=l_2=1,l_3+l_4=0$, then, as in the case of $\mathbf{B}$,
\[
    |\mathbf{A}_1|\prec|\u[i]| N^{2\delta}\Phi^2\E_i|\cY_i|\prec N^{-1/2}|\u[i]|N^{\delta(L+1)}(\Lambda_h+1)\Phi^L.
\]
Here, we used $L=2$. If $l_1=1,l_2+l_3+l_4=0$, then $L=1$, and we can write
\[
    \cP_i=(\beta+N^{-1}\Tr\Sigma A_i)^{-1}-(\beta+N^{-1}\Tr\Sigma A_i)^{-1}\cP_i\cA_i.
\]
Hence,
\begin{align*}
    |\mathbf{A}_1|&\prec|\u[i]||\beta+N^{-1}\Tr\Sigma A_i|^{-1}|\E_i\cY_i|+|\u[i]||\beta+N^{-1}\Tr\Sigma A_i|^{-1}|\E_i\cP_i\cY_i\cA_i|\\
    &\prec |\u[i]|N^{2\delta}\Phi\E_i|\cY_i|\prec N^{-1/2}|\u[i]|N^{3\delta}(\Lambda_h+1)\Phi\leq c_i N^{\delta(L+2)}(\Lambda_h+1)\Phi^L.
\end{align*}
Here, we used $\E_i\cY_i=0$. Similarly, if $l_2=1,l_1+l_3+l_4=0$, then $L=1$, and we have
\begin{align*}
    |\mathbf{A}_1|&\prec |\beta+N^{-1}\Tr\Sigma A_i|^{-1}|\E_i\cZ_i|+|\beta+N^{-1}\Tr\Sigma A_i|^{-1}|\E_i\cP_i\cZ_i\cA_i|\\
    &\prec N^{2\delta}\Phi\E_i|\cZ_i|\prec N^{-1}N^{4\delta}\Lambda_h\Phi\leq c_i N^{\delta(L+3)}\Lambda_h\Phi^L.
\end{align*}
Here, we used $\E_i\cZ_i=0$. Combining all the cases above and using $\Lambda_h\leq\Lambda_h+1$ completes the proof.
\end{proof}

\begin{remark}\label{remk:complex-conjugate}
Since both $\u$ and $\x_i$ are real and $\cR$ is complex symmetric, the preceding lemma and its proof extend directly to the case in which some factors are replaced by their complex conjugates.
\end{remark}

We are now ready to prove the zag estimate.
\begin{proposition}[Zag estimate]\label{prop:full-block-zag}
Fix $h,t\geq 0$. Suppose that the hypotheses of Lemma \ref{lem:L_i-bounds} hold uniformly for all $0\leq s\leq t$ and all $i\in[N]$, with a deterministic control parameter $\Phi_1:\bD\rightarrow[N^{-10},1]$. Suppose further that there exists a deterministic $\Phi_2:\bD\rightarrow[N^{-10},1]$ satisfying $N^{10\delta}\Phi_1\leq\Phi_2$ such that, uniformly over all $z\in\bD$ and all deterministic unit vectors $\u\in\R^{N+n}$, we have
\[
    t(\Lambda_h+1)\leq CN^\delta,\quad |\u^*(\cR_{h,t}-M_h)\u|\prec \Phi_2.
\]
Here, $C>0$ is a constant. Then, uniformly over all $0\leq s\leq t$, all $z\in\bD$, and all deterministic unit vectors $\u\in\R^{N+n}$,
\[
    |\u^*(\cR_{h,s}-M_h)\u|\prec \Phi_2.
\]
\end{proposition}
\begin{proof}
Let $\cB_i:=(\beta+N^{-1}\Tr\Sigma A_i)^{-1}$. By the minor resolvent identity, we have
\begin{align*}
    Y(s)&=\u^*(\cR_{h,s}-M_h)\u=\u^*(\cR_{h,s}^{(i)}-M_h)\u+\beta^{-1}(\u[i])^2-\frac{(\u[i])^2}{\beta+N^{-1}\x_i^*A_i\x_i}\\
    &-\frac{(N^{-1/2}\x_i^*B_i\u)^2}{\beta+N^{-1}\x_i^*A_i\x_i}+2\u[i]\frac{N^{-1/2}\x_i^*B_i\u}{\beta+N^{-1}\x_i^*A_i\x_i}\\
    &=\underbrace{\u^*(\cR_{h,s}^{(i)}-M_h)\u+\beta^{-1}(\u[i])^2-(\u[i])^2\cB_i-\cB_i\cC_i}_{\widetilde Y(s)}\\
    &+2\u[i]\cP_i\cY_i-\cP_i\cZ_i+\cB_i\cP_i\cC_i\cA_i+\cB_i(\u[i])^2\cP_i\cA_i.
\end{align*}
It follows that $\widetilde Y(s)$ is independent of $\x_i$ and, by the hypotheses of the proposition, that $|Y(s)-\widetilde Y(s)|\prec N^{2\delta}\Phi_1$.

To obtain high-probability estimates via Markov's inequality, fix a large $L>0$ and consider the $2L$-th moment of $Y(s)$. By Lemma \ref{lem:OU-generator},
\begin{align*}
    &\left|\frac{\de}{\de s}\E|Y(s)|^{2L}\right|=\left|\frac{\de}{\de s}\E(Y(s))^L(\overline{Y(s)})^L\right|\\
    &=\left|\sum_{i=1}^N\E\cL_i(\widetilde Y(s)+(Y(s)-\widetilde Y(s)))^L(\overline{\widetilde Y(s)}+(\overline{Y(s)}-\overline{\widetilde Y(s)}))^L\right|\\
    &=\left|\sum_{i=1}^N\sum_{a,a'=0}^L\E\cL_i\widetilde Y(s)^{L-a}\overline{\widetilde Y(s)}^{L-a'}(Y(s)-\widetilde Y(s))^a(\overline{Y(s)}-\overline{\widetilde Y(s)})^{a'}\right|\\
    &\leq C_L\sum_{i=1}^N\sum_{a,a'=0}^L\sum_{\substack{l_1+l_2+l_3+l_4=a\\l_1'+l_2'+l_3'+l_4'=a'}}\E|\widetilde Y(s)|^{2L-a-a'}|\cB_i|^{l_3+l_4+l_3'+l_4'}\\
    &\times|\E_i\cL_i\cP_i^{a}\overline{\cP_i}^{a'}(\u[i]\cY_i)^{l_1}(\u[i]\overline{\cY_i})^{l_1'}\cZ_i^{l_2}\overline{\cZ_i}^{l_2'}(\cC_i\cA_i)^{l_3}(\overline{\cC_i\cA_i})^{l_3'}(\u[i]^2\cA_i)^{l_4}(\u[i]^2\overline{\cA_i})^{l_4'}|\\
    &\prec N^{4\delta}(\Lambda_h+1)\sum_{i=1}^Nc_i\sum_{1\leq a+a'\leq 2L}N^{2\delta(a+a')}\Phi_1^{a+a'}\E|\widetilde Y(s)|^{2L-a-a'}\\
    &\prec N^{4\delta}(\Lambda_h+1)\sum_{l=1}^{2L}N^{2\delta l}\Phi_1^{l}\E|\widetilde Y(s)|^{2L-l}\prec N^{6\delta}(\Lambda_h+1)\sum_{l=1}^{2L}(N^{2\delta}\Phi_1)^{l}\E|Y(s)|^{2L-l}.
\end{align*}
Here we used that the summand with $a+a'=0$ vanishes under $\cL_i$, together with Lemma \ref{lem:L_i-bounds}, Remark \ref{remk:complex-conjugate}, the summability of $c_i$, the bound $|\cB_i|\prec N^\delta$, and, in the last step, the bound $|Y(s)-\widetilde Y(s)|\prec N^{2\delta}\Phi_1$.

We now complete the proof using a Gr\"onwall-type argument. Let $H(s)=1+(N^{2\delta}\Phi_1)^{-2L}\E|Y(s)|^{2L}$, so that $H(s)\geq 1$ for all $s\ge 0$. Multiplying both sides of the preceding inequality by $(N^{2\delta}\Phi_1)^{-2L}$ and applying the inequality
\[
    (\E|Y(s)|^{2L-l})^{1/(2L-l)}\leq (\E|Y(s)|^{2L})^{1/(2L)}
\]
gives for any $\eps>0$
\begin{align*}
    |H'(s)|&\leq N^\eps N^{6\delta}(\Lambda_h+1)\sum_{l=1}^{2L}[(N^{2\delta}\Phi_1)^{-2L}\E|Y(s)|^{2L}]^{(2L-l)/(2L)}\\
    &\leq N^\eps N^{6\delta}(\Lambda_h+1)\sum_{l=1}^{2L}H(s)^{(2L-l)/(2L)}\leq N^\eps N^{6\delta}(\Lambda_h+1) H(s)^{(2L-1)/(2L)}.
\end{align*}
On the other hand,
\[
    \left|\frac{\de}{\de s}H(s)^{1/(2L)}\right|=\frac{1}{2L}|H(s)^{-(2L-1)/(2L)}H'(s)|\leq C_LN^{6\delta}(\Lambda_h+1).
\]
It follows that, for any $0\leq s\leq t$,
\[
    H(s)^{1/(2L)}\leq H(t)^{1/(2L)}+\int_{s}^t\left|\frac{\de}{\de s}H(s)^{1/(2L)}\right|\de s\leq H(t)^{1/(2L)}+N^\eps N^{6\delta}t(\Lambda_h+1).
\]
The hypothesis of the proposition gives
\begin{align*}
    N^{2\delta}\Phi_1H(t)^{1/(2L)}&\leq((N^{2\delta}\Phi_1)^{2L}+\E|Y(t)|^{2L})^{1/(2L)}\\
    &\leq ((N^{2\delta}\Phi_1)^{2L}+(N^{\eps}\Phi_2)^{2L})^{1/(2L)}\leq CN^{\eps}\Phi_2.
\end{align*}
Combining this with the hypothesis $t(\Lambda_h+1)\leq CN^\delta$ gives
\begin{align*}
    N^{2\delta}\Phi_1H(s)^{1/(2L)}&=((N^{2\delta}\Phi_1)^{2L}+\E|Y(s)|^{2L})^{1/(2L)}\leq CN^{\eps}\Phi_2 + N^\eps N^{9\delta}\Phi_1\leq CN^{\eps}\Phi_2\\
    \implies&\E|Y(s)|^{2L}\leq (N^{2\delta}\Phi_1)^{2L}+C(N^{\eps}\Phi_2)^{2L}\leq C(N^{\eps}\Phi_2)^{2L}.
\end{align*}
Since $\eps$ was arbitrary, the claim follows from Markov's inequality.
\end{proof}

\section{Proof of Theorem \ref{thm:generalized-local-law}}\label{sec:proof-main}

We now prove Theorem \ref{thm:generalized-local-law} by implementing the spectral-scale bootstrap described in Section \ref{subsec:proof-outline}, with multiplicative increments of $N^{-\delta}$. This approach is similar to those of \cite{fan2026anisotropic,knowles2017anisotropic}. We begin by showing that, for fixed $h,t\geq 0$, the relevant quadratic forms of both $\cR_{h,t}(z)$ and $M_h(z)$ are Herglotz functions.

\begin{lemma}\label{lem:Herglotz-function}
For any $h,t\geq 0$, any unit vector $\u\in\R^{n+N}$, and any $z\in\C^+$, define $s_1(z):=\u^*\cR_{h,t}(z)\u$ and $s_2(z)=\u^*M_h(z)\u$ using the principal square root $w=\sqrt{z}\in\C^+$. Then, for $i=1,2$, $s_i$ is holomorphic from $\C^+$ to $\C^+$ and is therefore a Herglotz function. Moreover, the following statements hold:
\begin{enumerate}[label=(\alph*)]
    \item For any $E\in\R$ and $0<\eta_1<\eta_2$,
    \[
        \frac{\eta_1}{\eta_2}\Im s_i(E+i\eta_2)\leq\Im s_i(E+i\eta_1)\leq \frac{\eta_2}{\eta_1}\Im s_i(E+i\eta_2).
    \]
    \item For any $E\in\R$ and $0<\eta_1<\eta_2$,
    \[
        |s_i(E+i\eta_1)-s_i(E+i\eta_2)|\leq \frac{|\eta_2-\eta_1|}{\eta_1}\Im s_i(E+i\eta_2).
    \]
\end{enumerate}
\end{lemma}
\begin{proof}
We prove the claims for $s_1(z)=\u^*\cR_{h,t}(z)\u$, with the principal square root $w=\sqrt{z}\in\C^+$. The proof for $s_2(z)=\u^*M_h(z)\u$ is analogous and simpler, so we omit it.

Recall $s_1(z)=\u^*(\cX_t-\cD_h(z))^{-1}\u$, where $\cD_h(z)=\diag(Z_h(z),\beta_h(z)I_N)$ and
\begin{align*}
    Z_h(z)&=e^{h/2}\sqrt{z}I_n
    +(e^{h/2}-e^{-h/2})\sqrt{z}\widetilde m_0(z)\Sigma,\\
    \beta_h(z)&=e^{h/2}\sqrt{z}
    +(e^{h/2}-e^{-h/2})(-\sqrt{z}-(\sqrt{z}\widetilde m_0(z))^{-1}).
\end{align*}
The function $\widetilde m_0(z)$ is well-defined and holomorphic on all of $\C^+$. Moreover,
\[
    \Im\cD_h(z)=\diag(\Im Z_h(z),\Im\beta_h(z)I_N),
\]
where
\begin{align*}
    \Im Z_h(z)&=e^{h/2}\Im \sqrt{z}I_n+(e^{h/2}-e^{-h/2})\Im \sqrt{z}\widetilde m_0(z)\Sigma\\
    \Im\beta_h(z)&=e^{h/2}\Im \sqrt{z}+(e^{h/2}-e^{-h/2})(-\Im \sqrt{z}+\Im \sqrt{z}\widetilde m_0(z)/|\sqrt{z}\widetilde m_0(z)|^2)\\
    &=e^{-h/2}\Im \sqrt{z}+(e^{h/2}-e^{-h/2})\Im \sqrt{z}\widetilde m_0(z)/|\sqrt{z}\widetilde m_0(z)|^2.
\end{align*}
As in the proof of Lemma \ref{lem:basic-estimates},
\[
    \Im \sqrt{z}\widetilde m_0(z)=\int\frac{\Im \sqrt{z}(\lambda-\bar z)}{|\lambda-z|^2}\;\de\widetilde\mu_0(\lambda)=\int\frac{\lambda\Im \sqrt{z}+|z|\Im \sqrt{z}}{|\lambda-z|^2}\;\de\widetilde\mu_0(\lambda)\geq \frac{|z|\Im \sqrt{z}}{(C+|z|)^2}>0.
\]
Thus, $\Im\cD_h(z)>0$, so $\cR_{h,t}(z)$ is also well-defined and analytic. The same lower bound shows that $\Im s_1(z)=\u^*\cR_{h,t}\Im\cD_h\cR_{h,t}^*\u>0$. Hence, $s_1$ is a Herglotz function and admits the integral representation
\[
    s_1(z)=c+dz+\int\left(\frac{1}{\lambda-z}-\frac{\lambda}{1+\lambda^2}\right)\;\de\nu(\lambda),\quad \int\frac{1}{1+\lambda^2}\;\de\nu(\lambda)<\infty.
\]
Here, $c\in\R$, $d\geq0$, and $\nu$ is a Borel measure on $\R$. The coefficient $d$ is given by $d=\lim_{\eta\rightarrow\infty}s_1(i\eta)/(i\eta)$. We have
\begin{align*}
    Z_h(i\eta)&=\sqrt{i\eta}(e^{h/2}I_n+(e^{h/2}-e^{-h/2})\widetilde m_0(i\eta)\Sigma)\\
    \beta_h(i\eta)&=\sqrt{i\eta}(e^{h/2}+(e^{h/2}-e^{-h/2})(-1-(i\eta\widetilde m_0(i\eta))^{-1})).
\end{align*}
Using $\lim_{\eta\rightarrow\infty}\widetilde m_0(i\eta)=0$ and $\lim_{\eta\rightarrow\infty}i\eta\widetilde m_0(i\eta)=-1$, we deduce
\[
    \lim_{\eta\rightarrow\infty}\cD_h/\sqrt{i\eta}=\lim_{\eta\rightarrow\infty}\diag(Z_h(i\eta)/\sqrt{i\eta},\beta_h(i\eta)/\sqrt{i\eta}I_N)=e^{h/2}I_{N+n}.
\]
This limit is nonzero for all $h\geq 0$. Thus,
\[
    d=\lim_{\eta\rightarrow\infty}s_1(i\eta)/(i\eta)=\lim_{\eta\rightarrow\infty}\u^*(\cX_t/\sqrt{i\eta}-\cD_h/\sqrt{i\eta})^{-1}\u/(i\eta)^{3/2}=0\cdot (-e^{-h/2})=0.
\]
To prove (a), observe that, for $z=E+i\eta\in\C^+$,
\[
    \Im s_1(z)=\int\frac{\eta}{(\lambda-E)^2+\eta^2}\;\de\nu(\lambda).
\]
Thus, (a) follows from the fact that, for any $E\in\R$, the function $\eta\rightarrow \eta^2/((\lambda-E)^2+\eta^2)$ is increasing and the function $\eta\rightarrow 1/((\lambda-E)^2+\eta^2)$ is decreasing.

For (b), we can write
\begin{align*}
    |s_1(E+i\eta_1)-s_1(E+i\eta_2)|&\leq |\eta_2-\eta_1|\int\frac{1}{|\lambda-(E+i\eta_1)||\lambda-(E+i\eta_2)|}\;\de\nu(\lambda)\\
    &\leq|\eta_2-\eta_1|\left(\frac{\Im s_1(E+i\eta_1)\Im s_1(E+i\eta_2)}{\eta_1\eta_2}\right)^{1/2}.
\end{align*}
Here, the last step follows from the Cauchy--Schwarz inequality. Applying (a) completes the proof.
\end{proof}

Fix $C_0>100$ and choose $\delta\in(0,\tau/(10C_0))$. Following the approach of \cite{fan2026anisotropic,knowles2017anisotropic}, we bootstrap on the spectral scale in multiplicative increments of $N^{-\delta}$. Let $\bD$ be a regular domain. We first prove Theorem \ref{thm:generalized-local-law} on the smaller domain $\hat\bD:=\{z\in\bD:N\eta\Im\widetilde m_0(z)\geq N^{3C_0\delta}\}$. For any $E\in\R$, the function $\eta\rightarrow\eta\Im\widetilde m_0(E+i\eta)$ is increasing; therefore, this smaller domain is also regular.

The following lemma provides the minor estimates used throughout the zag step.
\begin{lemma}\label{lem:minor-estimates}
Under Assumptions \ref{ass:A1} and \ref{ass:A2}, let $\bD$ be a regular domain and fix $h,t\geq 0$. Suppose that, uniformly over all $z\in\bD$ and all $i\in[N]$, we have
\[
    |\beta+N^{-1}\x_i^*A_i\x_i|^{-1}\prec N^\delta.
\]
Then, uniformly over all $z\in\bD$ satisfying $N\eta\Im\widetilde m_0(z)\geq N^{3C_0\delta}$, all deterministic unit vectors $\u\in\R^{n+N}$, and all $i\in[N]$,
\begin{align*}
    &N^{-1/2}\|B_i\u\|_2,N^{-1}\|A_i\|_F\prec\sqrt{\frac{e^{-h/2}\u^*\Im\cR\u}{N\eta}}+\left(N^\delta\Psi_h\sqrt{\frac{e^{-h/2}\u^*\Im\cR\u}{N\eta}}\right)^{1/2}\\
    &+\max_{j=1}^{n+N}\sqrt{\frac{e^{-h/2}\e_j^*\Im\cR\e_j}{N\eta}}+\max_{j=1}^{n+N}\left(N^\delta\Psi_h\sqrt{\frac{e^{-h/2}\e_j^*\Im\cR\e_j}{N\eta}}\right)^{1/2}+N^{\delta/2}\Psi_h,
\end{align*}
\begin{align*}
&\|\Sigma^{1/2}B_i\u\|_2,N^{-1/2}\|\Sigma^{1/2}A_i\|_F\prec\sqrt{\frac{e^{-h/2}\u^*\Im\cR\u}{\eta_h}}+\left(\sqrt{\Lambda_h}\sqrt{\frac{e^{-h/2}\u^*\Im\cR\u}{\eta_h}}\right)^{1/2}\\
&+\max_{j=1}^{n+N}\sqrt{\frac{e^{-h/2}\e_j^*\Im\cR\e_j}{\eta_h}}+\max_{j=1}^{n+N}\left(\sqrt{\Lambda_h}\sqrt{\frac{e^{-h/2}\e_j^*\Im\cR\e_j}{\eta_h}}\right)^{1/2}+\sqrt{\Lambda_h}.
\end{align*}
Here, $\eta_h:=\eta+(1-e^{-h})\Im\widetilde m_0(z)$.
\end{lemma}
\begin{proof}
Let $\u^{(i)}\in\R^{n+N}$ denote the vector obtained by setting the $i$-th entry of $\u$ to $0$. Since the $i$-th column of $B_i$ is zero, we have $\|B_i\u\|_2=\|B_i\u^{(i)}\|_2$. Thus, using $\u^{(i)*}B_i^*B_i\u^{(i)}\leq \u^{(i)*}\cR^{(i)*}\cR^{(i)}\u^{(i)}$, the argument from the proof of Lemma \ref{lem:ward-estimate} shows
\[
    N^{-1/2}\|B_i\u\|_2=N^{-1/2}\|B_i\u^{(i)}\|_2\prec\sqrt{\frac{e^{-h/2}\u^{(i)*}\Im\cR^{(i)}\u^{(i)}}{N\eta}}.
\]
The resolvent identity (\ref{eq:minor-identity}) and the hypothesis of the lemma therefore give
\begin{align*}
    |\u^{(i)*}\cR\u^{(i)}-\u^{(i)*}\cR^{(i)}\u^{(i)}|&\prec N^\delta\frac{e^{-h/2}\u^{(i)*}\Im\cR^{(i)}\u^{(i)}}{N\eta}\\
    &\prec N^\delta\frac{e^{-h/2}\u^{(i)*}\Im\cR\u^{(i)}}{N\eta}+N^\delta\frac{e^{-h/2}|\u^{(i)*}\cR\u^{(i)}-\u^{(i)*}\cR^{(i)}\u^{(i)}|}{N\eta}.
\end{align*}
Since $e^{-h/2}(N\eta)^{-1}\leq (N\eta)^{-1}\leq N^{-\tau}$, rearranging shows
\[
    |\u^{(i)*}\cR\u^{(i)}-\u^{(i)*}\cR^{(i)}\u^{(i)}|\prec N^\delta\frac{e^{-h/2}\u^{(i)*}\Im\cR\u^{(i)}}{N\eta}.
\]
It follows that
\begin{equation}\label{eq:leave-out-u-i}
    N^{-1/2}\|B_i\u\|_2\prec\sqrt{\frac{e^{-h/2}\u^{(i)*}\Im\cR\u^{(i)}}{N\eta}}.
\end{equation}
On the other hand, Schur's complement gives
\[
    \u^*\cR\u=\u^{(i)*}\cR\u^{(i)}+2\u[i]\frac{N^{-1/2}\x_i^*B_i\u}{\beta+N^{-1}\x_i^*A_i\x_i}-\frac{(\u[i])^2}{\beta+N^{-1}\x_i^*A_i\x_i}.
\]
Recall $\e_i^*\cR\e_i=-(\beta+N^{-1}\x_i^* A_i\x_i)^{-1}$. Taking the imaginary part of the preceding equation and rearranging gives
\[
    \u^{(i)*}\Im\cR\u^{(i)}\prec \frac{N^\delta}{\sqrt{e^{h/2}N\eta}}\sqrt{\u^{(i)*}\Im\cR\u^{(i)}}+\u^*\Im\cR\u+\e_i^*\Im\cR\e_i . 
\]
This is a quadratic inequality in $(\u^{(i)*}\Im\cR\u^{(i)})^{1/2}$. Solving it gives
\[
    \sqrt{\u^{(i)*}\Im\cR\u^{(i)}}\prec\frac{N^\delta}{\sqrt{e^{h/2}N\eta}}+\sqrt{\frac{N^{2\delta}}{e^{h/2}N\eta}+\u^*\Im\cR\u+\e_i^*\Im\cR\e_i} .
\]
Squaring and simplifying then gives
\[
    \u^{(i)*}\Im\cR\u^{(i)}\prec \u^*\Im\cR\u+\e_i^*\Im\cR\e_i+ N^\delta\sqrt{\frac{e^{-h/2}\u^*\Im\cR\u}{N\eta}}+N^\delta\sqrt{\frac{e^{-h/2}\e_i^*\Im\cR\e_i}{N\eta}}+N^\delta\Psi_h.
\]
Substituting this bound into \eqref{eq:leave-out-u-i} proves the first claim. For the second claim, the definition of the Frobenius norm gives
\[
    N^{-1}\|A_i\|_F\leq N^{-1/2}\max_{j=1}^{n+N}\|B_i\e_j\|_2 . 
\]
This result follows directly from the first claim. The corresponding weighted estimates follow from the same argument as in Lemma \ref{lem:ward-estimate}, with $\beta=C_0\delta$, so we omit the details.
\end{proof}

The next lemma uses the coarse estimate from one spectral-scale bootstrap to verify the hypotheses needed for the zag step in the next bootstrap. Its high-level structure is similar to that of \cite[Lemma~4.12]{fan2026anisotropic}.
\begin{lemma}\label{lem:bootstrap-in-zag}
Suppose Assumptions \ref{ass:A1} and \ref{ass:A2} hold, $\bD$ is a regular domain, and $0\leq h,t\leq 5\log(N)$ are such that $t(\Lambda_h+1)\leq CN^\delta$ for all $z\in\bD$. For $l\in\N$, define the domain
\[
    \bD_l:=\{z\in\bD:\Im z\in[N^{-l\delta},1],N\eta\Im\widetilde m_0(z)\geq N^{3C_0\delta}\}.
\]
Then the following statements hold:
\begin{enumerate}[label=(\alph*)]
    \item Suppose that, uniformly over all $z\in\bD_0$ and all deterministic $\u\in\R^{n+N}$, we have
    \[
        |\u^*(\cR_{h,t}-M_h)\u|\prec N^{C_0\delta}\Psi_h,
    \]
    then, uniformly over all $z\in\bD_0$, all deterministic unit vectors $\u\in\R^{n+N}$, and all $0\leq s\le t$,
    \[
        |\u^*(\cR_{h,s}-M_h)\u|\prec N^{C_0\delta}\Psi_h.
    \]
    \item For all $0\leq l\leq \delta^{-1}$ and all $h,t\geq 0$, if, uniformly over all $z\in\bD_l$, all deterministic unit vectors $\u\in\R^{n+N}$, and all $0\leq s\leq t$,
    \[
        |\u^*(\cR_{h,s}-M_h)\u|\prec N^{C_0\delta}\Psi_h,
    \]
    then, uniformly over all $z\in\bD_l$, all deterministic unit vectors $\u\in\R^{n+N}$, and all $0\leq s\leq t$,
    \[
        \u^*\Im\cR_{h,s}\u\prec e^{-h/2}\Im\widetilde m_0(z)+N^{C_0\delta}\Psi_h.
    \]
    \item For all $0\leq l\leq \delta^{-1}$, if, uniformly over all $z\in\bD_l$, all deterministic unit vectors $\u\in\R^{n+N}$, and all $0\leq s\leq t$,
    \[
        |\u^*(\cR_{h,s}-M_h)\u|\prec N^{C_0\delta}\Psi_h,\quad \u^*\Im\cR_{h,s}\u\prec e^{-h/2}\Im\widetilde m_0(z)+N^{C_0\delta}\Psi_h,
    \]
    and uniformly over all $z\in\bD_{l+1}$, all deterministic unit vectors $\u\in\R^{n+N}$,
    \[
        |\u^*(\cR_{h,t}-M_h)\u|\prec N^{C_0\delta}\Psi_h,
    \]
    then, uniformly over all $z\in\bD_{l+1}$, all deterministic unit vectors $\u\in\R^{n+N}$, and all $0\leq s\leq t$,
    \[
        |\u^*(\cR_{h,s}-M_h)\u|\prec N^{C_0\delta}\Psi_h.
    \]
\end{enumerate}
\end{lemma}
\begin{proof}
For (a), if $z\in\bD_0$, then $\eta=1$. Hence, Lemma \ref{lem:basic-estimates} implies, for any $i\in[N]$ and any $0\leq s\leq t$,
\[
    N^{-1/2}\|B_i\u\|_2,N^{-1}\|A_i\|_F\leq CN^{-1/2}\|\cR_{h,s}^{(i)}\|_\op\leq N^{-1/2}\|(\Im\cD_h)^{-1}\|_\op\leq Ce^{-h/2}N^{-1/2}\prec\Psi_h . 
\]
Here, we used $C\eta\leq\Im\widetilde m_0(z)$ to deduce $N^{-1/2}\leq \sqrt{\Im\widetilde m_0(z)/(N\eta)}$.
Similarly,
\[
    \|\Sigma^{1/2} B_i\u\|_2,N^{-1/2}\|\Sigma^{1/2} A_i\|_F\leq C\|\cR_{h,s}^{(i)}\|_\op\leq Ce^{-h/2}\prec \sqrt{\Lambda_h}.
\]
Moreover,
\[
    |\beta+N^{-1}\x_i^* A_i\x_i|^{-1}=|\e_i^*\cR_{h,s}\e_i|\leq \|\cR_{h,s}\|_\op\prec 1\leq N^\delta,
\]
We also have
\begin{align*}
    |\beta+N^{-1}\Tr\Sigma A_i|&\geq ||\beta+N^{-1}\x_i^* A_i\x_i|-|N^{-1}\Tr(\x_i\x_i^*-\Sigma)A_i||\\
    &\geq C+\Oprec(N^{-1/2})\geq C/2 . 
\end{align*}
Here, Assumption \ref{ass:A2} gives $|N^{-1}\Tr(\x_i\x_i^*-\Sigma)A_i|\prec N^{-1}\|A_i\|_F\prec N^{-1/2}$. Consequently, $|\beta+N^{-1}\Tr\Sigma A_i|^{-1}\leq 2/C\prec N^\delta$. Proposition \ref{prop:full-block-zag} now proves (a).

For (b), the hypothesis of the lemma and Lemma \ref{lem:basic-estimates} give, uniformly over all $z\in\bD_l$ and $0\leq s\leq t$,
\[
    \u^*\Im\cR_{h,s}\u\prec\u^*\Im M_h\u+N^{C_0\delta}\Psi_h\prec e^{-h/2}\Im\widetilde m_0(z)+N^{C_0\delta}\Psi_h . 
\]
Here, we used $M_h=e^{-h/2}M_*$.

For (c), let $z=E+i\eta\in\bD_{l+1}$. Then $\eta\geq N^{-\delta(l+1)}$, and by construction there exists $z'=E+i\eta'\in\bD_l$ such that $\eta'/\eta\leq N^\delta$. Lemma \ref{lem:Herglotz-function} and the hypothesis of the lemma then give, uniformly over all $z\in\bD_{l+1}$, all $0\le s\leq t$, and all deterministic unit vectors $\u\in\R^{n+N}$,
\begin{align*}
    |\u^*\cR_{h,s}(z)\u|&\prec|\u^*\cR_{h,s}(z')\u|+|N^\delta-1|\u^*\Im\cR_{h,s}(z')\u\\
    &\prec N^\delta|\u^*M_h(z')\u|+N^{(C_0+1)\delta}\Psi_h(z')\prec N^\delta.
\end{align*}
Here, we used $|\u^*M_h(z')\u|=e^{-h/2}|\u^*M_*(z')\u|\leq C$, $\Psi_h\prec (N\eta)^{-1/2}\leq N^{-\tau/2}$, and $\delta\leq \tau/(10C_0)\leq\tau/1000$. In particular, this implies
\[ 
    |\e_i^*\cR_{h,s}(z)\e_i|=|\beta+N^{-1}\x_i^*A_i\x_i|^{-1}\prec N^\delta .
\]
Similarly,
\begin{align*}
    \u^*\Im\cR_{h,s}(z)\u&\leq \frac{\eta'}{\eta}\u^*\Im\cR_{h,s}(z')\u\prec N^\delta e^{-h/2}\Im\widetilde m_0(z')+N^{(C_0+1)\delta}\Psi_h(z')\\
    &\le N^{2\delta}e^{-h/2}\Im\widetilde m_0(z)+N^{(C_0+1)\delta}\Psi_h(z) . 
\end{align*}
In the last step, we used a version of Lemma \ref{lem:Herglotz-function} for $\widetilde m_0$, which holds because the function $\eta\rightarrow \eta\Im\widetilde m_0$ is increasing and the function $\eta\rightarrow\Im\widetilde m_0/\eta$ is decreasing. Consequently,
\[
    \frac{e^{-h/2}\u^*\Im\cR_{h,s}(z)\u}{N\eta}\prec N^{2\delta}\frac{e^{-h}\Im\widetilde m_0(z)}{N\eta}+N^{(C_0+1)\delta}\Psi_h(z)^2\prec N^{(C_0+1)\delta}\Psi_h(z)^2 .
\]
Moreover,
\[
    \frac{e^{-h/2}\u^*\Im\cR_{h,s}(z)\u}{\eta_h}\prec N^{2\delta}\Lambda_h+\frac{e^{-h/2}}{\eta_h}N^{(C_0+1)\delta}\Psi_h\prec N^{2\delta}\Lambda_h .
\]
In the last step, we used $N^{2C_0\delta}/(N\eta)\leq\Im\widetilde m_0(z)$ so that $N^{C_0\delta}\Psi_h\prec e^{-h/2}\Im\widetilde m_0(z)$. Combining these bounds with Lemma \ref{lem:minor-estimates} gives, uniformly over all $i\in[N]$, all $z\in\bD_{l+1}$, all $0\leq s\leq t$, and all deterministic unit vectors $\u\in\R^{n+N}$,
\[
    N^{-1/2}\|B_i\u\|_2,N^{-1}\|A_i\|_F\prec N^{(C_0+1)\delta/2}\Psi_h,\quad \|\Sigma^{1/2}B_i\u\|_2,N^{-1/2}\|\Sigma^{1/2}A_i\|_F\prec N^\delta\sqrt{\Lambda_h}.  
\]
The same argument as in the proof of (a) shows $|\beta+N^{-1}\Tr\Sigma A_i|^{-1}\prec N^\delta$. Applying Proposition \ref{prop:full-block-zag} with $\Phi_1:=N^{(C_0+1)\delta/2}\Psi_h$ and $\Phi_2:=N^{C_0\delta}\Psi_h$ then proves (c).
\end{proof}

We now combine the spectral bootstraps with the zig--zag steps. Let $K:=\min\{k\in\N:N^{k\delta-1}\geq 2\log(N)\}$. Then $K\leq 2\delta^{-1}$. We also define the zig--zag lattice
\[
    H_0:=h_0=3\log(N),\quad h_1=N^{(K-1)\delta}N^{-1},\quad ...\quad , h_{K-1}=N^{-1+\delta},\quad h_K=0,
\]
\[
    t_0=0,\quad t_j=h_{j-1}-h_j,1\leq j\leq K.
\]
Thus, for all $1\leq i\leq K-1$, we have the bounds
\[
    h_i\leq C,\quad t_i\leq 2N^\delta h_i,\quad t_K=h_{K-1}=N^{-1+\delta}.
\]
\begin{lemma}\label{lem:zig-zag-and-bootstraps}
Suppose Assumptions \ref{ass:A1} and \ref{ass:A2} hold, and let $\bD$ be a regular domain. For $0\leq l\leq \delta^{-1}$, let $\bD_l$ be defined as in Lemma \ref{lem:bootstrap-in-zag}. Then the following hold:
\begin{enumerate}[label=(\alph*)]
    \item For all $0\leq k\leq K$, uniformly over all $z\in\bD_0$, all deterministic unit vectors $\u\in\R^{n+N}$, and all $0\leq s\leq t_k$,
    \[
        |\u^*(\cR_{h_k,s}-M_{h_k})\u|\prec N^{C_0\delta}\Psi_{h_k}.
    \]
    \item For all $0\leq l\leq \delta^{-1}$, if, for all $0\leq k\leq K$, uniformly over all $z\in\bD_l$, all deterministic unit vectors $\u\in\R^{n+N}$, and all $0\leq s\leq t_k$,
    \[
        |\u^*(\cR_{h_k,s}-M_{h_k})\u|\prec N^{C_0\delta}\Psi_{h_k},
    \]
    then, for all $0\leq k\leq K$, uniformly over all $z\in\bD_{l+1}$, all deterministic unit vectors $\u\in\R^{n+N}$, and all $0\leq s\leq t_k$,
    \[
        |\u^*(\cR_{h_k,s}-M_{h_k})\u|\prec N^{C_0\delta}\Psi_{h_k}.
    \]
\end{enumerate}
\end{lemma}
\begin{proof}
We begin by showing the following bounds for all $1\leq k\leq K$ and all $z\in\bD$ satisfying $N\eta\Im\widetilde m_0(z)\geq N^{3C_0\delta}$:
\[
    \int_0^{h_{k-1}-h_k}\Lambda_{h_{k-1}-t}\de t\leq C\log(N),\quad t_k(\Lambda_{h_k}+1)\leq CN^\delta.
\]
For any $z=E+i\eta\in\bD$, we can bound $\Lambda_h$ from above as
\[
    \Lambda_h=\frac{e^{-h}\Im\widetilde m_0(z)}{\eta+(1-e^{-h})\Im\widetilde m_0(z)}\leq \begin{cases}
        Ce^{-h}+C/h, & \text{if $h>\eta$,}\\
        C/\eta, & \text{if $0\leq h\leq\eta$}.
    \end{cases}
\]
Indeed, if $0\leq h\leq \eta$, then
\[
    \Lambda_h\leq Ce^{-h}\frac{\Im\widetilde m_0(z)}{\eta}\leq C\frac{1}{\eta}.
\]
If $h>\eta$, then
\[
    \Lambda_h\leq \frac{e^{-h}}{1-e^{-h}}\leq \frac{e^{-h}(1+h)}{h}\leq e^{-h}+\frac{1}{h}.
\]
Here, for all $h\geq 0$, we used $e^h\geq 1+h$, which gives $1-e^{-h}\geq h/(h+1)$. For the integral bound, we have
\begin{align*}
    \int_0^{h_{k-1}-h_k}\Lambda_{h_{k-1}-t}\de t&=\int_{h_k}^{h_{k-1}}\Lambda_{h}\de h\leq \int_0^{N}\Lambda_h\de h\leq\int_0^{\eta}\frac{1}{\eta}\de h+\int_{\eta}^{N}\left(e^{-h}+\frac{1}{h}\right)\de h\\
    &\leq 1+e^{-\eta}-e^{-N}+\log(N)+\log(\eta^{-1})\leq C\log(N) .
\end{align*}
Here, we used $\eta\geq N^{-1+\tau}$. For the second bound, suppose first that $h_k>\eta$. Then $k\leq K-2$, and
\[
    t_k\Lambda_{h_k}\leq Ct_ke^{-h_k}+Ct_k/h_k\leq CN^\delta+CN^\delta h_k/h_k\leq CN^\delta .
\]
Here, we used $t_k\leq 2N^\delta h_k\leq CN^\delta$. If instead $h_k\leq \eta$, then
\[
    t_k\Lambda_{h_k}\leq Ct_k/\eta\leq C\max\{2N^\delta h_k,N^{-1+\delta}\}/\eta\leq CN^\delta .
\]
Here, we applied $N^{-1+\delta}/\eta\leq N^{\delta-\tau}$. Combining this estimate with $t_k\leq 2N^\delta$ shows $t_k(\Lambda_{h_k}+1)\leq CN^\delta$.

We now prove (a) and (b). For (a), we first apply Lemma \ref{lem:zig-initialization}. We then iterate Proposition \ref{prop:full-block-zig} for the pair $(h_{k-1},h_k)$ and Lemma \ref{lem:bootstrap-in-zag} (a) for the pair $(t_k,h_k)$; this proves the claim.

For (b), the hypothesis of the lemma and Lemma \ref{lem:bootstrap-in-zag} (b) show that, for all $0\leq k\leq K$, uniformly over all $z\in\bD_l$, all deterministic unit vectors $\u\in\R^{n+N}$, and all $0\leq s\leq t_k$,
\[
    \u^*\Im\cR_{h_k,s}\u\prec e^{-h_k/2}\Im\widetilde m_0(z)+N^{C_0\delta}\Psi_{h_k}.
\]
Applying the same initialization-and-comparison argument on $\bD_{l+1}$---namely, Lemma \ref{lem:zig-initialization} followed by iterative applications of Proposition \ref{prop:full-block-zig} and Lemma \ref{lem:bootstrap-in-zag} (c)---proves the claim.
\end{proof}

\begin{proof}[Proof of Theorem \ref{thm:generalized-local-law} and Theorem \ref{thm:target}]
For any regular domain $\bD$, consider the smaller domain $\hat\bD:=\{z\in\bD:N\eta\Im\widetilde m_0(z)\geq N^{3C_0\delta}\}$. Since the function $\eta\rightarrow\eta\Im\widetilde m_0$ is increasing, $\hat\bD$ is also a regular domain. Thus, applying Lemma \ref{lem:zig-zag-and-bootstraps} (a) and then iteratively applying Lemma \ref{lem:zig-zag-and-bootstraps} (b) shows that, for all $0\leq k\leq K$, uniformly over all $z\in\hat\bD$, all deterministic unit vectors $\u\in\R^{n+N}$, and all $0\leq s\leq t_k$,
\[
    |\u^*(\cR_{h_k,s}-M_{h_k})\u|\prec N^{C_0\delta}\Psi_{h_k}.
\]
In particular, at $h_K=0$ and $s=0$, this gives
\[
    |\u^*(\cR_{0,0}-M_*)\u|\prec N^{C_0\delta}\Psi.
\]
To extend this result to all of $\bD$, let $z=E+i\eta\in\bD$ satisfy $N\eta\Im\widetilde m_0(z)<N^{3C_0\delta}$. By continuity, there exists $\eta'>\eta$ such that
\[
    N\eta'\Im\widetilde m_0(E+i\eta')=N^{3C_0\delta}\iff \eta'\Im\widetilde m_0(E+i\eta')=\frac{N^{3C_0\delta}}{N}.
\]
Let $z'=E+i\eta'$. Then $z'\in\hat\bD$, and Lemma \ref{lem:Herglotz-function} gives
\begin{align*}
    |\u^*(\cR_{0,0}(z)-M_*(z))\u|
    &\leq |\u^*(\cR_{0,0}(z')-M_*(z'))\u|
    +|\u^*(\cR_{0,0}(z)-\cR_{0,0}(z'))\u|\\
    &\quad+|\u^*(M_*(z)-M_*(z'))\u|\\
    &\prec N^{C_0\delta}\Psi(z')+\frac{\eta'-\eta}{\eta}\u^*\Im\cR_{0,0}(z')\u
    +\frac{\eta'-\eta}{\eta}\u^*\Im M_*(z')\u\\
    &\prec N^{C_0\delta}\Psi(z')+\frac{\eta'}{\eta}\u^*\Im\cR_{0,0}(z')\u
    +\frac{\eta'}{\eta}\u^*\Im M_*(z')\u\\
    &\prec N^{C_0\delta}\Psi(z')+2\frac{\eta'}{\eta}\Im\widetilde m_0(z')
    +\frac{\eta'}{\eta}N^{C_0\delta}\Psi(z'),
\end{align*}
where we used $\u^*\Im\cR_{0,0}(z')\u\prec \u^*\Im M_*(z')\u+N^{C_0\delta}\Psi(z')$, $\eta'>\eta$, and $\u^*\Im M_*(z')\u\prec \Im\widetilde m_0(z')$. Since the function $\eta\rightarrow\Im\widetilde m_0/\eta$ is decreasing and $\eta'>\eta$, we have $\Psi(z')\leq \Psi(z)$. The defining property of $\eta'$ also shows that $N^{C_0\delta}\Psi(z')\prec \Im\widetilde m_0(z')$ and $\eta'\Im\widetilde m_0(z')=N^{3C_0\delta}/N$. Combining these estimates gives
\[
    |\u^*(\cR_{0,0}(z)-M_*(z))\u|\prec N^{C_0\delta}\Psi(z)+\frac{\eta'\Im\widetilde m_0(z')}{\eta}\prec N^{C_0\delta}\Psi(z)+\frac{N^{3C_0\delta}}{N\eta}\prec N^{3C_0\delta}\Psi(z).
\]
Since $\delta$ is small but arbitrary, this shows that $|\u^*(\cR_{0,0}(z)-M_*(z))\u|\prec\Psi(z)$ uniformly over $\bD$. Splitting into real and imaginary parts and applying the polarization identity proves Theorem \ref{thm:generalized-local-law}. Finally, since $\cR_{0,0},M_*$ and $\Psi$ are all $N^5$-Lipschitz over $\bD$, an $N^{20}$-net, together with a union bound and the equivalence \eqref{eq:linear-pencil-equivalence}, proves Theorem \ref{thm:target}.
\end{proof}

\section{Averaged local law along the characteristic}\label{sec:averaged}
In this section, we prove Proposition \ref{prop:averaged zig} by adapting the argument for the averaged local law in \cite[Theorem 2.5]{fan2026anisotropic}. Because the two proofs are closely aligned and often agree verbatim apart from cosmetic changes, we give detailed outlines of the corresponding lemmas, with precise citations, and present the technical modifications in full only where needed.

\subsection{Stability of the fixed-point equation}
For $z=E+i\eta\in\C^+$, let $w=\sqrt z\in\C^+$ denote the principal square root, and define
\begin{equation}\label{eq:deterministic-data}
 q_0(z)=w\widetilde m_0(z),\qquad
s_0(z)=\frac1N\Tr\Sigma (-wI-q_0(z)\Sigma)^{-1}.
\end{equation}
For $r\in[0,1]$, set
\begin{equation}\label{eq:interpolated-data}
 \widehat Z_r(z)=wI+(1-r)q_0(z)\Sigma,\qquad
 \widehat\beta_r(z)=w+(1-r)s_0(z).
\end{equation}
For $q\in\C^+$, define
\begin{equation}\label{eq:maps}
 F_{r,z}(q)=-\frac1N\Tr\Sigma(\widehat Z_r(z)+rq\Sigma)^{-1},
 \qquad
 H_{r,z}(q)=-\frac1q-\widehat\beta_r(z)-rF_{r,z}(q).
\end{equation}
By the deformed MP fixed-point relation (\ref{eq:MP}), we have
\begin{equation}\label{eq:fixed-root}
 s_0(z)=-w-\frac{1}{q_0(z)},\quad F_{r,z}(q_0(z))=s_0(z),\qquad H_{r,z}(q_0(z))=0.
\end{equation}

The stability of the fixed-point map $H_{r,z}(q)$ is inherited from that of the MP fixed point $z_0(m)-z$. The following theorem summarizes this connection.
\begin{theorem}[Strong stability of $H_r$]\label{thm:strong-stability}
Suppose Assumption \ref{ass:A1} holds and $\bD$ is a regular domain. For any $z\in\bD$, define the lattice
\[
    L(z):=\{z\}\cup\{\zeta\in\bD:\Re\zeta=\Re z,\Im\zeta\in[\Im z,1]\cap\{N^{-5},2N^{-5},3N^{-5},...,1\}\}.
\]
Suppose $q:\bD\rightarrow\C^+$ is a $N^2$-Lipschitz function and $\Delta:\bD\rightarrow(0,\infty)$ is a deterministic error-control parameter such that, for all $z\in\bD$,
\[
    N^{-1}\leq\Delta(z)\leq (\log N)^{-1},\quad  |\Delta(z)-\Delta(z')|\leq N^2|z-z'|,
\]
and $\eta\rightarrow\Delta(E+i\eta)$ is nonincreasing for fixed $E$. Then there exists a constant $C>0$, depending only on the regularity constants of
$\bD$, such that, for all sufficiently large $N$, the following holds
uniformly for $r\in[0,1]$:
If, for a target $z=E+i\eta\in\bD$,
\begin{equation}\label{eq:residual-on-lattice}
 |H_{r,\zeta}(q(\zeta))|\le\Delta(\zeta)
 \qquad\text{for every }\zeta\in L(z),
\end{equation}
then
\begin{equation}\label{eq:main-bound}
 |q(z)-q_0(z)|
 \le
 \frac{C\Delta(z)}
 {(1-r)+\sqrt{\kappa+\eta}+\sqrt{\Delta(z)}}.
\end{equation}
\end{theorem}
\begin{proof}
Since $\bD$ is a regular domain, we have, for all $z\in\bD$,
\begin{equation}\label{eq:regular-bounds}
 |q_0(z)|\asymp1,\qquad
 \|(wI+q_0(z)\Sigma)^{-1}\|_{\op}\le C,\qquad
 \|\Sigma\|_{\op}\le C.
\end{equation}
We use these bounds throughout without further comment. Write $u(z)=q(z)-q_0(z)$. We divide the proof into three steps.

\subsection*{Step 1: initialization at height one}
Fix $\zeta=E+i$ and define
\[
 x=(1-r)q_0+rq,\qquad
 P_x=(wI+x\Sigma)^{-1},\qquad
 P_0=(wI+q_0\Sigma)^{-1},
\]
and set $A_x=N^{-1}\Tr\Sigma P_x\Sigma P_0$, where we suppress the dependence on $\zeta$ to avoid notational clutter. The identity $A^{-1}-B^{-1}=A^{-1}(B-A)B^{-1}$ gives
\[
 F_{r}(q)-F_{r,\zeta}(q_0)=r(q-q_0)A_x.
\]
Using \eqref{eq:fixed-root}, we obtain
\begin{align}\label{eq:exact-difference}
 H_r(q)&=-\frac{1}{q}-\hat\beta_r-rF_r(q)=-\frac{1}{q}-\hat\beta_r-rF_r(q_0)-r^2(q-q_0)A_x\\
 &=-\frac{1}{q}-w-(1-r)s_0-rs_0-r^2(q-q_0)A_x\\
 &=-\frac{1}{q}+\frac{1}{q_0}-r^2(q-q_0)A_x
 =\frac{q-q_0}{qq_0}
   \bigl(1-r^2qq_0 A_x\bigr).
\end{align}

We claim that, at height one,
\begin{equation}\label{eq:strict-cross-contraction}
 r^2|qq_0 A_x|\le 1-c
\end{equation}
for a fixed $0<c<1$. Indeed, since $P_x,P_0$ are spectral functions of $\Sigma$, direct diagonalization gives
\begin{equation}\label{eq:imaginary-Fr}
 \Im F_r(q)
 \ge(\Im x)\frac1N\Tr\Sigma P_x\Sigma P_x^*
 \ge r(\Im q)\frac1N\Tr\Sigma P_x\Sigma P_x^*.
\end{equation}
Recall the equation
\begin{equation}\label{eq:q-contraction}
 H_r(q)=-\frac{1}{q}-\hat\beta_r-rF_r(q)\iff -\frac{1}{q}=\hat\beta_r+rF_r(q)+H_r(q).
\end{equation}
At height one, $\Im w=\Im\sqrt{\zeta}\ge c$ and $|H_r(q)|\le(\log N)^{-1}$, so
\[
    \frac{1}{|q|}\geq\frac{\Im q}{|q|^2}
 \ge c+r^2(\Im q)\frac1N\Tr\Sigma P_x\Sigma P_x^*\geq c.
\]
Moreover, $\|P_x\|_{\op}\le(\Im w)^{-1}\le C$, so
\eqref{eq:q-contraction} yields $c\le|q|\le C$. Consequently,
\[
 r^2|q|^2\frac1N\Tr\Sigma P_x\Sigma P_x^*\le1-c'.
\]
The same argument with $q=q_0$ gives
$r^2|q_0|^2N^{-1}\Tr\Sigma P_0\Sigma P_0^*\le1-c'$. Finally, Cauchy--Schwarz gives
\[
    r^2|qq_0||A_x|\leq \sqrt{r^2|q|^2N^{-1}\Tr\Sigma P_x\Sigma P_x^*}\sqrt{r^2|q_0|^2N^{-1}\Tr\Sigma P_0\Sigma P_0^*}\leq 1-c'.
\]
This proves \eqref{eq:strict-cross-contraction}. Since $q$ and $q_0$ are uniformly
bounded at height one, \eqref{eq:exact-difference} and
\eqref{eq:residual-on-lattice} imply
\begin{equation}\label{eq:top-anchor}
 |q(E+i)-q_0(E+i)|\le C\Delta(E+i).
\end{equation}

\subsection*{Step 2: the uniform normal form}
For a general $\zeta\in\C^+$, the exact resolvent expansion gives
\begin{equation}\label{eq:normal-form}
 H_r(q_0+u)=(q_0^{-2}-r^2A_0) u+\alpha_r(u)u^2,
\end{equation}
where
\[
     A_0=\frac1N\Tr\Sigma P_0\Sigma P_0,\quad \alpha_r(u)=-\frac1{q_0^2(q_0+u)}
 +\frac{r^3}{N}\Tr\!\left[
 \Sigma^3(wI+q_0\Sigma)^{-2}
 (wI+(q_0+ru)\Sigma)^{-1}
 \right].
\]

Put $c_0=q_0^2A_0$ and
$B_0=N^{-1}\Tr\Sigma P_0\Sigma P_0^*$.  The fixed-point
equation for $q_0$ and the imaginary-part calculation used above give
\[
 |A_0|\le B_0,\qquad r^2|q_0|^2B_0\le1,
\]
Since the bound holds for any $r\in[0,1]$ and $c_0$ is independent of $r$, this also shows $|c_0|\leq 1$. Moreover, we have the identity
\begin{align*}
    |1-r^2c_0|^2&=(1-r^2+r^2(1-c_0))(1-r^2+r^2(1-\bar c_0))\\
    &=(1-r^2)^2+r^4|1-c_0|^2+2r^2(1-r^2)(1-\Re c_0)\\
    &=(1-r^2)^2+r^4|1-c_0|^2+r^2(1-r^2)(|1-c_0|^2+1-|c_0|^2)\\
    &=(1-r^2)^2+r^2|1-c_0|^2+r^2(1-r^2)(1-|c_0|^2),
\end{align*}
which implies
\begin{equation}\label{eq:noncancellation-scale}
 |1-r^2c_0|\asymp(1-r)+|1-c_0|.
\end{equation}
At $r=1$, writing $q=wm$, we have
\[
 wH_1(wm)=z_0(m)-z,
 \qquad
 1-c_0
 =\widetilde m_0(z)^2z_0'(\widetilde m_0(z)).
\]
Since $\bD$ is a regular domain, fixed-point stability implies
\begin{equation}\label{eq:fixed-point stability}
    |1-c_0|\asymp\sqrt{\kappa+\eta},\quad |z_0''(\widetilde m_0(z))|\geq c \text{ if $\kappa+\eta\leq c_*$}.
\end{equation}
In either case, we deduce
\begin{equation}\label{eq:beta-scale}
 |q_0^{-2}-r^2A_0|=|q_0|^{-2}|1-r^2c_0|\asymp \underbrace{(1-r)+\sqrt{\kappa+\eta}}_{\Gamma_r(z)}.
\end{equation}

Since $\bD$ is a regular domain, $|\alpha_r(0)|\leq C$; moreover, for $|u|$ sufficiently small, $\alpha_r(u)$ is uniformly Lipschitz in $u$ for some constant $C>0$ and all $r\in[0,1]$. We may therefore choose a sufficiently small constant $\rho>0$ such that
\[
    |\alpha_r(u)|\leq C,\forall |u|\leq \rho.
\]
Now suppose that, for some constant $\Gamma_*>0$, we have $\Gamma_r\geq \Gamma_*$. Then (\ref{eq:normal-form}) gives
\begin{equation}\label{eq:large gamma}
    |H_r(q_0+u)|\geq |q_0|^{-2}|1-r^2c_0||u|-|\alpha_r(u)||u|^2\geq c\Gamma_*|u|-C|u|^2\geq \frac{c\Gamma_*}{2}|u|.
\end{equation}

If $\Gamma_r$ is small, then both $1-r$ and $\sqrt{\kappa+\eta}$ must be small, placing us in a neighborhood of a regular edge. At $u=0$, we have
\begin{equation}\label{eq:curvature}
 \alpha_r(0)
 =-q_0^{-3}
 +\frac{r^3}{N}\Tr\!\left[
 \Sigma^3(wI+q_0\Sigma)^{-3}
 \right].
\end{equation}
Moreover,
\[
 \alpha_1(0)
 =\frac12w^{-3}z_0''(\widetilde m_0(\zeta)),
 \qquad
 |\alpha_r(0)-\alpha_1(0)|\le C(1-r).
\]
By (\ref{eq:fixed-point stability}), $|z_0''(\widetilde m_0(\zeta))|\geq c$. Since $\alpha_r(u)$ is uniformly Lipschitz in $u$ with constant $C>0$, it follows that there are fixed
$\Gamma_*,\rho>0$ such that
\begin{equation}\label{eq:alpha-scale}
 c\le|\alpha_r(u)|\le C
 \qquad
 \text{if }\Gamma_r(\zeta)<\Gamma_*\text{ and }|u|\leq\rho.
\end{equation}

Consequently, when $\Gamma_r<\Gamma_*$,
\begin{equation}\label{eq:factorization}
    c\le|\alpha_r(u)|\le C,
 \qquad
 |q_0^{-2}-r^2A_0|\asymp\Gamma_r(\zeta),
\end{equation}
for $|u|\le\rho$.

\subsection*{Step 3: vertical continuation}
We use a radial version of the vertical continuation argument of
\cite[Definitions~5.4 and~A.2 and Lemma~A.5]{knowles2017anisotropic};
for the full $L(z)$ formulation, see
\cite[Definition~C.1 and Lemma~C.3(b), equations~(87)--(97)]
{fan2022tracy}.
The required inputs are the top-height estimate
\eqref{eq:top-anchor} and the coefficient estimates established above.

Enumerate $L(z)$ from top to bottom as
$z_j=E+i\eta_j$, $0\leq j\leq M$, and set
\[
    u_j=q(z_j)-q_0(z_j),\qquad
    \Delta_j=\Delta(z_j),\qquad
    \Gamma_j=\Gamma_r(z_j).
\]
The Lipschitz assumption on $q$ and Lemma
\ref{lem:Herglotz-function} give
\[
    |u_j-u_{j-1}|
    \leq CN^2(\eta_{j-1}-\eta_j)
    \leq CN^{-3}.
\]
Moreover, $\Delta_j$ is nondecreasing and $\Gamma_j$ is
nonincreasing in $j$.

Write $a_j=q_0(z_j)^{-2}-r^2A_0(z_j)$, and let $\alpha_j(u)$
denote the coefficient $\alpha_r(u)$ evaluated at $z_j$.
By \eqref{eq:normal-form} and the preceding coefficient bounds,
\[
    H_{r,z_j}(q_0(z_j)+u)=a_ju+\alpha_j(u)u^2,
    \qquad
    |a_j|\asymp\Gamma_j,\qquad
    |\alpha_j(u)|\leq C
\]
for $|u|\leq\rho$.
Choose a fixed $c_1>0$ sufficiently small.
Whenever $|u_j|\leq\rho$ and $|u_j|\leq c_1\Gamma_j$, we have
\[
    |a_j+\alpha_j(u_j)u_j|\geq c\Gamma_j.
\]
The residual bound \eqref{eq:residual-on-lattice} therefore
implies the radial alternative
\begin{equation}\label{eq:radial-alternative}
    |u_j|\leq C_1\frac{\Delta_j}{\Gamma_j}
    \qquad\text{or}\qquad
    |u_j|\geq c_1\Gamma_j.
\end{equation}
Enlarge $C_1$ if necessary to cover the top-height estimate,
and choose a fixed $c_2>0$ such that $C_1c_2\leq c_1/4$.

We now proceed by induction. At height one,
$\Gamma_0\asymp1$ and $\Delta_0=o(1)$, so
\eqref{eq:top-anchor} gives
\[
    |u_0|\leq C_1\frac{\Delta_0}{\Gamma_0}.
\]
In particular, \eqref{eq:main-bound} holds at $j=0$.
Suppose it holds at $j-1$. Then
$|u_{j-1}|\leq C\sqrt{\Delta_{j-1}}=o(1)$, and the increment
bound gives $u_j=o(1)$. Thus all local expansions apply at
$z_j$ for sufficiently large $N$.

First suppose that $\Delta_j\leq c_2\Gamma_j^2$.
By monotonicity, the same inequality holds at every preceding
lattice point. Maintaining the first alternative in
\eqref{eq:radial-alternative}, we obtain
\[
    |u_{j-1}|
    \leq C_1\frac{\Delta_{j-1}}{\Gamma_{j-1}}
    \leq C_1\frac{\Delta_j}{\Gamma_j}
    \leq \frac{c_1}{4}\Gamma_j.
\]
Since $\Delta_j\geq N^{-1}$, this regime also implies
$\Gamma_j\geq c_2^{-1/2}N^{-1/2}$.
Consequently,
\[
    |u_j|\leq \frac{c_1}{4}\Gamma_j+CN^{-3}
    <c_1\Gamma_j
\]
for sufficiently large $N$.
The second alternative in \eqref{eq:radial-alternative}
is therefore excluded, and
\[
    |u_j|\leq C_1\frac{\Delta_j}{\Gamma_j}.
\]

Now suppose that $\Delta_j>c_2\Gamma_j^2$.
Since $\Delta_j\leq(\log N)^{-1}$, we have
\[
    \Gamma_j^2<(c_2\log N)^{-1},
\]
so $\Gamma_j<\Gamma_*$ for sufficiently large $N$.
The lower bound in \eqref{eq:alpha-scale} is therefore
available. The normal form and the residual estimate give
\[
    c|u_j|^2
    \leq |H_{r,z_j}(q_0(z_j)+u_j)|+|a_ju_j|
    \leq \Delta_j+C\Gamma_j|u_j|.
\]
Solving this quadratic inequality yields
\[
    |u_j|\leq C\bigl(\Gamma_j+\sqrt{\Delta_j}\bigr)
    \leq C\sqrt{\Delta_j}.
\]
By monotonicity, this regime persists at all subsequent
lattice points.

Combining the two regimes gives
\[
    |u_j|
    \leq C\frac{\Delta_j}{\Gamma_j+\sqrt{\Delta_j}}.
\]
This closes the induction and proves \eqref{eq:main-bound}.
\end{proof}

\subsection{Basic identities}
We now lay out the ingredients needed to prove Proposition \ref{prop:averaged zig}. We closely follow the method of \cite[Theorem 2.5, Section 4]{fan2026anisotropic}. Throughout this section, $X\in\R^{n\times N}$ denotes a matrix satisfying Assumptions \ref{ass:A1} and \ref{ass:A2}.
\begin{defi}[Minors]
For any $S\subseteq[N]$, let $X^{(S)}\in\R^{n\times N}$ be the matrix with columns
\[
    X^{(S)}\e_i=\begin{cases}
        \x_i & \text{if $i\notin S$}\\
        0 & \text{otherwise}.
    \end{cases}
\]
For a regular domain $\bD$, $z\in\bD$, $w=\sqrt{z}\in\C^+$, and $h\geq 0$, let 
\[
    \cR_h^{(S)}(z)=\begin{bmatrix}
        -Z_h & N^{-1/2}X^{(S)}\\
        N^{-1/2}X^{(S)*} & -\beta_h I_N
    \end{bmatrix}^{-1},\quad \hat\cR_h^{(S)}=e^{h/2}\cR_h^{(S)},\quad \hat Z_h=e^{-h/2}Z_h,\quad \hat\beta_h=e^{-h/2}\beta_h,
\]
\[
    Q_h^{(S)}(z)=[\hat\cR_h^{(S)}(z)]_{22}=e^{h/2}(N^{-1}X^{(S)*}Z_h^{-1}X^{(S)}-\beta_hI_N)^{-1},\quad q^{(S)}(z)=N^{-1}\Tr Q_h^{(S)}(z),
\]
\[
    G_h^{(S)}(z)=[\hat\cR_h^{(S)}(z)]_{11}=e^{h/2}(\beta_h^{-1}N^{-1}X^{(S)}X^{(S)*}-Z_h)^{-1},\quad s^{(S)}(z)=N^{-1}\Tr \Sigma G_h^{(S)}(z).
\]
The two identities follow from the Schur complement formula. We also define their deterministic equivalents
\[
    q_0(z)=w\widetilde m_0(z)=N^{-1}\Tr [M_*]_{22},\quad s_0(z)=N^{-1}\Tr\Sigma(-wI_n-q_0(z)\Sigma)^{-1}=N^{-1}\Tr\Sigma[M_*]_{11}.
\]
We often suppress the subscript $h$ and the parameter $z$ to avoid notational clutter.
\end{defi}

We begin with several basic identities that follow from the Schur complement formula.
\begin{lemma}[Resolvent identities]\label{lem:resolvent identities}
For any $S\subseteq[N]$:
\begin{enumerate}[label=(\alph*)]
    \item For any distinct $i,j\notin S$,
    \[
        Q_{ii}^{(S)}=-\frac{1}{\hat\beta+e^{-h}N^{-1}\x_i^*G^{(iS)}\x_i},\quad Q_{ij}^{(S)}=e^{-h}Q_{ii}^{(S)}Q_{jj}^{(iS)}\cdot N^{-1}\x_i^*G^{(ijS)}\x_j.
    \]
    \item For any $i,j,k\notin S$ (including $i=j$) with $k\notin\{i,j\}$,
    \[
        Q_{ij}^{(S)}=Q_{ij}^{(kS)}+\frac{Q_{ik}^{(S)}Q_{kj}^{(S)}}{Q_{kk}^{(S)}},\quad \frac{1}{Q_{ii}^{(S)}}=\frac{1}{Q_{ii}^{(kS)}}-\frac{Q_{ik}^{(S)}Q_{ki}^{(S)}}{Q_{ii}^{(kS)}Q_{ii}^{(S)}Q_{kk}^{(S)}}.
    \]
    \item (Sherman-Morrison) For any $i\notin S$,
    \[
        G^{(S)}=G^{(iS)}-\frac{e^{-h}N^{-1}G^{(iS)}\x_i\x_i^*G^{(iS)}}{\hat\beta+e^{-h}N^{-1}\x_i^*G^{(iS)}\x_i}=G^{(iS)}+e^{-h}Q_{ii}^{(S)}\cdot N^{-1}G^{(iS)}\x_i\x_i^*G^{(iS)}.
    \]
    \item (Ward identity) We have
    \[
        \|G^{(S)}\|_F^2\leq C\frac{\Im\Tr G^{(S)}}{\eta},\quad \|\Sigma^{1/2}G^{(S)}\|_F^2\leq C\frac{N\Im s^{(S)}}{\eta}.
    \]
    \item (Weighted-unweighted trace relation) We have
    \[
        N^{-1}\Tr[G-(-wI_n-q_0\Sigma)^{-1}]=w^{-1}\hat\beta_h(q-q_0)-w^{-1}(1-e^{-h})q_0(s-s_0).
    \]
\end{enumerate}
\end{lemma}
\begin{proof}
Parts (a) and (b) follow by applying the Schur complement formula to the linearized resolvent $e^{h/2}\cR_h^{(S)}(z)$. Part (c) is the standard Sherman--Morrison formula for a rank-one update of the matrix inverse $G^{(S)}$.

For part (d), Lemma \ref{lem:basic-estimates} gives
\begin{align*}
    \|G^{(S)}\|_F^2&=\Tr G^{(S)*}G^{(S)}\leq\Tr\hat\cR^{(S)*}\diag(I_n,0I_N)\hat\cR^{(S)}\\
    &=\Tr\hat\cR^{(S)}\hat\cR^{(S)*}\diag(I_n,0I_N)\\
    &\leq C\frac{\Tr\Im\hat\cR^{(S)}\diag(I_n,0I_N)}{\eta}=C\frac{\Tr\Im G^{(S)}}{\eta},
\end{align*}
and the same argument shows $\|\Sigma^{1/2}G^{(S)}\|_F^2\leq CN\Im s^{(S)}/\eta$.

For part (e), take the traces of the upper-left and lower-right blocks of
\[
    e^{-h/2}\begin{bmatrix}
        -Z_h & N^{-1/2}X\\
        N^{-1/2}X^* & -\beta_hI_N
    \end{bmatrix}\hat\cR_h=I_{n+N}.
\]
We have
\[
    -\Tr\hat Z_hG+e^{-h/2}N^{-1/2}\Tr X[\hat\cR_h]_{21}=n,\quad e^{-h/2}N^{-1/2}\Tr X^*[\hat\cR_h]_{12}-\hat\beta_h\Tr Q=N.
\]
Since $X$ is real and $\hat\cR_h$ is complex symmetric, subtracting the two equations and applying $\hat Z_h=wI_n+(1-e^{-h})q_0\Sigma$ gives
\[
    wN^{-1}\Tr G+(1-e^{-h})q_0s-\hat\beta_hq=1-\frac{n}{N}.
\]
Subtracting the corresponding deterministic identity gives
\begin{align*}
    &wN^{-1}\Tr [G-(-wI_n-q_0\Sigma)^{-1}]+(1-e^{-h})q_0(s-s_0)-\hat\beta_h(q-q_0)\\
    &=wN^{-1}\Tr(wI_n+q_0\Sigma)^{-1}+wq_0+1-\frac{n}{N}=0.
\end{align*}
Here, the last equality follows from the definitions of $\hat\beta_h,w,q_0$ and the deformed MP fixed-point relation (\ref{eq:MP}). Rearranging completes the proof.
\end{proof}

\begin{lemma}\label{lem:exact fixed point}
For any $z\in\C^+$ and $h\geq 0$, we have
\[
    s-\underbrace{N^{-1}\Tr\Sigma(-\hat Z-e^{-h}q\Sigma)^{-1}}_{F(q)}=-\frac{e^{-h}}{N}\sum_{i=1}^NQ_{ii}d_i.
\]
Here,
\[
    d_i=N^{-1}\x_i^*G^{(i)}(\hat Z+e^{-h}q\Sigma)^{-1}\Sigma\x_i-N^{-1}\Tr\Sigma G(\hat Z+e^{-h}q\Sigma)^{-1}\Sigma.
\]
Moreover, using the fixed-point equation from (\ref{eq:fixed-root}) with $e^{-h}$ in place of $r$, we have
\begin{align*}
    H(q)&=-\frac{1}{q}-\hat{\beta}-e^{-h}F(q)\\
    &=-\frac{e^{-h}}{q}\cdot \frac{1}{N}\sum_{i=1}^N(e^{-h}Q_{ii}^2N^{-2}\x_i^*G^{(i)}\Sigma G^{(i)}\x_i-Q_{ii}N^{-1}\Tr(\x_i\x_i^*-\Sigma)G^{(i)})-\frac{e^{-2h}}{N}\sum_{i=1}^NQ_{ii}d_i.
\end{align*}
\end{lemma}
\begin{proof}
We begin by proving the formula for $s-F(q)$. We have the matrix identity
\[
    \begin{bmatrix}
        -\hat Z & e^{-h/2}N^{-1/2}X\\
        e^{-h/2}N^{-1/2}X^* & -\hat\beta
    \end{bmatrix}\hat\cR=I_{n+N}.
\]
The upper-left blocks of the two sides give
\[
    I_n=-\hat ZG+e^{-h/2}N^{-1/2}\sum_{i=1}^N\x_i\e_i^*[\hat\cR]_{21}=-\hat ZG-\frac{e^{-h}}{N}\sum_{i=1}^NQ_{ii}\x_i\x_i^*G^{(i)},
\]
The last equality follows by applying the Schur complement formula to $\hat\cR$ and using $\e_i^*[\hat\cR]_{21}=-e^{-h/2}N^{-1/2}Q_{ii}\x_i^*G^{(i)}$. Multiplying both sides on the left by $B=(\hat Z+e^{-h}q\Sigma)^{-1}\Sigma$ and taking the trace, while noting that both $\hat Z$ and $B$ are spectral functions of $\Sigma$, gives
\begin{align*}
    -F(q)&=-N^{-1}\Tr G\hat Z(\hat Z+e^{-h}q\Sigma)^{-1}\Sigma-\frac{e^{-h}}{N^2}\sum_{i=1}^NQ_{ii}\x_i^* G^{(i)}(\hat Z+e^{-h}q\Sigma)^{-1}\Sigma\x_i\\
    &=-s+e^{-h}qN^{-1}\Tr\Sigma G(\hat Z+e^{-h}q\Sigma)^{-1}\Sigma-\frac{e^{-h}}{N^2}\sum_{i=1}^NQ_{ii}\x_i^* G^{(i)}(\hat Z+e^{-h}q\Sigma)^{-1}\Sigma\x_i.
\end{align*}
Since $q=N^{-1}\sum Q_{ii}$, rearranging proves the first identity. For $H(q)$, Lemma \ref{lem:resolvent identities} gives
\begin{align*}
    &q-(-\hat\beta-e^{-h}s)^{-1}=\frac{1}{N}\sum_{i=1}^N(-\hat\beta-e^{-h}N^{-1}\x_i^*G^{(i)}\x_i)^{-1}-(-\hat\beta-e^{-h}s)^{-1}\\
    &=-\frac{e^{-h}}{(\hat\beta+e^{-h}s)}\cdot\frac{1}{N}\sum_{i=1}^NQ_{ii}(N^{-1}\x_i^*G^{(i)}\x_i-s)\\
    &=-\frac{e^{-h}}{(\hat\beta+e^{-h}s)}\cdot\frac{1}{N}\sum_{i=1}^NQ_{ii}(N^{-1}\Tr\Sigma G^{(i)}-s+N^{-1}\Tr(\x_i\x_i^*-\Sigma)G^{(i)})\\
    &=-\frac{e^{-h}}{(\hat\beta+e^{-h}s)}\cdot\frac{1}{N}\sum_{i=1}^N(-e^{-h}Q_{ii}^2N^{-2}\x_i^*G^{(i)}\Sigma G^{(i)}\x_i+Q_{ii}N^{-1}\Tr(\x_i\x_i^*-\Sigma)G^{(i)}).
\end{align*}
Here, the last equality follows from Lemma \ref{lem:resolvent identities} and the identity
\[
    s=N^{-1}\Tr\Sigma G=N^{-1}\Tr\Sigma G^{(i)}+e^{-h}Q_{ii}N^{-2}\x_i^*G^{(i)}\Sigma G^{(i)}\x_i.
\]
Thus, we may write
\begin{align*}
    H(q)&=-q^{-1}-\hat\beta-e^{-h}F(q)=-(q^{-1}+\hat\beta+e^{-h}s)+e^{-h}s-e^{-h}F(q)\\
    &=-\frac{(\hat\beta+e^{-h}s)}{q}(q+(\hat\beta+e^{-h}s)^{-1})+e^{-h}(s-F(q)).
\end{align*}
Combining this identity with the two preceding formulas completes the proof.
\end{proof}

\subsection{Resolvent bounds}
For each $z\in\C^+$, define the error-control parameters
\[
    \Gamma:=\max_{i,j=1}^N|Q_{ij}-q_0\1\{i=j\}|,\quad \Theta:=|q-q_0|,\quad \Psi_\Theta:=\sqrt{\frac{\Im\widetilde m_0+\Theta}{N\eta}}+\frac{1}{N\eta}.
\]
For a fixed constant $C_0>0$, define the ($z$-dependent) event
\[
    \Xi:=\{\Gamma\leq N^{-\tau/C_0}\}.
\]

\begin{lemma}[Global bounds]\label{lem:global-bounds}
Suppose Assumptions \ref{ass:A1} and \ref{ass:A2} hold, $\bD$ is a regular domain, and $h\geq 0$. Then, for any $i\in[N]$ and any $z=E+i\eta\in\bD$, we have
\[
    |q|\leq C\eta^{-1},\quad |Q_{ii}|\leq C\eta^{-1},\quad N^{-1/2}\|G^{(i)}\|_F\leq C\eta^{-1},
\]
\[
    \|(\hat Z+e^{-h}q\Sigma)^{-1}\|_\op\leq C\eta^{-1},\quad \|(\hat Z+e^{-h}q^{(i)}\Sigma)^{-1}\|_\op\leq C\eta^{-1}.
\]
Furthermore, uniformly over $i\in[N]$, $h\geq 0$, and $z\in\bD$,
\[
    |q|^{-1}\prec\eta^{-1},\quad |q-q^{(i)}|\prec N^{-1}\eta^{-5}.
\]
\end{lemma}
\begin{proof}
The first three bounds follow from $\|Q\|_\op\leq \|\hat\cR\|_\op,\|Q^{(i)}\|_\op\leq\|\hat\cR^{(i)}\|_\op\leq C\eta^{-1}$ by Lemma \ref{lem:basic-estimates}. To control $\|(\hat Z+e^{-h}q\Sigma)^{-1}\|_\op$, let $\sigma$ be any eigenvalue of $\Sigma$. Then
\[
    \Im w+(1-e^{-h})\sigma\Im q_0+e^{-h}\sigma\Im q\geq \Im w\geq C\eta,
\]
because $\Im q_0,\Im q\geq 0$. Hence $\|(\hat Z+e^{-h}q\Sigma)^{-1}\|_\op\leq C\eta^{-1}$, and similarly $\|(\hat Z+e^{-h}q^{(i)}\Sigma)^{-1}\|_\op\leq C\eta^{-1}$. To bound $|q|^{-1}$, observe that
\begin{align*}
    |q|&\geq \Im q=N^{-1}\Tr\Im Q=N^{-1}\Tr\diag(0I_n,I_N)\Im\hat\cR\diag(0I_n,I_N)\\
    &=N^{-1}\Tr\diag(0I_n,I_N)\hat\cR\diag(\Im\hat Z,\Im\hat\beta I_N)\hat\cR^*\diag(0I_n,I_N)\\
    &\geq C(\|\cX\|_\op+\|\hat Z\|_\op+|\hat\beta|)^{-2}\eta\geq C(\|\cX\|_\op+1)^{-2}\eta.
\end{align*}
By \cite[Lemma 3.8]{fan2026anisotropic}, $\|\cX\|_\op\prec 1$, and therefore $|q|^{-1}\prec \eta^{-1}$. For $|q-q^{(i)}|$, Lemma \ref{lem:resolvent identities} gives
\begin{align*}
    q-q^{(i)}&=N^{-1}\sum_{k\neq i}Q_{kk}-N^{-1}\sum_{k\neq i}Q_{kk}^{(i)}+N^{-1}Q_{ii}+N^{-1}\hat\beta^{-1}\\
    &=N^{-1}\sum_{k\neq i}\frac{Q_{ki}Q_{ik}}{Q_{ii}}+N^{-1}Q_{ii}+N^{-1}\hat\beta^{-1}\\
    &=N^{-1}\sum_{k\neq i}e^{-2h}Q_{kk}Q_{ii}^{(k)}Q_{kk}^{(i)}(N^{-1}\x_i^*G^{(ik)}\x_k)^2+N^{-1}Q_{ii}+N^{-1}\hat\beta^{-1}\\
    &\prec \eta^{-3}(N^{-1}\|G^{(ik)}\|_F)^2+N^{-1}\eta^{-1}\prec N^{-1}\eta^{-5}.
\end{align*}
Here, we used $\hat\beta^{-1}\leq C\eta^{-1}$. The preceding argument gives $|Q_{kk}^{(i)}|,|Q_{ii}^{(k)}|\leq C\eta^{-1}$ and $N^{-1}\|G^{(ik)}\|_F\leq CN^{-1/2}\eta^{-1}$; we also used Assumption \ref{ass:A2}. This completes the proof.
\end{proof}

\begin{lemma}[Local bounds]\label{lem:local bounds}
Suppose Assumptions \ref{ass:A1} and \ref{ass:A2} hold and $\bD$ is a regular domain. Then, uniformly over all $z\in\bD$, all $h\geq 0$, and all $i\neq j\in[N]$,
\[
    |Q_{ij}|\1_\Xi\prec \Psi_\Theta,\quad |Q_{ii}-q|\1_\Xi\prec\Psi_\Theta.
\]
Moreover, for any fixed $L\ge 1$, uniformly over all $S\subseteq[N]$ with $|S|\leq L$, all $i\in[N]\setminus S$, all $h\geq 0$, and all $z\in\bD$, we have
\[
    |Q_{ii}^{(S)}|\1_\Xi\prec 1,\quad N^{-1}\|G^{(S)}\|_F\1_\Xi\prec\Psi_\Theta. 
\]
Finally, let $F(q)$ and $H(q)$ be defined as in Lemma \ref{lem:exact fixed point}. Then there exists a sufficiently large constant $L'>0$ such that, uniformly over all $z\in\bD$ and $h\geq 0$,
\[
    |s-F(q)|\1_\Xi\prec\Psi_\Theta^2+\max_{0\leq l\leq L'}\left|\frac{1}{N}\sum_{i=1}^N\frac{N^{-1}\Tr(\x_i\x_i^*-\Sigma)G^{(i)}(wI_n+q_0\Sigma)^{-l-1}\Sigma^{l+1}}{\hat\beta+e^{-h}N^{-1}\x_i^*G^{(i)}\x_i}\right|\1_\Xi\prec \Psi_\Theta,
\]
\[
    |s-s_0|\1_\Xi\prec\Theta+|s-F(q)|\1_\Xi,
\]
and
\begin{align*}
    |H(q)|\1_\Xi&\prec \Psi_\Theta^2+\left|\frac{1}{N}\sum_{i=1}^N\frac{N^{-1}\Tr(\x_i\x_i^*-\Sigma)G^{(i)}}{\hat\beta+e^{-h}N^{-1}\x_i^*G^{(i)}\x_i}\right|\1_\Xi\\
    &+\max_{0\leq l\leq L'}\left|\frac{1}{N}\sum_{i=1}^N\frac{N^{-1}\Tr(\x_i\x_i^*-\Sigma)G^{(i)}(wI_n+q_0\Sigma)^{-l-1}\Sigma^{l+1}}{\hat\beta+e^{-h}N^{-1}\x_i^*G^{(i)}\x_i}\right|\1_\Xi\prec \Psi_\Theta.
\end{align*}
\end{lemma}
\begin{proof}
Since $\bD$ is a regular domain, we have $|q_0|\asymp 1$. It follows that, uniformly over all $i\in[N]$,
\[
    |Q_{ii}|\1_\Xi+|Q_{ii}|^{-1}\1_\Xi\prec 1.
\]
Thus, iteratively applying Lemma \ref{lem:resolvent identities} gives, uniformly over all $S\subseteq[N]$ with $|S|\leq L$ and all $i\in[N]\setminus S$
\[
    |Q_{ii}^{(S)}|\1_\Xi+\frac{1}{|Q_{ii}^{(S)}|}\1_\Xi\prec 1.
\]
Using this bound, Lemma \ref{lem:resolvent identities}, and Assumption \ref{ass:A2}, we follow the inductive argument from the proof of \cite[Lemma 4.2]{fan2026anisotropic} to obtain
\[
    |N^{-1}\Tr G-N^{-1}\Tr G^{(S)}|\1_\Xi\prec \frac{N^{-1}\Tr\Im G}{N\eta}\1_\Xi,\quad |s-s^{(S)}|\1_\Xi\prec \frac{\Im s}{N\eta}\1_\Xi.
\]
Now, suppose we have the bound $|s-s_0|\1_\Xi\prec\Theta+\Psi_\Theta$. Then Lemma \ref{lem:resolvent identities} gives
\begin{align*}
    &N^{-2}\|G^{(S)}\|_F^2\1_\Xi\prec\frac{N^{-1}\Tr\Im G^{(S)}}{N\eta}\1_\Xi\prec \frac{N^{-1}\Tr\Im G}{N\eta}\1_\Xi\\
    &\prec \Psi_\Theta^2+\frac{|N^{-1}\Tr G-N^{-1}\Tr(-wI_n-q_0\Sigma)^{-1}|}{N\eta}\1_\Xi\prec \Psi_\Theta^2+\frac{|s-s_0|}{N\eta}\1_\Xi\prec \Psi_\Theta^2.
\end{align*}
Here, we used $|N^{-1}\Tr\Im(-wI_n-q_0\Sigma)^{-1}|\leq \|\Im(wI_n+q_0\Sigma)^{-1}\|_\op\leq C\Im\widetilde m_0$. Similarly, Lemma \ref{lem:resolvent identities} gives
\[
    |Q_{ij}|\1_\Xi\prec |Q_{ii}Q_{jj}^{(i)}||N^{-1}\x_i^*G^{(ij)}\x_j|\1_\Xi\prec N^{-1}\|G^{(ij)}\|_F\1_\Xi\prec\Psi_\Theta.
\]
Moreover,
\begin{align*}
    |Q_{ii}-q|\1_\Xi&\leq \max_{1\leq j\leq N}|Q_{ii}-Q_{jj}|\1_\Xi=\max_{1\leq j\leq N}\frac{|Q_{ii}Q_{jj}||Q_{ii}-Q_{jj}|}{|Q_{ii}Q_{jj}|}\1_\Xi\\
    &\prec \max_{1\leq j\leq N}|Q_{ii}^{-1}-Q_{jj}^{-1}| \1_\Xi\prec\max_{1\leq j\leq N}|N^{-1}\x_i^*G^{(i)}\x_i-N^{-1}\x_j^*G^{(j)}\x_j|.
\end{align*}
By Assumption \ref{ass:A2} and Lemma \ref{lem:resolvent identities},
\begin{align*}
    &\max_{1\leq j\leq N}|N^{-1}\x_i^*G^{(i)}\x_i-N^{-1}\x_j^*G^{(j)}\x_j|\prec \max_{1\leq j\leq N}|N^{-1}\Tr\Sigma G^{(i)}-N^{-1}\Tr\Sigma G^{(j)}|+\Psi_\Theta\\
    &\prec |N^{-2}\x_i^*G^{(i)}\Sigma G^{(i)}\x_i|+\max_{1\leq j\leq N}|N^{-2}\x_j^*G^{(j)}\Sigma G^{(j)}\x_j|+\Psi_\Theta\prec \Psi_\Theta^2+\Psi_\Theta\prec\Psi_\Theta.
\end{align*}
It remains to prove the bounds on $|s-s_0|,|s-F(q)|$, and $|H(q)|$. We begin with $|s-s_0|$:
\begin{align*}
    |s-s_0|\1_\Xi&\leq |s-F(q)|\1_\Xi+|F(q)-s_0|\1_\Xi\\
    &\leq |s-F(q)|\1_\Xi+\Theta|N^{-1}\Tr(-\hat Z-e^{-h}q\Sigma)^{-1}(-wI_n-q_0\Sigma)^{-1}\Sigma^{2}|\1_\Xi\\
    &\leq |s-F(q)|\1_\Xi+\Theta\|(\hat Z+e^{-h}q\Sigma)^{-1}\|_\op\|(wI_n+q_0\Sigma)^{-1}\|_\op\1_\Xi\prec |s-F(q)|\1_\Xi+\Theta.
\end{align*}
Here, we used $\|(wI_n+q_0\Sigma)^{-1}\|_\op\leq C$. Moreover, since $|q-q_0|\1_\Xi\leq N^{-1}\sum|Q_{ii}-q_0|\1_\Xi\leq \Gamma\1_\Xi\leq N^{-\tau/C_0}$, the bound also holds for $\|(\hat Z+e^{-h}q\Sigma)^{-1}\|_\op\1_\Xi$. We next consider $|s-F(q)|\1_\Xi$. By the identity $A^{-1}-B^{-1}=A^{-1}(B-A)B^{-1}$, for any $L'>0$, we have
\begin{align*}
    (\hat Z+e^{-h}q\Sigma)^{-1}&=\sum_{l=0}^{L'}e^{-lh}(q_0-q)^l\underbrace{\Sigma^l(wI_n+q_0\Sigma)^{-l-1}}_{B_l}\\
    &+e^{-h(L'+1)}(q_0-q)^{L'+1}\underbrace{\Sigma^{L'+1}(wI_n+q_0\Sigma)^{-L'-1}(\hat Z+e^{-h}q\Sigma)^{-1}}_{J}.
\end{align*}
Using this identity and Lemma \ref{lem:exact fixed point}, we obtain
\begin{align*}
    |s-F(q)|\1_\Xi&\leq\left|\frac{1}{N}\sum_{i=1}^NQ_{ii}d_i\right|\1_\Xi\prec \max_{0\leq l\leq L'}\left|\frac{1}{N}\sum_{i=1}^NQ_{ii}(N^{-1}\x_i^*G^{(i)}\Sigma B_l\x_i-N^{-1}\Tr\Sigma G\Sigma B_l)\right|\1_\Xi\\
    &+N^{-\tau (L'+1)/C_0}\underbrace{\left|\frac{1}{N}\sum_{i=1}^NQ_{ii}(N^{-1}\x_i^*G^{(i)}\Sigma B_{L'}\x_i-N^{-1}\Tr\Sigma G\Sigma J)\right|}_{\mathbf{I}}\1_\Xi .
\end{align*}
For any $0\leq l\le L'$, Lemma \ref{lem:resolvent identities} gives
\begin{align*}
    &\left|\frac{1}{N}\sum_{i=1}^NQ_{ii}(N^{-1}\x_i^*G^{(i)}\Sigma B_l\x_i-N^{-1}\Tr\Sigma G\Sigma B_l)\right|\1_\Xi\\
    &\leq \left|\frac{1}{N}\sum_{i=1}^NQ_{ii}N^{-1}\Tr(\x_i\x_i^*-\Sigma)G^{(i)}\Sigma B_l\right|\1_\Xi+\left|\frac{1}{N}\sum_{i=1}^NQ_{ii}^2N^{-2}\x_i^*G^{(i)}\Sigma B_l\Sigma G^{(i)}\x_i\right|\1_\Xi\\
    &\prec\max_{1\leq i\leq N}N^{-1}\|\Sigma^{1/2}G^{(i)}\|_F\1_\Xi+\max_{1\leq i\leq N}(N^{-1}\|\Sigma^{1/2}G^{(i)}\|_F)^2\1_\Xi\prec \sqrt{\frac{\Im s}{N\eta}}\1_\Xi+\frac{\Im s}{N\eta}\1_\Xi.
\end{align*}
The remainder $\mathbf{I}$ admits the crude bound
\begin{align*}
    |\mathbf{I}|\1_\Xi\prec \max_{1\leq i\leq N}(\|G^{(i)}\|_\op+\|G\|_\op)\|B_{L'}\|_\op\prec \eta^{-2}\leq N^2.
\end{align*}
Taking $L'$ large enough that $\tau(L'+1)/C_0-2>1$ gives
\[
    N^{-\tau(L'+1)/C_0}|\mathbf{I}|\1_\Xi\prec N^{-\tau(L'+1)/C_0+2}\leq N^{-1}\leq \Psi_\Theta^2.
\]
Combining the preceding estimates yields
\begin{align*}
    |s-s_0|\1_\Xi
    &\prec |s-F(q)|\1_\Xi+\Theta\\
    &\prec \Theta+\Psi_\Theta^2+\frac{\Im s}{N\eta}\1_\Xi
      +\max_{0\leq l\leq L'}\left|\frac{1}{N}\sum_{i=1}^NQ_{ii}N^{-1}\Tr(\x_i\x_i^*-\Sigma)G^{(i)}\Sigma B_l\right|\1_\Xi\\
    &\leq \Theta+\Psi_\Theta^2+\sqrt{\frac{\Im s}{N\eta}}\1_\Xi
      +\frac{\Im s}{N\eta}\1_\Xi\\
    &\leq \Theta+\Psi_\Theta+\sqrt{\frac{|s-s_0|}{N\eta}}\1_\Xi
      +\frac{|s-s_0|}{N\eta}\1_\Xi.
\end{align*}
This is a quadratic inequality in $\sqrt{|s-s_0|}$. Solving it and then squaring gives
\[
    |s-s_0|\1_\Xi\prec \Theta+\Psi_\Theta,
\]
which closes the argument for $N^{-1}\|G^{(S)}\|_F\1_\Xi\prec \Psi_\Theta$. Substituting this bound into the preceding calculation for $|s-F(q)|\1_\Xi$ also gives
\[
    |s-F(q)|\1_\Xi\leq \left|\frac{1}{N}\sum_{i=1}^NQ_{ii}d_i\right|\1_\Xi\prec \Psi_\Theta^2+\max_{0\leq l\leq L'}\left|\frac{1}{N}\sum_{i=1}^NQ_{ii}N^{-1}\Tr(\x_i\x_i^*-\Sigma)G^{(i)}\Sigma B_l\right|\1_\Xi \prec \Psi_\Theta.
\]
Finally, Lemma \ref{lem:exact fixed point} gives
\begin{align*}
    &|H(q)|\1_\Xi\leq \max_{1\leq i\leq N}(N^{-1}\|G^{(i)}\|_F)^2\1_\Xi+\left|\frac{1}{N}\sum_{i=1}^NQ_{ii}d_i\right|\1_\Xi+\left|\frac{1}{N}\sum_{i=1}^NQ_{ii}N^{-1}\Tr(\x_i\x_i^*-\Sigma)G^{(i)}\right|\\
    &\prec \Psi_\Theta^2+\left|\frac{1}{N}\sum_{i=1}^NQ_{ii}N^{-1}\Tr(\x_i\x_i^*-\Sigma)G^{(i)}\right|+\max_{0\leq l\leq L'}\left|\frac{1}{N}\sum_{i=1}^NQ_{ii}N^{-1}\Tr(\x_i\x_i^*-\Sigma)G^{(i)}\Sigma B_l\right|\1_\Xi.
\end{align*}
This completes the proof.
\end{proof}

\subsection{Fluctuation averaging}
We now establish the analogue of \cite[Lemma 3.4]{fan2026anisotropic}. Its proof requires only minor notational changes.

\begin{lemma}[Fluctuation averaging]\label{lem:fluctuation averaging}
Suppose Assumptions \ref{ass:A1} and \ref{ass:A2} hold and $\bD$ is a regular domain. Suppose there exist a constant $\tau'>0$ and a deterministic function $\Phi:\bD\rightarrow[N^{-1/2},N^{-\tau'}]$ such that for any fixed $L\geq 1$, uniformly over all $S\subseteq[N]$ with $|S|\leq L$, all $i\in[N]\setminus S$, all $h\geq 0$, and all $z\in\bD$, we have
\[
    |Q_{ii}^{(S)}|\prec 1,\quad N^{-1}\|G^{(S)}\|_F\prec \Phi.
\]
Then, uniformly over all $h\geq 0$, $z\in\bD$, and deterministic matrices $A\in\C^{n\times n}$ with $\|A\|_\op\leq 1$,
\[
    \left|\sum_{i=1}^N\frac{N^{-1}\Tr(\x_i\x_i^*-\Sigma)G^{(i)}A}{\hat\beta +e^{-h}N^{-1}\x_i^*G^{(i)}\x_i}\right|\prec N\Phi^2.
\]
\end{lemma}
\begin{proof}
The proof is identical to that of \cite[Lemma 3.4]{fan2026anisotropic} after the following notational substitutions:
\[
    \g_i\rightarrow\x_i,\quad R^{(S)}\rightarrow G^{(S)},\quad 1+N^{-1}\g_i^*R^{(iS)}\g_i\rightarrow\hat\beta+e^{-h}N^{-1}\x_i^*G^{(iS)}\x_i.
\]
We briefly recall the argument. For external
indices $i,j,k\in S$, $i\notin\{j,k\}$, define
\begin{align*}
 \cY_i^{(S)}[A]
 &:=
 N^{-1/2}\Tr(\x_i\x_i^*-\Sigma)G^{(S)}A,\quad \cZ_{ijk}^{(S)}[A]
 :=
 \Tr(\x_i\x_i^*-\Sigma)
 \frac{G^{(S)}}N \x_j\x_k^*
 \frac{G^{(S)}A}{\sqrt N},\\
 \cB_{jk}^{(S)}
 &:=
 N^{-1}\x_j^*G^{(S)}\x_k,j\ne k,\quad \cP_i^{(S)}:=
 N^{-1}\Tr(\x_i\x_i^*-\Sigma)G^{(S)}\\
 \cQ_i^{(S)}&:=(\hat\beta+e^{-h}N^{-1}\x_i^*G^{(S)}\x_i)^{-1},\quad \cC^{(S)}:=(\hat\beta+e^{-h}N^{-1}\Tr\Sigma G^{(S)})^{-1}.
\end{align*}
Conditioning on the minor and applying quadratic- and mixed-form
concentration (Assumption \ref{ass:A2}) gives the exact analogue of FMPW Lemma 3.9:
\begin{equation}\label{eq:fmpw-aux-bounds}
 |\cY_i^{(S)}[A]|\prec N^{1/2}\Phi,\quad
 |\cZ_{ijk}^{(S)}[A]|\prec N^{1/2}\Phi^2,\quad
 |\cB_{jk}^{(S)}|+|\cP_i^{(S)}|\prec\Phi .
\end{equation}
For instance,
\[
 |\cY_i^{(S)}[A]|
 \prec N^{-1/2}\|G^{(S)}A\|_{F}
 \prec N^{1/2}\Phi.
\]

If $l\notin S$, then Lemma \ref{lem:resolvent identities} (the Sherman--Morrison recursion) gives the exact identities
\begin{align}
 \cY_i^{(S)}[A]
 &=
 \cY_i^{(lS)}[A]
 -e^{-h}\cZ_{ill}^{(lS)}[A]\cQ_l^{(lS)},\label{eq:Y-recursion}\\
 \cB_{ij}^{(S)}
 &=
 \cB_{ij}^{(lS)}
 -e^{-h}\cB_{il}^{(lS)}\cB_{lj}^{(lS)}\cQ_l^{(lS)},\label{eq:B-recursion}\\
 (\cQ_i^{(S)})^{-1}
 &=
 (\cQ_i^{(lS)})^{-1}
 -e^{-2h}\cB_{il}^{(lS)}\cB_{li}^{(lS)}\cQ_l^{(lS)},\label{eq:Q-recursion}\\
 (\cQ_i^{(S)})^{-1}
 &=(\cC^{(S)})^{-1}+e^{-h}\cP_i^{(S)}.\label{eq:QC-recursion}
\end{align}
The corresponding recursion for $\cZ_{ijk}^{(S)}$ is obtained by applying
Lemma \ref{lem:resolvent identities} to its two $G^{(S)}$-factors; it has the four terms
\[
 \cZ_{ijk}^{(lS)},\quad
 -e^{-h}\cZ_{ijl}^{(lS)}\cB_{kl}^{(lS)}\cQ_l^{(lS)},\quad
 -e^{-h}\cZ_{ilk}^{(lS)}\cB_{lj}^{(lS)}\cQ_l^{(lS)},\quad
 e^{-2h}\cZ_{ill}^{(lS)}\cB_{kl}^{(lS)}\cB_{lj}^{(lS)}[\cQ_l^{(lS)}]^2,
\]
The four displayed terms coincide with those in \cite[eq (3.6)--(3.12)]{fan2026anisotropic}. In particular, the
geometric expansions of $\cQ_i^{(S)}$ about $\cQ_i^{(lS)}$ and
$\cC^{(S)}$, with an arbitrarily high fixed truncation order, have the
same remainders as in \cite{fan2026anisotropic}.

For completeness, set
\[
 Y_A
 =
 \frac1{\sqrt N}\sum_{i=1}^N
 \cY_i^{(i)}[A]\cQ_i^{(i)}.
\]
In a $2p$-th moment expansion, let $S$ be the set of distinct column
indices. Repeated use of
\eqref{eq:Y-recursion}--\eqref{eq:QC-recursion} expands every factor into
minor-$S$ variables. More precisely, the single-factor expansion of \cite[Lemma 3.11]{fan2026anisotropic} applies verbatim:
each term produced from one original $\cY\cQ$-factor retains exactly one $\cY$- or $\cZ$-factor. Thus, the full $2p$-fold monomial contains the corresponding $2p$ distinguished factors. Properties (b) and (c) of \cite[Lemma 3.11]{fan2026anisotropic} apply with the substitutions above: the number of $\cB$-factors incident to the original column index is even, and the number of other distinct lower indices is bounded by the stated $\cZ$- and $\cB$-factor count. These
properties are unchanged because the extra coefficients in the present recursions are powers of $e^{-h}\in[0,1]$. Apply the singleton counting argument in the proof of
\cite[Lemma 3.4]{fan2026anisotropic}. The bounds
\eqref{eq:fmpw-aux-bounds} and $|S|\le p+\frac12|\{\text{singleton indices}\}|$ give
\[
    \E|Y_A|^{2p}\prec(N\Phi^2)^{2p}.
\]
Here, $\Phi\ge N^{-1/2}$ is used in the last power count. Markov's
inequality yields $Y_A\prec N\Phi^2$, proving the lemma.
\end{proof}

We are now ready to prove Proposition \ref{prop:averaged zig}.
\begin{proof}[Proof of Proposition \ref{prop:averaged zig}]
Uniformly over all $t\geq 0$, $X_t$ satisfies Assumptions \ref{ass:A1} and \ref{ass:A2}. We therefore work below with a general $X$ satisfying these two assumptions. Fix any $z\in\bD$ with $\eta=1$. Using Lemma \ref{lem:global-bounds}, Lemma \ref{lem:exact fixed point}, Assumption \ref{ass:A1}, and Assumption \ref{ass:A2}, we obtain
\begin{align*}
    |H(q)|&\prec \max_{1\leq i\leq N}|N^{-2}\x_i^* G^{(i)}\Sigma G^{(i)}\x_i|+\max_{1\leq i\leq N}|N^{-1}\Tr(\x_i\x_i^*-\Sigma)G^{(i)}|\\
    &+\max_{1\leq i\leq N}|N^{-1}\x_i^*G^{(i)}(\hat Z+e^{-h}q\Sigma)^{-1}\Sigma\x_i-N^{-1}\Tr\Sigma G(\hat Z+e^{-h}q\Sigma)^{-1}\Sigma|\\
    &\prec N^{-1}+N^{-1/2}+\max_{1\leq i\leq N}|N^{-1}\x_i^*G^{(i)}(\hat Z+e^{-h}q\Sigma)^{-1}\Sigma\x_i-N^{-1}\Tr\Sigma G(\hat Z+e^{-h}q\Sigma)^{-1}\Sigma|.
\end{align*}
We further expand the last term as follows:
\begin{align*}
    &|N^{-1}\x_i^*G^{(i)}(\hat Z+e^{-h}q\Sigma)^{-1}\Sigma\x_i-N^{-1}\Tr\Sigma G(\hat Z+e^{-h}q\Sigma)^{-1}\Sigma|\\
    &\leq |N^{-1}\x_i^*G^{(i)}(\hat Z+e^{-h}q^{(i)}\Sigma)^{-1}\Sigma\x_i-N^{-1}\Tr\Sigma G(\hat Z+e^{-h}q^{(i)}\Sigma)^{-1}\Sigma|\\
    &+|q-q^{(i)}||N^{-1}\x_i^*G^{(i)}(\hat Z+e^{-h}q\Sigma)^{-1}\Sigma(\hat Z+e^{-h}q^{(i)}\Sigma)^{-1}\Sigma\x_i|\\
    &+|q-q^{(i)}||N^{-1}\Tr\Sigma G(\hat Z+e^{-h}q\Sigma)^{-1}\Sigma(\hat Z+e^{-h}q^{(i)}\Sigma)^{-1}\Sigma|\\
    &\prec |N^{-1}\x_i^*G^{(i)}(\hat Z+e^{-h}q^{(i)}\Sigma)^{-1}\Sigma\x_i-N^{-1}\Tr\Sigma G(\hat Z+e^{-h}q^{(i)}\Sigma)^{-1}\Sigma|+N^{-1},
\end{align*}
Here, we used $|q-q^{(i)}|\prec N^{-1}$ and $|N^{-1}\x_i^*G^{(i)}(\hat Z+e^{-h}q\Sigma)^{-1}\Sigma(\hat Z+e^{-h}q^{(i)}\Sigma)^{-1}\Sigma\x_i|\prec \|G^{(i)}(\hat Z+e^{-h}q\Sigma)^{-1}\Sigma(\hat Z+e^{-h}q^{(i)}\Sigma)^{-1}\Sigma\|_\op\prec 1$, together with the analogous bound for $|N^{-1}\Tr\Sigma G(\hat Z+e^{-h}q\Sigma)^{-1}\Sigma(\hat Z+e^{-h}q^{(i)}\Sigma)^{-1}\Sigma|$. Finally,
\begin{align*}
    &|N^{-1}\x_i^*G^{(i)}(\hat Z+e^{-h}q^{(i)}\Sigma)^{-1}\Sigma\x_i-N^{-1}\Tr\Sigma G(\hat Z+e^{-h}q^{(i)}\Sigma)^{-1}\Sigma|\\
    &\prec |N^{-1}\Tr(\x_i\x_i^*-\Sigma)G^{(i)}(\hat Z+e^{-h}q^{(i)}\Sigma)^{-1}\Sigma|+|N^{-2}\x_i^*G^{(i)}(\hat Z+e^{-h}q^{(i)}\Sigma)^{-1}\Sigma^2G^{(i)}\x_i|\\
    &\prec N^{-1/2}+N^{-1}.
\end{align*}
Combining these estimates gives $|H(q)|\prec N^{-1/2}$.

We now follow the two-step weak-law argument of \cite[Lemmas~4.6 and~4.7]{fan2026anisotropic}. Fix $h\geq0$ and set $r=e^{-h}$; then the function $H(q)$ in Lemma \ref{lem:exact fixed point} is precisely $H_{r,z}(q)$ from Theorem \ref{thm:strong-stability}. At height one, the residual estimate above, Theorem \ref{thm:strong-stability}, Lemma \ref{lem:global-bounds}, Assumption \ref{ass:A2}, and Lemma \ref{lem:resolvent identities} give $\Gamma\prec N^{-1/4}$, exactly as in the proof of \cite[Lemma~4.6]{fan2026anisotropic}. For a general $z\in\bD$, enumerate the vertical lattice $L(z)$ downward with mesh $N^{-5}$ and repeat the stochastic-continuity induction from \cite[Lemma~4.7]{fan2026anisotropic}. The resolvent identity supplies the required Lipschitz control, Lemma \ref{lem:local bounds} supplies the residual estimate and propagates the a priori event $\Xi$, and Theorem \ref{thm:strong-stability} selects the correct root at each lattice point. The same union-bound argument therefore yields
\[
    \Gamma\prec (N\eta)^{-1/4}
\]
uniformly over $z\in\bD$.

With this weak law in hand, suppose that $\Theta\prec (N\eta)^{-c}$ for some $0<c<1$, and set
\[
    \Psi_c:=\sqrt{\frac{\Im\widetilde m_0+(N\eta)^{-c}}{N\eta}}.
\]
Then $\Psi_\Theta\prec\Psi_c$. Applying Lemma \ref{lem:fluctuation averaging} and Theorem \ref{thm:strong-stability} with $\Phi=\Psi_c$ gives
\begin{align*}
    |q-q_0|&\prec \frac{(\Im\widetilde m_0+(N\eta)^{-c})/(N\eta)}{\sqrt{(\Im\widetilde m_0+(N\eta)^{-c})/(N\eta)}+\sqrt{\kappa+\eta}}\leq \frac{\Im\widetilde m_0}{N\eta\sqrt{\kappa+\eta}}+(N\eta)^{-1/2-c/2}\\
    &\prec (N\eta)^{-1}+(N\eta)^{-1/2-c/2}.
\end{align*}
Here, we used the facts that, on a regular domain, $\eta\leq 1$ and $\Im\widetilde m_0\leq \max\{\sqrt{\kappa+\eta},\eta/\sqrt{\kappa+\eta}\}\leq \sqrt{\kappa+\eta}$. Thus, we obtain the implication
\[
    \Theta\prec (N\eta)^{-c}\implies \Theta\prec \max\{(N\eta)^{-1},(N\eta)^{-1/2-c/2}\}.
\]
For any $\eps>0$, starting with $c=1/4$ and iterating this implication a fixed number $C_\eps$ of times gives $\Theta\prec (N\eta)^{-1+\eps}$. Since $\eps$ was arbitrary, we conclude that $\Theta\prec (N\eta)^{-1}$. In particular, $\Psi_\Theta\prec\Psi$. Applying Lemma \ref{lem:local bounds} and Lemma \ref{lem:fluctuation averaging} with $\Phi=\Psi$ then gives
\[
    |s-s_0|\prec \Theta+\Psi^2\prec (N\eta)^{-1}.
\]
Returning to the normalized block traces, the definitions of $s-s_0$ and $q-q_0$ give, uniformly over all $t,h\geq 0$,
\[
    |N^{-1}\Tr\Sigma[\cR_{h,t}-M_h]_{11}|+|N^{-1}\Tr[\cR_{h,t}-M_h]_{22}|=e^{-h/2}(|s-s_0|+|q-q_0|)\prec e^{-h/2}(N\eta)^{-1}.
\]
To make the estimate uniform in $h$ and $t$, set
$T:=5\log N$ and $\delta_N:=N^{-50}$, and take uniform grids
$h_k=k\delta_N\leq T$ and $t_j=j\delta_N\leq T$.
Write $I_j=[t_j,\min\{t_j+\delta_N,T\}]$.
The fixed-time norm bound and a union bound give
$\max_jN^{-1/2}\|X_{t_j}\|_\op\prec1$.
For $t\in I_j$, the OU equation yields
\[
    X_t-X_{t_j}
    =(e^{-(t-t_j)/2}-1)X_{t_j}
    +\Sigma^{1/2}\int_{t_j}^{t}e^{-(t-u)/2}\,\de B_u.
\]
Brownian maximal estimates applied entrywise, followed by the
Frobenius-norm bound and a union bound, therefore imply
\[
    \max_j\sup_{t\in I_j}
    \|\cX_t-\cX_{t_j}\|_\op
    \prec \delta_N+\sqrt{N\delta_N}.
\]
The explicit source formulas also give
\[
    \|\cD_{h_1}-\cD_{h_2}\|_\op
    +\|M_{h_1}-M_{h_2}\|_\op
    \leq CN^{5/2}|h_1-h_2|,
    \qquad h_1,h_2\in[0,T].
\]
Consequently, using the resolvent identity and
$\|\cR_{h,t}\|_\op\leq C\eta^{-1}\leq CN$, we obtain
\begin{align*}
    &\|\cR_{h,t}-\cR_{h_k,t_j}\|_\op+\|M_h-M_{h_k}\|_\op\\
    &\qquad\prec
    N^2\bigl(\sqrt{N\delta_N}+N^{5/2}\delta_N\bigr)
    \prec N^{-45/2},
\end{align*}
simultaneously over all cells with
$0\leq h-h_k<\delta_N$ and $t\in I_j$.
Since $n/N$ and $\|\Sigma\|_\op$ are bounded, the same
interpolation error bound holds for the normalized block traces.

A union bound makes the previously established estimate
simultaneous on this polynomial-size grid.
Moreover, neighboring weights $e^{-h/2}$ differ by a factor
$1+O(\delta_N)$, and
\[
    e^{-h/2}(N\eta)^{-1}\geq N^{-7/2},
    \qquad h\in[0,T],\quad \eta\leq1.
\]
The interpolation error is therefore negligible, proving
uniformity over $(h,t)\in[0,T]^2$, with stochastic-domination
constants uniform in $z\in\bD$.
Restricting to $(h-t,t)$ for $0\leq t\leq h-h'$ completes
the proof.
\end{proof}

\section*{Acknowledgments}

T.M.~thanks Hugo Latourelle-Vigeant and Basil Saeed for their
early work exploring the zig-zag strategy, and Elliot Paquette
for helpful discussions. R.M.~thanks Zhou Fan, Elliot Paquette,
and Zhichao Wang for their collaboration on
\cite{fan2026anisotropic}, which motivated the present work,
and Haoyu Wang for helpful discussions on stochastic calculus.
AI tools assisted with language editing, proofreading, figure
preparation, and mathematical checks. The authors take full
responsibility for the content of the paper.

\bibliographystyle{amsalpha}
\bibliography{biblio}

\end{document}